\documentclass[a4paper,twoside,12pt,finnish,british]{amsart}

\usepackage{etex}
\reserveinserts{28}										

\usepackage{babel} 									%Support for many languages
\usepackage[T1]{fontenc}							%The last two are related to making "ä" etc. more usable in Latex
\usepackage[utf8]{inputenx}
				
\usepackage{mathtools}								%Includes amsmath, basic math commands
\usepackage{empheq}									%Useful for systems of equations
\usepackage{enumitem}								%Improves enumerate environment among other things
\usepackage{thm-restate, thm-autoref}				%Theorem restate that works with Cref.

\usepackage{amssymb}								%Fonts and mathematical symbols
\usepackage[sc]{mathpazo}							%Fonts		
\usepackage{mathrsfs}								%\mathscr command; http://www.math.washington.edu/~lee/Courses/typesetting-script.pdf

\usepackage{iflang}									%Used in theoremenv.sty
\usepackage{xifthen}								%See http://tex.stackexchange.com/questions/58628/optional-argument-for-newcommand
\usepackage[final]{pdfpages}						%Enables including multiple pages of a pdf easily.
\usepackage{todonotes}								%\todo{note} \usepackage[disable]{todonotes}
\usepackage{aliascnt}								%Useful for making a counter share the number with other counter, but with different names.. used in thereomenv.sty in amsthm 																								 variants of the code
\usepackage{needspace}								%Needed in theoremenv, checks that proofs etc. don't get separated too easily: http://tex.stackexchange.com/a/30125
\usepackage{refcount}

\usepackage{subcaption}			%Naming figures (subfigure)
\usepackage{placeins} 			%Restricting movement of floats (FloatBarrier)
\usepackage{grffile}			%More possibilities with naming figures
\usepackage[all]{xy}			%Useful in Algebraic Topology

\usepackage{transparent}
\usepackage{color}				%Colorful pictures

\usepackage{graphicx}			%More options for includegraphic. See http://ctan.org/pkg/graphicx for example.

\usepackage[babel]{csquotes}							%More options for \cite etc.					

\usepackage{nameref}									%I.e. can refer to Theorem (Euler's theorem) as "Euler's theorem" with a hyperlink
\usepackage[colorlinks	=	true,						%Hyperlinks to references (an absolute must have)
          	 	linkcolor	=	red,
            	urlcolor	=	red,
            	citecolor	=	red]%
            	{hyperref}
\usepackage{bookmark}									%Hyperref options; used in theoremenv to bookmark theorems in pdf files automatically
\usepackage{cleveref}									%Further options for hyperref, \cref and \Cref are interesting (and their multiple variants)

\newcommand{\id}{ \operatorname{id} }											%Identity mapping
\newcommand{\diam}{ \operatorname{diam} }										%Diameter of a set
\newcommand{\cont}[2][]{%														%Spaces of continuous or differentiable functions
	\ifthenelse{\equal{#1}{}}%
		{\mathcal{C}\left(#2\right)}%
		{\mathcal{C}^{#1}\left(#2\right)}%
	}

\newcommand{\inter}[2][]{%														%Interior of a set
	\ifthenelse{\equal{#1}{}}%
		{#2^{\circ}}%
		{#1^{\circ}}%
	}

\newcommand{\bndr}[2][]{%														%Boundary of a set
	\ifthenelse{\equal{#1}{}}%
		{\partial #2}%
		{\partial \left( #1 \right)}%
	}
				
\newcommand{\ext}[2][]{%														%Exterior of a set
	\ifthenelse{\equal{#1}{}}%
		{\operatorname{ext} #2}%
		{\operatorname{ext} \left( #1 \right)}%
	}
				
\newcommand{\modulus}{\operatorname{mod}}										%Modulus
	
\newcommand{\norm}[1]{ \left\Vert #1 \right\Vert }								%Norm without bolding (Mathematics)
\newcommand{\abs}[1]{ \left\lvert {#1} \right\rvert }							%Absolute value
	
\newcommand{\Exp}[1]{ \operatorname{exp}\left( #1 \right) }	
\newcommand\restr[2]{															%Restriction function with modified spacing; http://tex.stackexchange.com/a/22255
					{ 									
  					\left.\kern-\nulldelimiterspace 	
 					#1 									
  					\vphantom{\big|} 					
  					\right|_{#2} 						
  					}}

\newcommand{\aint}[2][]{%
	\ifthenelse{\equal{#1}{}}%
					{%
\mathchoice%
      {\mathop{\kern 0.2em\vrule width 0.6em height 0.69678ex depth -0.58065ex
              \kern -0.8em \intop}\nolimits_{\kern -0.45em#2}^{#1}}%
      {\mathop{\kern 0.1em\vrule width 0.5em height 0.69678ex depth -0.60387ex
              \kern -0.6em \intop}\nolimits_{#2}^{#1}}%
      {\mathop{\kern 0.1em\vrule width 0.5em height 0.69678ex depth -0.60387ex
              \kern -0.6em \intop}\nolimits_{#2}^{#1}}%
      {\mathop{\kern 0.1em\vrule width 0.5em height 0.69678ex depth -0.60387ex
              \kern -0.6em \intop}\nolimits_{#2}^{#1}}}%
					{%
\mathchoice%
      {\mathop{\kern 0.2em\vrule width 0.6em height 0.69678ex depth -0.58065ex
              \kern -0.8em \intop}\nolimits_{\kern -0.45em#1}^{#2}}%
      {\mathop{\kern 0.1em\vrule width 0.5em height 0.69678ex depth -0.60387ex
              \kern -0.6em \intop}\nolimits_{#1}^{#2}}%
      {\mathop{\kern 0.1em\vrule width 0.5em height 0.69678ex depth -0.60387ex
              \kern -0.6em \intop}\nolimits_{#1}^{#2}}%
      {\mathop{\kern 0.1em\vrule width 0.5em height 0.69678ex depth -0.60387ex
              \kern -0.6em \intop}\nolimits_{#1}^{#2}}}}
              
\DeclareMathOperator*{\esssup}{ess\,sup}

\DeclareMathOperator*{\aplim}{aplim}												%Approximate Limit
\newcommand{\md}[2][]{																%Metric Derivative with options
	\ifthenelse{\equal{#1}{}}%
					{ \operatorname{md}(#2)	}%
					{ \operatorname{md}(#1)(#2) 	}}
\newcommand{\apmd}[2][]{															%Approximate Metric Derivative with options
	\ifthenelse{\equal{#1}{}}%
					{ \operatorname{F}_{#2}	}%
					{ \operatorname{F}_{#1}(#2) 	}}
\newcommand{\D}{\,\mathrm{d}}
\usepackage{amsthm}
\newtheorem{theorem}{Theorem}[section]
\newtheorem{proposition}[theorem]{Proposition}
\newtheorem{corollary}[theorem]{Corollary}
\newtheorem{lemma}[theorem]{Lemma}
\newtheorem{definition}[theorem]{Definition}

\newtheorem{remark}[theorem]{Remark}

\newtheorem{openproblem}{Open Problem}

\numberwithin{equation}{section}						%equation numbering
\bookmarksetup{open,numbered}							%bookmark numbering: see documentation for bookmark.sty (default in teoreemaymp -> thbookmark command
\begin{document}
\selectlanguage{\british}

\title[Uniformization Of Metric Surfaces]{Uniformization Of Metric Surfaces Using Isothermal Coordinates}

\author{Toni Ikonen}

\address{University of Jyvaskyla \\ Department of Mathematics and Statistics \\
P.O. Box 35 (MaD) \\
FI-40014 University of Jyvaskyla}
\email{toni.m.h.ikonen@jyu.fi}

\thanks{The author was partially supported by the Academy of Finland, project number 308659}


%\date{\today}

\begin{abstract}
We establish a uniformization result for metric surfaces --- metric spaces that are topological surfaces with locally finite Hausdorff $2$-measure.

Using the geometric definition of quasiconformality, we show that a metric surface that can be covered by quasiconformal images of Euclidean domains is quasiconformally equivalent to a Riemannian surface. To prove this, we construct suitable isothermal coordinates.
\end{abstract}

\maketitle\thispagestyle{empty}

\begingroup
\hypersetup{hidelinks}
\tableofcontents
\endgroup

\section{Introduction}\label{sec:intro}
The Riemann mapping theorem states that given a simply connected proper subdomain $U$ of $\mathbb{R}^{2}$, there exists a conformal map $\phi \colon \mathbb{D} \rightarrow U$, where $\mathbb{D}$ is the Euclidean disk. Recall that conformal maps preserve angles but they do not necessarily preserve lengths of paths or areas. We say that domains $U$ and $V$ are \emph{conformally equivalent} if there exists a conformal map from $U$ to $V$.

When the topology of $U$ is more complicated, so is the classification result. For example, if $U = A( 1, \, r ) \subset \mathbb{R}^{2}$ in an Euclidean annulus of inner radius $1$ and outer radius $r > 1$, two such annuli $A( 1, \, r )$ and $A( 1, \, r' )$ are conformally equivalent if and only if $r = r'$.

If we relax the definition of conformal map to allow for distortion of infinitesimal balls in a uniformly controlled manner, we obtain the class of quasiconformal maps. With this relaxation, it turns out that for every pair of outer radii $1 < r$ and $1 < r'$, there exists a quasiconformal map from $A( 1, \, r )$ onto $A( 1, \, r' )$.

Such a map takes the infinitesimal Euclidean balls in $A( 1, \, r )$ to infinitesimal ellipses in $A( 1, \, r' )$, and the distortion is determined from the eccentricity of the ellipses. For a fixed $r > 1$, the distortion of every quasiconformal map $\eta \colon A( 1, \, r ) \rightarrow A( 1, \, r' )$ has a lower bound $C_{r'}$ that blows up as $r' \rightarrow \infty$.

Similar questions can be considered when the topology type of the surface is more complicated. This is the domain of Teichmüller theory of surfaces; see for example \cite{lehto,imayoshi,hubbard}. Roughly speaking, the Teichmüller theory classifies Riemann surfaces up to conformal maps, and quasiconformal maps measure how far apart two Riemann surfaces are from one another. We are interested in metric surfaces therefore it is more natural to consider Riemannian surfaces instead of Riemann surfaces.

Quasiconformal maps also arise when we try to find isothermal coordinates in a given Riemannian surface. Indeed, given a Riemannian surface $( Y, \, g )$ and a smooth chart $f \colon V \rightarrow U \subset \mathbb{R}^{2}$, by considering a smaller open set $V' \subset V$, we may assume without loss of generality that $f$ is quasiconformal. We interpret the Riemannian metric $g$ on $V$ as a particular choice of an ellipse at each point of $V$. Then the chart $f$ maps these ellipses to ellipses in $U$. We ask whether it is possible to find a diffeomorphism $\eta \colon U \rightarrow W \subset \mathbb{R}^{2}$ such that the particular ellipses in $V$ are mapped to Euclidean balls by $\eta \circ f$.

The existence of such a diffeomorphism $\eta$ is guaranteed by the measurable Riemann mapping theorem; see for example \cite{ahlfors-bers,astala}. When we apply this theorem to the ellipse field of $f$, the composition $\eta \circ f$ maps the ellipses in $V$ to Euclidean balls. Classically, the coordinates $\eta \circ f$ are called \emph{isothermal coordinates}.

We are interested in two questions. Given a metric space $Y$ homeomorphic to a surface, what conditions guarantee that there exists a Riemannian surface $Z$ and a quasiconformal map $f \colon Y \rightarrow Z$? Moreover, is it possible to find a good notion of isothermal coordinates on $Y$?

One approach to these questions is to use \emph{quasisymmetric} parametrizations. Quasisymmetric maps are homeomorphisms that distort all shapes in a controlled manner, not just infinitesimal ones.

We recall some known results. Suppose that $Y$ is homeomorphic to $\mathbb{S}^{2}$ and $Y$ is Ahlfors $2$-regular (see \Cref{sec:recip:qs}). Then the Bonk--Kleiner theorem \cite{bonk-kleinerthm} shows that the surface $Y$ is quasisymmetrically equivalent to $\mathbb{S}^{2}$ if and only if $Y$ is linearly locally contractible (see \Cref{sec:recip:qs}). This approach has applications to group theory; see \cite{bonk-kleinerthm,cannon:bonk-kleiner,uniformization} and references therein. Similar results hold when $Y$ is compact and orientable \cite{compactqs} or if $Y$ is homeomorphic to a planar domain \cite{qs:koebe}. If the surface $Y$ is locally Ahlfors $2$-regular and locally linearly locally contractible (as defined in \cite{locallyqs}), the surface $Y$ has quasisymmetric coordinates \cite{locallyqs}. Such coordinates are a possible metric analog of smooth coordinates on $Y$.

We use a different approach motivated by \cite{uniformization}. Theorem 1.4 of \cite{uniformization} states the following. Suppose that $Y$ is homeomorphic to $\mathbb{R}^{2}$ and has locally finite Hausdorff $2$-measure. Then there exists a quasiconformal map $f \colon Y \rightarrow U$ into a planar domain $U \subset \mathbb{R}^{2}$ if and only if $Y$ is a \emph{reciprocal disk}. We use the \emph{geometric} definition of quasiconformality (see \eqref{eq:outer:dilatation}). Geometrically quasiconformal maps distort families of paths in a controlled manner; to quantify this we use the \emph{(path) modulus} (see \eqref{eq:modulus}).

The space $Y$ is a reciprocal disk if it satisfies the following two conditions. First, we require that there exists a constant $\kappa > 0$ such that for any set $Q \subset Y$ homeomorphic to $\left[ 0, \, 1 \right]^{2}$, the moduli of paths $M_{\rightarrow}$ and $M_{\uparrow}$ joining the opposite sides of $\partial Q$ in $Q$ control each other in the following quantitative way: for such a $Q$,
\begin{equation*}
	\kappa^{-1}
	\leq
	M_{\rightarrow} \cdot M_{\uparrow}
	\leq
	\kappa.
\end{equation*}
This condition is defined precisely in \eqref{upper:bound} and \eqref{lower:bound}. The second condition requires that points have zero modulus in a suitable sense (see \eqref{point:zero:modulus}). As an example, if the Hausdorff $2$-measure of the metric balls of $Y$ of radius $r$ are bounded from above by a constant multiple of $r^{2}$, then the space $Y$ is reciprocal \cite[Theorem 1.6]{uniformization}.

One of the reasons why Theorem 1.4 of \cite{uniformization} requires that the Hausdorff $2$-measure is locally finite is that the distance $d$ of $Y$ uniquely determines this measure. Moreover, in order to guarantee that quasiconformal maps have good properties, the spaces in question need to have a lot of rectifiable paths with respect to modulus; the local finiteness of the Hausdorff $2$-measure combined with coarea estimates and planar topology guarantee a wealth of rectifiable paths in $Y$. Usually, the existence of large families of rectifiable paths is guaranteed by strong assumptions on the geometry of the space via Poincaré inequalities \cite{PIspaces,1poincareandmodulus,currentspoincare}, synthetic curvature lower bounds \cite{sturm1,sturm2,lott-villani}, or strong control on the measure of balls (Ahlfors regularity) and strong connectivity properties of the space via modulus estimates (Loewner property) \cite{controlledgeometry,locallyboundedgeometry}.

We are interested in generalizing the notion of reciprocal disks to surfaces $Y$ with non-trivial fundamental group. Suppose that $Y$ is a metric space with locally finite Hausdorff $2$-measure homeomorphic to some surface. In this paper, such a $Y$ is called a \emph{metric surface}, and we say that $Y$ is \emph{locally reciprocal} if it can be covered by reciprocal disks. We know from Theorem 1.4 of \cite{uniformization} that local reciprocality of $Y$ is equivalent to requiring that every point of $Y$ has a neighbourhood that is quasiconformally homeomorphic to a domain in $\mathbb{R}^{2}$.

Suppose that $V \subset Y$ is a reciprocal disk and $f \colon V \rightarrow U \subset \mathbb{R}^{2}$ is a quasiconformal map. We ask whether it is possible to find a quasiconformal map $\eta \colon U \rightarrow W \subset \mathbb{R}^{2}$ such that $\eta \circ f$ is isothermal in some meaningful way; compare this to the Riemannian setting outlined above.

It turns out that we can associate to $U$ and $f$ a measurable field of convex bodies whose shapes govern the distortion of $f$ (and the distortion of $f$ controls the shapes of the bodies). Then we obtain a quasiconformal map $\eta$ from the measurable Riemann mapping theorem by associating to the field of convex bodies a canonical choice of a field of ellipses and requiring that $\eta$ maps the ellipses to Euclidean balls. This idea was used in \cite{uniformization} to minimize the distortion of $f$, where the particular choice of an ellipse was the \emph{John ellipse}. The key idea of \cite{upper:modulus:bound} was to use the \emph{distance ellipse} --- a canonical choice of an ellipse related to the Banach--Mazur distance (see for example \cite[Chapter 37]{enoughsymmetries}, \cite{upper:modulus:bound}, or \Cref{sec:uniformization}). We use the latter approach and the resulting map $\eta \circ f$ will be our notion of \emph{isothermal coordinates} (see \Cref{sec:uniformization}). In general, the map $\eta \circ f$ is not conformal but its distortion is still well-controlled, and what we prove is that the atlas of isothermal coordinates yield a natural choice of ellipses (conformal structure) on $Y$ and thus a natural choice of a Riemannian distance $d_{G}$ on $Y$. Analogously to the classical uniformization result, the distance $d_{G}$ is constructed in such a way that $( Y, \, d_{G} )$ is complete and has constant curvature $-1$, $0$ or $1$ (see \Cref{sec:uniformization}). We prove that the homeomorphism $f = \id \colon ( Y, \, d ) \rightarrow ( Y, \, d_{G} )$ is quasiconformal (see \Cref{sec:uniformization:main:theorem}). We also prove that $f$ has the smallest distortion among all choices of Riemannian surface $( Z, \, d_{G'} )$ and quasiconformal map $h \colon ( Y, \, d ) \rightarrow ( Z, \, d_{G'} )$ (see \Cref{sec:isothermal:Riemannian}).

To get a better understanding of the above, we consider the Euclidean annulus $Y = A( 1, \, r )$ with an outer radius $r > 1$ and distance $d$ induced by different $\norm{ \cdot }_{p}$-norms for $\infty \geq p \geq 1$. First, consider the case $p = 2$. In this case, the Riemannian surface $( Y, \, d_{G} ) = ( A(1, \, r), \, d_{G} )$ is complete and has constant curvature $-1$. We emphasize that since $( Y, \, d )$ is not complete and it has constant curvature $0$, the distance $d_{G}$ is not the Euclidean distance $d = \norm{ \cdot }_{2}$. However, the map
\begin{equation*}
	f = \id \colon ( A( 1, \, r ), \, \norm{ \cdot }_{2} ) \rightarrow ( A(1, \, r), \, d_{G} )
\end{equation*}
is conformal. The existence of such a distance $d_{G}$ is the Riemannian version of the classical uniformization result.

Fix some $\infty \geq p \geq 1$ with $p \neq 2$ and consider the metric surface $( Y, \, d ) = ( A(1, \, r), \, \norm{ \cdot }_{p} )$. It turns out that, even if $p \neq 2$, the Riemannian surface
\begin{equation*}
	( Y, \, d_{G} ) = ( A(1, \, r), \, d_{G} )
\end{equation*}
obtained from our methods is the same surface as in the case $p = 2$. However, now the map
\begin{equation*}
	f_{p} = \id \colon ( A( 1, \, r ), \, \norm{ \cdot }_{p} ) \rightarrow ( A(1, \, r), \, d_{G} )
\end{equation*}
is not conformal; the distortion of $f_{p}$ equals a constant $D_{p} > 1$. In fact, the results of \Cref{sec:isothermal:Riemannian} imply that the distortion $D_{p}$ is the \emph{Banach--Mazur distance} between the Banach spaces $( \mathbb{R}^{2}, \, \norm{ \cdot }_{p} )$ and $( \mathbb{R}^{2}, \, \norm{ \cdot }_{2} )$ (see \cite[Chapter 37]{enoughsymmetries}). The distortion is bounded from above by $\sqrt{2}$ and equals $\sqrt{2}$ if $p = 1$ or $p = \infty$. As $p \rightarrow 2$, the distortion $D_{p}$ converges to one.

Suppose that $( Z, \, d_{G'} )$ is a Riemannian surface quasiconformally equivalent to $( Y, \, d_{G} )$. We prove in \Cref{sec:isothermal:Riemannian} that the distortion of any quasiconformal map
\begin{equation*}
	h
	\colon
	( A(1, \, r), \, \norm{ \cdot }_{p} )
	\rightarrow
	( Z, \, d_{G'} )
\end{equation*}
is bounded from below by $D_{p}$. Our methods imply that the Riemannian surface $( Y, \, d_{G} )$ is unique up to conformal maps in the following sense: The distortion of $h$ equals $D_{p}$ if and only if the composition $h \circ f_{p}^{-1} \colon ( Y, \, d_{G} ) \rightarrow ( Z, \, d_{G'} )$ is conformal (see \Cref{sec:isothermal:Riemannian}). Therefore, for $r' > 1$, a quasiconformal map
\begin{equation*}
	h
	\colon
	( A(1, \, r), \, \norm{ \cdot }_{p} )
	\rightarrow
	( A(1, \, r'), \, \norm{ \cdot }_{2} )
\end{equation*}
can have the optimal distortion $D_{p}$ only if $r' = r$.
\subsection{Preliminaries}\label{sec:intro:definitions}
In this paper, $Y$ refers to a topological surface, i.e., a $2$-manifold that is connected. We do not assume that $Y$ is orientable. Moreover, we assume that $Y$ has a distance $d$ that induces the surface topology. We do not assume that $( Y, \, d )$ is complete. The Hausdorff $2$-measure induced by $d$ is denoted by $\mathcal{H}^{2}_{d}$. For topological reasons $\mathcal{H}^{2}_{d}$ is positive on open subsets of $Y$ \cite[Chapter 7]{dimension}.

The triple $Y_{d} = (Y, \, d, \, \mathcal{H}^{2}_{d})$ is a \emph{metric surface} if $Y$ and $d$ are as above and $\mathcal{H}^{2}_{d}$ is locally finite. If there is no chance for confusion, we omit the subscript $d$ from $Y_{d}$ and from the measure $\mathcal{H}^{2}_{d}$.

Consider a metric surface $Y$. The \emph{modulus} of a family of paths $\Gamma$ in $Y$ is
\begin{equation}
	\label{eq:modulus}
	\modulus \Gamma
	=
	\inf
	\int_{Y}
		\rho^{2}
	\D \mathcal{H}^{2},
\end{equation}
where the infimum is taken over all Borel functions $\rho \colon Y \rightarrow \left[0, \, \infty\right]$ that satisfy $\int_{ \gamma } \rho \D s \geq 1$ for all locally rectifiable paths $\gamma \in \Gamma$. If $E$, $F$ and $G$ are subsets of $Y$, the notation $\Gamma( E, \, F; \, G )$ refers to the family of paths joining $E$ to $F$ in $G$.

A homeomorphism $\phi \colon Y \rightarrow Z$ between metric surfaces is \emph{quasiconformal (in the geometric sense)} if there exists a constant $K \geq 1$ such that
\begin{align}
	\label{eq:outer:dilatation}
	K^{-1}
	\modulus \Gamma
	&\leq
	\modulus \phi \Gamma
	\leq
	K
	\modulus \Gamma
\end{align}
for all path families $\Gamma$ in $Y$, where $\phi \Gamma = \left\{ \phi \circ \gamma \mid \gamma \in \Gamma \right\}$. The smallest constant $K \geq 1$ for which \eqref{eq:outer:dilatation} holds for all path families is called the \emph{maximal dilatation} of $\phi$ and is denoted by $K( \phi )$. The map $\phi$ is $K$-quasiconformal if $K( \phi ) \leq K$. We say that a $1$-quasiconformal map is \emph{conformal (in the geometric sense)}.

The \emph{inner dilatation} of a quasiconformal map $\phi$ is the smallest constant $K_{I}( \phi ) \geq 1$ for which the latter inequality in \eqref{eq:outer:dilatation} holds for all path families. The \emph{outer dilatation} of $\phi$ is $K_{I}( \phi^{-1} )$ and it is denoted by $K_{O}( \phi )$. The \emph{distortion} of $\phi$ is $\left[ K_{O}( \phi )K_{I}( \phi ) \right]^{1/2}$.

We recall known results from the simply connected setting. To that end, we say that a compact set $Q \subset Y$ in a metric surface is a \emph{quadrilateral} if it is homeomorphic to the closed square $\left[0, \, 1\right]^{2}$ with its boundary arcs $\xi_{1}$, $\xi_{2}$, $\xi_{3}$ and $\xi_{4}$ labelled in cyclic order.

We say that an open subset $V \subset Y$ of a metric surface is a \emph{disk} if it is homeomorphic to $\mathbb{R}^{2}$. A disk $V \subset \mathbb{R}^{2}$ is \emph{reciprocal} (a \emph{reciprocal disk}) if there exists a constant $\kappa = \kappa( V ) > 0$ such that for every quadrilateral $Q \subset V$ the following three conditions hold: First, the product of the moduli of opposite sides has the upper bound
\begin{align}
	\label{upper:bound}
	\modulus \Gamma\left( \xi_{1}, \, \xi_{3}; \, Q \right)
	\modulus \Gamma\left( \xi_{2}, \, \xi_{4}; \, Q \right)
	\leq
	\kappa.
\end{align}
Secondly, the product has the lower bound
\begin{equation}
	\label{lower:bound}
	\frac{ 1 }{ \kappa }
	\leq
	\modulus \Gamma\left( \xi_{1}, \, \xi_{3}; \, Q \right)
	\modulus \Gamma\left( \xi_{2}, \, \xi_{4}; \, Q \right).
\end{equation}
Thirdly, whenever $x \in Q$ and $R > 0$ such that $V \setminus B( x, \, R ) \neq \emptyset$, then
\begin{equation}
	\label{point:zero:modulus}
	\lim_{ r \rightarrow 0^{+} }
	\modulus \Gamma\left( \overline{B}( x, \, r ) \cap V, \, V \setminus B( x, \, R ); \, \overline{B}( x, \, R ) \cap V \right)
	=
	0.
\end{equation}
The lower bound \eqref{lower:bound} holds for a uniform $\kappa$ for quadrilaterals in metric surfaces \cite{uniformization:lower:reciprocal}. Therefore only \eqref{upper:bound} and \eqref{point:zero:modulus} can fail for a disk in a metric surface. Reciprocal disks have the following characterization.
\begin{theorem}[Theorems 1.4 and 1.5, \cite{uniformization}]\label{thm:Kai}
A disk $V$ in a metric surface is reciprocal if and only if there exists a quasiconformal map $f \colon V \rightarrow U \subset \mathbb{R}^{2}$. Moreover, the map $f$ can be taken to be $2$-quasiconformal.
\end{theorem}
Romney proved that $f$ can be taken to be a quasiconformal map with $K_{O}( f ) \leq \frac{ \pi }{2}$ and $K_{I}( f ) \leq \frac{ 4 }{ \pi }$. Romney's result is sharp due to Example 2.2 of \cite{uniformization}.

A central goal of this paper is to generalize \Cref{thm:Kai} to cases where $Y$ has a non-trivial fundamental group.
\begin{definition}\label{def:local:reciprocality}
A metric surface $Y$ is \emph{locally reciprocal} if it can be covered by reciprocal disks.
\end{definition}
By \Cref{thm:Kai}, a metric surface is locally reciprocal if and only if it can be covered by quasiconformal images of subdomains of $\mathbb{R}^{2}$. As a consequence, local reciprocality is a quasiconformal invariant.

We also recall Theorem 1.6 of \cite{uniformization}: A disk $D \subset Y_{d}$ is reciprocal if there exists a constant $C = C(D) > 0$ such that for every $x \in D$ and $r > 0$
\begin{equation}
	\label{eq:local:mass:upper:bound}
	\mathcal{H}^{2}_{d}( B(x, \, r) \cap D )
	\leq
	C r^{2}.
\end{equation}
We see from this result that Riemannian surfaces, or metric surfaces that have locally $2$-bounded geometry in the sense of Heinonen--Koskela \cite{controlledgeometry,locallyboundedgeometry} are locally reciprocal.

A local mass upper bound of the type \eqref{eq:local:mass:upper:bound} is not a necessary condition for local reciprocality: there exists a distance $d_{w}$ on $\mathbb{R}^{2}$ for which the identity map $\id \colon \mathbb{R}^{2} \rightarrow \mathbb{R}^{2}_{d_{w}}$ is conformal and that the Hausdorff $2$-density of $\mathcal{H}^{2}_{d_{w}}$ at the origin is $\infty$; see \cite[Section 5.1]{RRR} for details.

There are metric surfaces that are not locally reciprocal. An example is obtained by collapsing a small geodesic disk of a Riemannian surface to a point $y$ and endowing the quotient space $Y$ with the quotient metric \cite[Example 5.9]{collapsedisk}. The space $Y \setminus \left\{ y \right\}$ is locally reciprocal but every domain containing $y$ violates \eqref{point:zero:modulus}. See also Example 2.1 of \cite{uniformization}.
\subsection{Main results}\label{sec:intro:main}
We introduce three main theorems in this section. The first result provides a characterization of locally reciprocal surfaces. The latter results are applications.
\begin{restatable}{theorem}{thmuniformizationgeneral}\label{thm:uniformization:general}
A metric surface $Y_{d}$ is locally reciprocal if and only if there exists a complete Riemannian surface $Z_{d_{G}}$ of curvature $-1$, $0$, or $1$ together with a $\frac{\pi}{2}$-quasiconformal map
\begin{equation}
	\label{eq:uniformization:map}
	\phi
	\colon
	Z_{d_{G}}
	\rightarrow
	Y_{d},
\end{equation}
such that the outer dilatation $K_{O}( \phi ) \leq \frac{4}{\pi}$, the inner dilatation $K_{I}( \phi ) \leq \frac{ \pi }{ 2 }$, and the distortion $\left[ K_{O}( \phi )K_{I}( \phi ) \right]^{1/2} \leq \sqrt{2}$. The bounds on the dilatations and the distortion are best possible.
\end{restatable}
We prove a stronger version of \Cref{thm:uniformization:general} in \Cref{sec:isothermal:Riemannian}; see \Cref{thm:isothermal:surface} and \Cref{cor:essentiallymaintheorem}.

The classical uniformization theorem and \Cref{thm:uniformization:general} imply the following.
\begin{corollary}\label{thm:unifomization:general:cor}
Let $Y_{d}$ be a locally reciprocal metric surface. Then every disk $D \subset Y_{d}$ is reciprocal.
\end{corollary}
\Cref{thm:unifomization:general:cor} implies that reciprocality is a local condition: a disk in a metric surface is reciprocal if and only if it is locally reciprocal.

When we apply our techniques to Alexandrov surfaces we obtain the following uniformization result.
\begin{restatable}{theorem}{thmuniformizationalexandrov}\label{thm:uniformization:alexandrov}
Any Alexandrov surface $Y_{d}$ is locally reciprocal and the map in \Cref{thm:uniformization:general} can be taken to be conformal and absolutely continuous in measure.
\end{restatable}
Alexandrov spaces are metric spaces that satisfy a generalized curvature lower bound; see for example \cite[Chapter 10]{lengthspace}. We prove the result in \Cref{sec:stuff}.

We obtain the following quasisymmetric uniformization of compact metric surfaces.
\begin{restatable}{theorem}{thmcompactQSunif}\label{thm:compact:QS:unif}
Suppose that $Y_{d}$ is a compact Ahlfors $2$-regular linearly locally contractible metric surface. Then the map in \Cref{thm:uniformization:general} can be taken to be $\eta$-quasisymmetric with $\eta$ depending only on the data of $Y_{d}$.
\end{restatable}
The data of $Y_{d}$ refers to the constants in the definition of Ahlfors $2$-regularity and linear local contractibility. The main point of \Cref{thm:compact:QS:unif} is that we have optimal control on the dilatations of $\phi$, and at the same time, good control on the quasisymmetric distortion function of $\phi$. Also $Y_{d}$ does not have to be orientable, cf. \cite[Theorem 2]{compactqs}, \cite[Theorem 1.1]{bonk-kleinerthm}. Similar dilatation control is achieved in \cite{collapsedisk} when $Y_{d}$ is homeomorphic to the sphere $\mathbb{S}^{2}$. We prove \Cref{thm:compact:QS:unif} in \Cref{sec:recip:qs}.
\subsection{Structure of the paper}
We recall some important results from analysis on metric spaces in \Cref{sec:prelim}. Readers familiar with path modulus and Newtonian--Sobolev spaces \cite{newtonianspacesorigin,metricsobolev} may consider skipping this section.

In \Cref{sec:analysisonQCmaps}, we define pointwise outer and inner dilatations for quasiconformal maps; these pointwise quantities describe the infinitesimal behaviour of quasiconformal maps. We prove that the quantities satisfy useful composition laws. Moreover, we show that the outer and inner dilatations of quasiconformal maps can be recovered from the corresponding pointwise quantities.

We recall some known results for Sobolev maps from planar domains into metric spaces. These results allow us to associate a field of convex bodies to a given quasiconformal map in \Cref{sec:uniformization}, where we also construct the isothermal coordinates and a Riemannian distance $d_{G}$ on a given locally reciprocal surface.

We outline the construction of the isothermal coordinates. Suppose that $V$ is an open subset of a metric surface and that there exists a quasiconformal homeomorphism $\phi \colon \mathbb{R}^{2} \supset U \rightarrow V$. We associate a field of ellipses to the field of convex bodies of $\phi$. The ellipse field is defined using the Banach--Mazur distance and the techniques of \cite{upper:modulus:bound}. We solve the corresponding Beltrami equation and prove that the solution $g \colon U \rightarrow W \subset \mathbb{R}^{2}$ is such that the product of the pointwise outer and inner dilatations of $\phi \circ g^{-1}$ is minimal among all choices of $g$. The map $f = ( \phi \circ g^{-1} )^{-1}$ provides isothermal coordinates for $V$.

We prove that the transition maps between isothermal coordinates are conformal, thus every locally reciprocal surface has a natural atlas. The distance $d_{G}$ is then constructed following a proof of the classical uniformization theorem.

In \Cref{sec:uniformization:main:theorem}, using the results in \Cref{sec:analysisonQCmaps} and \Cref{sec:uniformization}, we associate a field of convex bodies to a given locally reciprocal surface $Y$. We prove that studying quasiconformal maps between locally reciprocal surfaces reduces to understanding these bodies and how they are distorted by such maps. We express the Jacobians, minimal upper gradients, and pointwise dilatations of quasiconformal maps using these bodies.

In \Cref{sec:applications}, we show some applications of the theory of \Cref{sec:uniformization:main:theorem}. In \Cref{sec:isothermal:Riemannian}, we prove that given a locally reciprocal surface $Y_{d}$ and the Riemannian surface $Y_{d_{G}}$ constructed in \Cref{sec:uniformization}, the \emph{uniformization map} $u = \id_{Y} \colon Y_{d_{G}} \rightarrow Y_{d}$ provides a (global) \emph{isothermal parametrization} of $Y$: The product of the pointwise dilatations of $u$ is minimal over all other parametrizations of $Y_{d}$ by Riemannian surfaces. The results of this section imply that the distance $d_{G}$ is unique up to a conformal map. We prove that the uniformization map satisfies the conclusions of \Cref{thm:uniformization:general}.

In \Cref{sec:stuff}, we study the conformal automorphism group of a locally reciprocal surface $Y_{d}$. They turn out to be conformal automorphisms of the Riemannian surface $Y_{d_{G}}$, which improves the regularity of such maps. Also, we show that a locally reciprocal surface $Y_{d}$ is conformally equivalent to a Riemannian surface if and only if the convex bodies associated to $Y_{d}$ are ellipses. Equivalently, the uniformization map $u$ is conformal. We prove that if $Y_{d}$ is an Alexandrov surface, the uniformization map $u$ is conformal and absolutely continuous in measure thereby showing \Cref{thm:uniformization:alexandrov}.

\Cref{sec:recip:qs} is split into two parts. In the first part, we prove that the uniformization map $u$ is locally $\eta$-quasisymmetric if and only if the metric surface $Y_{d}$ has an atlas of $\eta'$-quasisymmetric charts. The quasisymmetric distortions $\eta$ and $\eta'$ depend on each other in a quantitative way. This statement has a qualitative version as well.

In the second part, we prove a quantitative quasisymmetric uniformization result when $Y_{d}$ is an Ahlfors $2$-regular linearly locally contractible compact metric surface. Using the isothermal charts constructed in \Cref{sec:uniformization} and by slightly modifying the proof of Theorem 2 of \cite{compactqs}, we prove that the uniformization map $u$ is $\eta$-quasisymmetric, where $\eta$ depends only on the data of $Y_{d}$. If $Y_{d}$ is homeomorphic to the sphere $\mathbb{S}^{2}$, the precise claim requires some care due to the abundance of Möbius transformations of $\mathbb{S}^{2}$. \Cref{thm:compact:QS:unif} follows from this result. We highlight some open problems in \Cref{sec:concluding}.
\subsection*{Acknowledgements}
Part of this paper was completed when the author was visiting the University of Michigan, Ann Arbor. The author thanks the Department of Mathematics for their hospitality.

This paper is part of the author's PhD thesis. The author thanks his advisor Kai Rajala and Matthew Romney for many helpful conversations.
\section{Newtonian--Sobolev spaces}\label{sec:prelim}
We introduce some terminology and notation. The reader familiar with modulus and Newtonian--Sobolev spaces may consider skipping this section.
\begin{definition}\label{def:metric:measure:space}
A \emph{metric measure space} is a triple $U = (U, \, d, \, \mu_{U})$, where $(U, \, d)$ is a separable metric space and $\mu_{U}$ is a nonnegative Borel regular outer measure on $U$ that is positive on metric balls of $U$ and such that $\mu_{U}$ is locally finite.
\end{definition}
\subsection{Paths}
We recall some basic properties of paths in metric spaces. %We mainly work on paths that are defined on compact subintervals of $\mathbb{R}$.
%\begin{definition}\label{def:paths}
\emph{A path} $\gamma$ in a metric space $U$ is a continuous map from a subinterval $\emptyset \neq I \subset \mathbb{R}$ into $U$. The set $I$ is the \emph{domain} of $\gamma$ and is denoted by $\mathrm{dom}( \gamma )$. A \emph{compact path} is a path whose domain is compact. The length of a compact path is
\begin{equation}
	\label{eq:length:partition}
	L( \gamma )
	=
	\sup
	\sum_{ i = 0 }^{ n }
		d( \gamma(t_{i}), \, \gamma( t_{i+1} ) ),
\end{equation}
where the supremum is taken over finite partitions $\min I = t_{0} \leq t_{1} \leq \ldots \leq t_{n+1} = \max I$.

A path $\gamma$ is \emph{rectifiable} if the path is compact and $L( \gamma ) < \infty$. A path $\gamma'$ is a \emph{subpath} of $\gamma$ if $\gamma' = \restr{ \gamma }{ I' }$ for some subinterval $I'$ of the domain of $\gamma$. A path $\gamma$ is \emph{locally rectifiable} if every compact subpath of $\gamma$ is rectifiable.

Given a path $\gamma \colon I \rightarrow U$, the \emph{metric speed} of $\gamma$ at $t \in I$ is
\begin{equation}
	\label{eq:metric:speed:path}
	v_{\gamma}(t)
	=
	\lim_{ I \ni s \rightarrow t }
		\frac{
			d( \gamma(s), \, \gamma(t) )
		}{
			\abs{ s - t }
		}
\end{equation}
whenever the limit exists.

A path $\gamma \colon I \rightarrow U$ is \emph{absolutely continuous} if $\gamma$ is rectifiable and $\gamma$ maps sets $N \subset I$ of $m_{1}$-measure zero to sets of $\mathcal{H}^{1}$-measure zero in $U$; the measure $m_{1}$ is the Lebesgue measure on $\mathbb{R}$ and $\mathcal{H}^{1}$ the Hausdorff $1$-measure on $U$. The collection of compact absolutely continuous paths of $U$ is denoted by $AC( U )$. The elements of $AC( U )$ that have positive length are denoted by $AC_{+}( U )$.

A path $\gamma' \colon I' \rightarrow U$ is a \emph{reparametrization} of a path $\gamma$ if there exists a non-decreasing, continuous and surjective function $p \colon I \rightarrow I'$ such that $\gamma = \gamma' \circ p$. If $p$ is also absolutely continuous, we say that $\gamma'$ is an \emph{absolutely continuous reparametrization of $\gamma$}.

A path $\gamma_{s}$ is a \emph{unit speed (re)parametrization} of a path $\gamma$ if $\gamma = \gamma_{s} \circ s$ for some non-decreasing, continuous and surjective function $s$ and that $v_{\gamma_s} = 1$ $m_{1}$-almost everywhere in the domain of $\gamma_{s}$. A path has \emph{unit speed} if $\gamma = \gamma_{s}$.
%\end{definition}

We recall some basic results concerning paths; see for example \cite[Section 5.1]{metricsobolev} or \cite{chainrulepaths} for proofs. Any rectifiable path $\gamma$ has a unit speed parametrization $\gamma_{s}$. Additionally, the unit speed parametrization $\gamma_{s}$ of a rectifiable path is Lipschitz and thus absolutely continuous.

If $\gamma = \gamma_{s} \circ s$ is a unit speed parametrization of a rectifiable path $\gamma$, then $\gamma$ is absolutely continuous if and only if $\gamma_{s}$ is an absolutely continuous reparametrization of $\gamma$ (if and only if $s$ is absolutely continuous). If $\gamma$ is absolutely continuous, then the metric speed $v_{\gamma}$ exists at $m_{1}$-almost every $t \in \mathrm{dom}( \gamma )$.
\begin{definition}\label{def:path:integral}
Suppose that $\rho \colon U \rightarrow \left[0, \, \infty\right]$ is a Borel function and $\gamma$ is a rectifiable path. Then the \emph{path integral of $\rho$ over $\gamma$} is
\begin{equation}
	\label{path:integral}
	\int_{ \gamma }
		\rho
	\D s
	=
	\int_{ \mathrm{dom}( \gamma_{s} ) }
		\rho \circ \gamma_{s}
	\D s,
\end{equation}
where $\gamma_{s}$ is a unit speed parametrization of $\gamma$. 
\end{definition}
The notation $\D s$ refers to the integration over the Lebesgue measure $m_{1}$ of $\mathrm{dom}( \gamma_{s} )$.

Suppose that $\gamma$ is absolutely continuous path and $\gamma'$ is an absolutely continuous reparametrization of $\gamma$. Then there exists a non-decreasing, continuous, surjective and absolutely continuous function $p$ for which $\gamma = \gamma' \circ p$. The metric speeds $v_{\gamma}$ and $v_{\gamma'}$ satisfy the chain rule
\begin{equation}
	\label{eq:metric:speed:chainrule}
	v_{ \gamma }
	=
	\left( v_{ \gamma' } \circ p \right)
	p'
	\in
	L^{1}( \mathrm{dom}( \gamma ) ),
\end{equation}
where the right-hand side is interpreted to be zero whenever $p' = 0$. This is proved by Theorem 3.16 and Remark 3.4 of \cite{chainrulepaths}.

If $\gamma \colon I \rightarrow U$ is absolutely continuous, then via the chain rule \eqref{eq:metric:speed:chainrule},
\begin{equation}
	\label{path:integral:changeofvariables}
	\int_{ \gamma }
		\rho
	\D s
	=
	\int_{ I }
		\rho( \gamma(s) )
		v_{ \gamma }(s)
	\D s.
\end{equation}
If $\gamma$ is absolutely continuous, the measure $v_{ \gamma }m_{1}$ is the \emph{variation measure of $\gamma$}. The measure $v_{ \gamma }m_{1}$ is defined as follows: For Borel sets $A \subset \mathrm{dom}( \gamma )$,
\begin{equation*}
	v_{\gamma}m_{1}( A )
	=
	\int_{ A }
		v_{\gamma}
	\D m_{1}.
\end{equation*}
If $\gamma$ is a locally rectifiable path, the path integral of a Borel function $\rho \colon U \rightarrow \left[0, \, \infty\right]$ over $\gamma$ is defined as
\begin{equation*}
	\int_{ \gamma }
		\rho
	\D s
	=
	\sup
	\int_{ \gamma' }
		\rho
	\D s,
\end{equation*}
where the supremum is taken over compact subpaths $\gamma'$ of $\gamma$.
\subsection{Modulus}
An important tool for this paper is the path modulus, which we defined in \eqref{eq:modulus}. The definition extends to a metric measure space $U = ( U, \, d, \, \mu_{U} )$, where the Hausdorff $2$-measure is replaced by the measure $\mu_{U}$ (recall \Cref{def:metric:measure:space}).

A property holds for \emph{almost every path} if the family of paths where the property fails has zero modulus. If $A$ is a Borel subset of $U$, then $\Gamma^{+}_{A}$ is the collection of locally rectifiable paths $\gamma$ of $U$ for which
\begin{equation*}
	L( \gamma \cap A )
	=
	\int_{ \gamma }
		\chi_{A}
	\D s
	>
	0.
\end{equation*}
A locally rectifiable path $\gamma$ has \emph{positive length in $A$} if $\gamma \in \Gamma^{+}_{A}$ and \emph{zero length in $A$} otherwise.

We recall some basic properties of modulus. See Section 5.2 of \cite{metricsobolev} or Section 2 of \cite{williams} for the proofs.
\begin{enumerate}[label=(\alph*)]
\item\label{mod:1} The family of paths that are not locally rectifiable has zero modulus. If a path family contains a constant path, the modulus of that path family is infinite.
\item\label{mod:2} The modulus is an outer measure.
\item\label{mod:3} If $\rho \in L^{2}_{loc}( U )$ is a Borel function, then $\rho$ is integrable along almost every path.
\item\label{mod:4} A path family $\Gamma$ has zero modulus if and only if there exists a nonnegative Borel function $\rho \in L^{2}( U )$ such that for every locally rectifiable $\gamma \in \Gamma$,
\begin{equation*}
	\int_{ \gamma }
		\rho
	\D s
	=
	\infty.
\end{equation*}
\item\label{mod:5} If $A \subset U$ is a Borel set and $\mu_{U}( A ) = 0$, then $\Gamma^{+}_{A}$ has zero modulus; almost every locally rectifiable path $\gamma$ has zero length in $A$. We sharpen this result in \Cref{identity:minimal:upper:gradient}.
\end{enumerate}
\subsection{Upper gradients}\label{sec:sobolev}
We fix a metric measure space $U = ( U, \, d, \mu_{U} )$ as in \Cref{def:metric:measure:space}. Moreover, suppose that $V$ is a separable metric space and $h \colon U \rightarrow V$ is a map.
%\begin{definition}\label{def:sobolev:uppergradient}

A Borel function $\rho \colon U \rightarrow \left[0, \, \infty\right]$ is an \emph{upper gradient of $h$} if for every absolutely continuous path $\gamma \colon \left[a, \, b\right] \rightarrow U$,
\begin{equation}
	\label{eq:uppergradient:inequality}
	d( h \circ \gamma(a), \, h \circ \gamma(b) )
	\leq
	\int_{ \gamma }
		\rho
	\D s.
\end{equation}
The triple $( h, \, \rho, \, \gamma )$ satisfies the \emph{upper gradient inequality} if \eqref{eq:uppergradient:inequality} holds. We say that $\rho$ is a \emph{weak upper gradient} of $h$ if \eqref{eq:uppergradient:inequality} holds for almost every path in $U$. 

A weak upper gradient $\rho \in L^{2}_{loc}( U )$ of $h$ is \emph{minimal} if for every other weak upper gradient $\rho' \in L^{2}_{loc}( U )$ of $h$ we have that $\rho \leq \rho'$ $\mu_{U}$-almost everywhere. Such a function $\rho$ is called a \emph{minimal upper gradient of $h$} and denoted by $\rho_{h}$.
%\end{definition}

Every map $h \colon U \rightarrow V$ that has a weak upper gradient in $L^{2}_{loc}( U )$ has a minimal one. Moreover, suppose that $\rho$ is a minimal upper gradient of $h$. Then a Borel function $\rho_{1} \colon U \rightarrow \left[0, \, \infty\right]$ is a minimal upper gradient of $h$ if and only if $\rho_{1} = \rho$ $\mu_{U}$-almost everywhere (see for example Section 6 of \cite{metricsobolev}). Therefore it makes sense to talk about \emph{the} minimal upper gradient of $h$ and about its representatives.

Lemmas 3.2 and 3.3 of \cite{williams} establish that a Borel function $\rho \colon U \rightarrow \left[0, \, \infty\right] \in L^{2}_{loc}( U )$ is a weak upper gradient of $h$ if and only if for almost every $\gamma \in AC_{+}( U )$, the composition $h \circ \gamma$ is an absolutely continuous path for which
\begin{equation}
	\label{eq:metric:speed:uppergrad}
	v_{ h \circ \gamma }
	\leq
	(\rho \circ \gamma)
	v_{ \gamma }
	\in
	L^{1}( \mathrm{dom}( \gamma ) )
\end{equation}
$m_{1}$-almost everywhere in $\mathrm{dom}( \gamma )$.

Sometimes \eqref{eq:metric:speed:uppergrad} is stated only when $\gamma$ is rectifiable and when the path $\gamma$ is replaced by the unit speed parametrization $\gamma_{s}$ of $\gamma$. However, if $\gamma = \gamma_{s} \circ s$ is absolutely continuous, the chain rule \eqref{eq:metric:speed:chainrule} shows that \eqref{eq:metric:speed:uppergrad} holds since $v_{ h \circ \gamma }$ equals $\left( v_{ h \circ \gamma_{s} } \circ s \right) s'$ and $v_{ \gamma }$ equals $s'$ $m_{1}$-almost everywhere, where $\left( v_{ h \circ \gamma_{s} } \circ s \right) s'$ is interpreted to be zero in the set $\left\{ s' = 0 \right\}$.
\begin{definition}\label{def:sobolev}
A map $h \colon U \rightarrow V$ is in the \emph{Newtonian--Sobolev space $N^{1, \, 2}_{loc}( U, \, V )$} if $h \in L^{2}_{loc}( U, \, V )$ and $h$ has a minimal upper gradient $\rho_{h} \in L^{2}_{loc}( U )$.

The map $h$ is in the \emph{Newtonian--Sobolev space $N^{1, \, 2}( U, \, V )$} if $h \in N^{1, \, 2}_{loc}( U, \, V )$, $h \in L^{2}( U, \, V )$ and $\rho_{h} \in L^{2}( U )$.
\end{definition}
We only apply the definition of $L^{2}( U, \, V )$ when $\mu_{U}( U ) < \infty$. In this case, $L^{2}( U, \, V )$ consists of measurable maps $u \colon U \rightarrow V$ such that $u_{x}( v ) = d( u(v), \, x ) \in L^{2}( U )$ for some $x \in V$.

The space $L^{2}_{loc}( U, \, V )$ consist of measurable maps $u \colon U \rightarrow V$ for which every point $x \in U$ has a neighbourhood $U_{x}$ with finite $\mu_{U}$-measure and for which $\restr{ u }{ U_{x} } \in L^{2}( U_{x}, \, V )$.
\begin{lemma}\label{absolutecontinuity}
If $h \in N^{1, \, 2}_{loc}( U, \, V )$, then for almost every $\gamma \in AC_{+}( U )$, the map $h \circ \gamma$ is absolutely continuous.
\end{lemma}
Recall that $AC_{+}( U )$ refers to absolutely continuous paths of positive length. \Cref{absolutecontinuity} follows from \eqref{eq:metric:speed:uppergrad} and the fact that $\left( \rho \circ \gamma \right) v_{\gamma} \in L^{1}( \mathrm{dom}( \gamma ) )$ for almost every $\gamma \in AC_{+}( U )$.
\section{Analytic properties of quasiconformal maps}\label{sec:analysisonQCmaps}
We recall some known results of quasiconformal maps in \Cref{sec:qc}. We prove in \Cref{lem:QC:differentequivalent} that quasiconformal maps behave well along almost every path and quasiconformal maps have good measure-theoretic properties outside a particular set (\Cref{lem:Condition(N)}).

Pointwise dilatations of quasiconformal maps are defined and studied in \Cref{sec:pointwise:dilatation}. We prove that they satisfy a composition law (\Cref{lem:outer:inner:characterization}) and global dilatations can be recovered from them (\Cref{cor:global:inner:and:outer:dilatation}). We state these results for metric measure spaces.

We show in \Cref{sec:approximate} that when the domain of the quasiconformal map is an open subset of $\mathbb{R}^{2}$, we can associate a field of norms (of convex bodies) to such maps.
\subsection{Basic properties}\label{sec:qc}
For this section $U = ( U, \, d, \, \mu_{U} )$ and $V = ( V, \, d, \, \mu_{V} )$ are metric measure spaces in the sense of \Cref{def:metric:measure:space}. Also, the map $\phi \colon U \rightarrow V$ is a homeomorphism. We recall some basic terminology from measure theory.

The \emph{$\phi$-pullback measure} of $\mu_{V}$ is the measure $\phi^{*}\mu_{V}$ defined by $\phi^{*}\mu_{V}( B ) = \mu_{V}( \phi(B) )$ for Borel sets $B \subset U$. The \emph{Jacobian} of $\phi$ is the Radon--Nikodym derivative of the absolutely continuous part of $\phi^{*}\mu_{V}$ with respect to $\mu_{U}$; see \cite[Section 3.1]{bogachev1}. The Jacobian of $\phi$ is denoted by $J_{\phi}$.

The measure $\phi_{*}\mu_{U}$ is the $\phi^{-1}$-pullback of $\mu_{U}$. It is also called the \emph{$\phi$-pushforward measure}.

The map $\phi$ satisfies \emph{Condition ($N$)} if $\phi^{*}\mu_{V}$ is absolutely continuous with respect to $\mu_{U}$: that is, for any Borel set $B \subset U$, $\mu_{U}( B ) = 0$ implies that $\mu_{V}( \phi(B) ) = \phi^{*}\mu_{V}( B ) = 0$. The map $\phi$ satisfies \emph{Condition ($N^{-1}$)} if $\phi^{-1}$ satisfies Condition ($N$).

The following statement is the two-dimensional case of Theorem 1.1 of \cite{williams}. Recall that $\rho_{\phi}$ is the minimal upper gradient of $\phi$.
\begin{theorem}\label{thm:QC:differentequivalent}
Let $\phi \colon U \rightarrow V$ be a homeomorphism. The following conditions are equivalent with the same constant $K$:
\begin{enumerate}[label=(\alph*)]
\item\label{QC:analytic} The map $\phi$ is in the Newtonian--Sobolev space $N^{1, \, 2}_{loc}( U, \, V )$ and
\begin{equation}
	\label{eq:analytic:outer}
	\rho_{ \phi }^{2}
	\leq
	K J_{\phi}
\end{equation}
$\mu_{U}$-almost everywhere in $U$.
\item\label{QC:geometric} For every path family $\Gamma$ in $U$,
\begin{equation}
	\label{eq:geometric:outer}
	\modulus( \Gamma )
	\leq
	K \modulus( \phi \Gamma ),
\end{equation}
where $\phi \Gamma = \left\{ \phi \circ \gamma \mid \gamma \in \Gamma \right\}$.
\end{enumerate}
\end{theorem}
The smallest positive constant for which \eqref{eq:geometric:outer} holds for all path families is called the \emph{outer dilatation} of $\phi$ and is denoted by $K_{O}( \phi )$. The \emph{inner dilatation} of $\phi$ is $K_{O}( \phi^{-1} )$ and is denoted by $K_{I}( \phi )$. The \emph{maximal dilatation} $K( \phi )$ is the maximum of $K_{O}( \phi )$ and $K_{I}( \phi )$. We recall from \eqref{eq:outer:dilatation} that we say that $\phi$ is quasiconformal if the maximal dilatation $K( \phi )$ is finite. The map $\phi$ is $K$-quasiconformal if $K( \phi )$ is bounded from above by $K$.

We show in \Cref{sec:qcmaps:surfaces} that if $\phi$ is a quasiconformal map between locally reciprocal surfaces, then $K_{O}( \phi ) \geq 1$, $K_{I}( \phi ) \geq 1$ and $K( \phi ) \geq 1$. Therefore for our purposes, the definition we stated just now coincides with the one in \Cref{sec:intro:definitions}. It is more convenient to not restrict $K( \phi )$, $K_{O}( \phi )$, and $K_{I}( \phi )$ from below in this section.

For the purposes of the next definition, recall that $AC_{+}( U )$ is the collection of absolutely continuous paths on $U$ that have positive length. The collection $AC_{+}( V )$ is defined similarly.
\begin{definition}\label{def:QC:singularset;goodpaths}
Consider a homeomorphism $\phi \colon U \rightarrow V$. The collection of \emph{$\phi$-good paths}, denoted by $\Gamma_{G}( \phi )$, consists of paths with the following properties.
\begin{enumerate}[label=(\alph*)]
\item The paths $\gamma$ and $\phi \circ \gamma$ are elements of $AC_{+}( U )$ and $AC_{+}( V )$, respectively;
\item The variation measures $v_{\gamma} m_{1}$ and $v_{ \phi \circ \gamma } m_{1}$ of $\gamma$ and $\phi \circ \gamma$, respectively, are mutually absolutely continuous.
\end{enumerate}
The collection of locally rectifiable paths for which every compact subpath of positive length is an element of $\Gamma_{G}( \phi )$ is denoted by $\Gamma_{G}^{loc}( \phi )$.

The locally rectifiable paths of $U$ that have positive length and that are not elements of $\Gamma_{G}^{loc}( \phi )$ are called \emph{$\phi$-singular paths} and their collection is denoted by $\Gamma_{S}( \phi )$.
\end{definition}
\begin{lemma}[Singular paths]\label{lem:QC:differentequivalent}
For a quasiconformal map $\phi \colon U \rightarrow V$, the $\phi$-singular paths $\Gamma_{S}( \phi )$ have zero modulus, and for Borel sets $B \subset U$,
\begin{equation*}
	\phi
	\colon
	\Gamma_{G}^{loc}( \phi ) \cap \Gamma^{+}_{B}
	\rightarrow
	\Gamma_{G}^{loc}( \phi^{-1} ) \cap \Gamma^{+}_{ \phi(B) }, \quad
	\gamma
	\mapsto
	\phi \circ \gamma
\end{equation*}
is well-defined and bijective. Moreover, its restriction to $\Gamma_{G}( \phi ) \cap \Gamma^{+}_{B}$ induces a bijection onto $\Gamma_{G}( \phi^{-1} ) \cap \Gamma^{+}_{\phi(B)}$.
\end{lemma}
\begin{proof}
We start with the claim that the $\phi$-singular paths $\Gamma_{S}( \phi )$ have zero modulus. Fix minimal upper gradients $\rho_{\phi} \in L^{2}_{loc}( U )$ and $\rho_{\phi^{-1}} \in L^{2}_{loc}( V )$ of $\phi$ and $\phi^{-1}$, respectively. Then, by recalling \eqref{eq:metric:speed:uppergrad} and \Cref{absolutecontinuity}, there exist path families $\Gamma_{ \phi }$ and $\Gamma_{ \phi^{-1} }$ of zero modulus such that the following properties hold.

First, for every $\gamma \in AC_{+}(U) \setminus \Gamma_{\phi}$, the path $\phi \circ \gamma$ is in $AC(V)$ and
\begin{equation}
	\label{eq:phi:uppergradient}
	v_{ \phi \circ \gamma }
	\leq
	\left( \rho_{ \phi } \circ \gamma \right)
	v_{\gamma}
	\in
	L^{1}( \mathrm{dom}( \gamma) )
\end{equation}
$m_{1}$-almost everywhere. Secondly, for every $\phi \circ \gamma$ in $AC_{+}(V) \setminus \Gamma_{\phi^{-1}}$, $\gamma$ is in $AC(U)$ and
\begin{equation}
	\label{eq:phi:inverse:uppergradient}
	v_{ \gamma }
	\leq
	\rho_{ \phi^{-1} } \circ ( \phi \circ \gamma )
	v_{ \phi \circ \gamma }
	\in
	L^{1}( \mathrm{dom}( \gamma) )
\end{equation}
$m_{1}$-almost everywhere.

Define $\Gamma_{0} = \Gamma_{\phi} \cup \phi^{-1}( \Gamma_{ \phi^{-1} } )$. The family $\Gamma_{0}$ has zero modulus by construction. Given $\gamma \in AC_{+}( U ) \setminus \Gamma_{0}$, the inequalities \eqref{eq:phi:uppergradient} and \eqref{eq:phi:inverse:uppergradient} hold for $m_{1}$-almost every point in $\mathrm{dom}( \gamma)$. Therefore the variation measures $v_{ \gamma } \cdot m_{1}$ and $v_{ \phi \circ \gamma } \cdot m_{1}$ are absolutely continuous with respect to one another.

We have deduced that $AC_{+}( U ) \setminus \Gamma_{0}$ are $\phi$-good paths. This implies that $AC_{+}( U ) \setminus \Gamma_{G}( \phi )$ has zero modulus. Then Property \ref{mod:4} of modulus shows that there exists an $L^{2}$-integrable Borel function $\rho \colon U \rightarrow \left[0, \, \infty\right]$ whose path integral over every $\theta \in AC_{+}( U ) \setminus \Gamma_{G}( \phi )$ is $\infty$. Thus if $\gamma \in \Gamma_{S}( \phi )$ is locally rectifiable, $\int_{ \gamma } \rho \D s = \infty$. Property \ref{mod:4} of modulus implies that $\Gamma_{S}( \phi )$ has zero modulus.

Fix a Borel set $B \subset U$. The restriction of the map $\gamma \mapsto \phi \circ \gamma$ to $\Gamma^{+}_{ B } \cap \Gamma_{G}( \phi )$ is a bijection onto $\Gamma^{+}_{ \phi(B) } \cap \Gamma_{G}( \phi^{-1} )$ since the measures $v_{\gamma} \cdot m_{1}$ and $v_{ \phi \circ \gamma } \cdot m_{1}$ are mutually absolutely continuous; the path $\gamma \in \Gamma_{G}( \phi )$ has zero length in a Borel set $B$ if only if $\phi \circ \gamma$ has zero length in $\phi(B)$. This implies that the map
\begin{equation*}
	\phi
	\colon
	\Gamma_{G}^{loc}( \phi ) \cap \Gamma^{+}_{B}
	\rightarrow
	\Gamma_{G}^{loc}( \phi^{-1} ) \cap \Gamma^{+}_{ \phi(B) }, \quad
	\gamma
	\mapsto
	\phi \circ \gamma
\end{equation*}
is well-defined and bijective.
\end{proof}
\begin{proposition}[Absolute continuity]\label{lem:Condition(N)}
Suppose that $\phi \colon U \rightarrow V$ is quasiconformal. Then there exists a Borel set $B_{0} \subset U$ such that $\rho_{ \id_{U} } = \chi_{ U \setminus B_{0} }$ and $\rho_{ \id_{V} } = \chi_{ V \setminus \phi(B_{0}) }$ are minimal upper gradients of $\id_{U}$ and $\id_{V}$, respectively.

Additionally, the following four conditions are equivalent for Borel sets $B \subset U$,
\begin{align*}
	\rho_{\id_{U}}\mu_{U}( B )
	&=
	0;
	\quad
	\modulus \Gamma^{+}_{B}
	=
	0;
	\\
	\rho_{\id_{V}}\mu_{V}( \phi(B) )
	&=
	0;
	\quad
	\modulus \Gamma^{+}_{\phi(B)}
	=
	0.
\end{align*}
In particular, the map $\phi \colon ( U, \, \rho_{\id_{U}} \mu_{U} ) \rightarrow ( V, \, \rho_{\id_{V}} \mu_{V} )$ satisfies Conditions ($N$) and ($N^{-1}$).
\end{proposition}
\begin{remark}
We postpone the proof of \Cref{lem:Condition(N)} until the end of \Cref{sec:qc}. Note that if $\mu_{U}( B_{0} ) = 0$, then the map $\phi$ satisfies Condition ($N^{-1}$) as a map
\begin{equation*}
	\phi \colon ( U, \, \mu_{U} ) \rightarrow ( V, \, \rho_{\id_{V}} \mu_{V} ).
\end{equation*}
Several important geometric assumptions imply $\mu_{U}( B_{0} ) = 0$. For example, if $U$ has locally $2$-bounded geometry in the sense of Heinonen--Koskela \cite{controlledgeometry,locallyboundedgeometry} or if $U$ is a PI-space in the sense of \cite{PIspaces}. We do not need these facts, so we omit the proofs.

Recall that a metric surface $U_{d}$ is a metric measure space $( U, \, d, \, \mathcal{H}^{2}_{d} )$, where $U$ is homeomorphic to a surface and $d$ is a distance inducing the surface topology. The measure is the Hausdorff $2$-measure induced by $d$.

Proposition 17.1 of \cite{uniformization} provides an example of a metric surface $Y \subset \mathbb{R}^{3}$ (that is locally reciprocal), where $B_{0}$ has positive $\mathcal{H}^{2}_{d}$-measure. The fact that $B_{0}$ has positive measure follows from \Cref{lem:Condition(N):inverse}. This means that in general we cannot simply take $B_{0} = \emptyset$ in \Cref{lem:Condition(N)}.
\end{remark}
\begin{corollary}\label{lem:Condition(N):inverse}
Suppose that $U_{d}$ is a metric surface biLipschitz homeomorphic to a planar domain.

Then the minimal upper gradient $\rho_{ \id_{U} }$ equals $\chi_{U}$ $\mathcal{H}^{2}_{d}$-almost everywhere, and any quasiconformal map $\phi \colon U_{d} \rightarrow V$ satisfies Condition ($N^{-1}$).

Furthermore, the map $\phi$ satisfies Condition ($N$) if and only if $\rho_{ \id_{V} } = \chi_{V}$ $\mu_{V}$-almost everywhere.
\end{corollary}
\begin{proof}[Proof of \Cref{lem:Condition(N):inverse} assuming \Cref{lem:Condition(N)}]
We use the following two facts. A biLipschitz map between metric surfaces is quasiconformal and satisfies Conditions ($N$) and ($N^{-1}$). Also, if $W$ is a planar domain, then $\chi_{W}$ is a minimal upper gradient of $\id_{W}$. Combining these two facts with \Cref{lem:Condition(N)} implies that for any metric surface $U$ biLipschitz homeomorphic to a planar domain, $\rho_{ \id_{U} } = \chi_{U}$ $\mathcal{H}^{2}_{d}$-almost everywhere.

Suppose that $\phi \colon U_{d} \rightarrow V$ is a quasiconformal map and that $U_{d}$ is biLipschitz homeomorphic to a planar domain. Since $\rho_{ \id_{U} } > 0$ $\mathcal{H}^{2}_{d}$-almost everywhere, the Borel set $B_{0}$ in \Cref{lem:Condition(N)} has zero $\mathcal{H}^{2}_{d}$-measure. Since $\mu_{V}( \phi(B) ) = 0$ implies $\rho_{\id_{V}}\mu_{V}( \phi(B) ) = 0$, Condition ($N^{-1}$) of $\phi$ follows from \Cref{lem:Condition(N)}. In this case, the map $\phi$ satisfies Condition ($N$) if and only if $\mu_{V}( \phi(B_{0} ) ) = 0$ but this happens if and only if $\rho_{\id_{V}} = \chi_{V}$ $\mu_{V}$-almost everywhere.
\end{proof}
We pass to the proof of \Cref{lem:Condition(N)}. We start with the following lemma.
\begin{lemma}\label{identity:minimal:upper:gradient}
There exists a Borel set $B \subset U$ such that $\chi_{U \setminus B}$ is a minimal upper gradient of $\id_{U}$. Moreover, the following are equivalent for all Borel sets $\tilde{B} \subset U$: 
\begin{equation*}
	\modulus \Gamma^{+}_{ \tilde{B} } = 0, \, \quad
	\mu_{U}( \tilde{B} \setminus B ) = 0 \quad \text{ and } \quad
	\rho_{ \id_{U} }\mu_{U}( \tilde{B} ) = 0.
\end{equation*}
The weighted measure $\rho_{ \id_{U} }\mu_{U}$ is independent of the representative of the minimal upper gradient $\rho_{ \id_{U} }$ of $\id_{U}$.
\end{lemma}
\begin{proof}
Since $\id_{U} \in L^{2}_{loc}( U, \, U )$ and $\chi_{U} \in L^{2}_{loc}( U )$ is an upper gradient of $\id_{U}$, the minimal upper gradient $\rho_{\id_{U}} \in L^{2}_{loc}( U )$ exists. Fix a representative $\rho$ of $\rho_{\id_{U}}$. We prove that $\rho$ equals a characteristic function $\mu_{U}$-almost everywhere.

If $A = \left\{ \rho < 1 \right\}$, then almost any path cannot have positive length in $A$: For almost every path $\gamma$, the inequality $v_{ \gamma } \leq \rho v_{ \gamma }$ holds at $m_{1}$-almost every $t \in \mathrm{dom}( \gamma )$ by \eqref{eq:metric:speed:uppergrad}. This immediately implies that the family $\Gamma_{A}^{+}$ has zero modulus.

The minimality of $\rho$ implies that we must have $\rho \leq \chi_{U}$ $\mu_{U}$-almost everywhere. This means that the paths that have positive length in $A' = \left\{ \rho > 1 \right\}$ have zero modulus (property \ref{mod:5} of modulus). We conclude that $B_{1} = A \cup A'$ satisfies $\modulus \Gamma^{+}_{B_{1}} = 0$.

Suppose that $B$ is a Borel set such that $\modulus \Gamma^{+}_{B} = 0$. We claim that $\rho' = \rho \chi_{ U \setminus B }$ is a minimal upper gradient of $\id_{U}$. To that end, for almost every rectifiable path $\gamma \not\in \Gamma^{+}_{B}$, i.e., almost every rectifiable path $\gamma$, we have that
\begin{equation*}
	\int_{ \gamma }
		\rho'
	\D s
	=
	\int_{ \gamma }
		\chi_{ U \setminus B }
		\rho
	\D s
	=
	\int_{ \gamma }
		\rho
	\D s
	\geq
	d( \gamma(a), \, \gamma(b) );
\end{equation*}
the second equality follows from the fact that $\gamma$ has zero length in $B$. The inequality follows from the upper gradient inequality.

We have deduced that $\rho'$ is a weak upper gradient of $\id_{U}$. Since $\rho' \leq \rho$ everywhere, $\rho'$ is a representative of $\rho_{\id_{U}}$. This means that the equality $\rho' = \rho$ holds $\mu_{U}$-almost everywhere. This happens if and only if the set $B \setminus \left\{ \rho = 0 \right\}$ has zero $\mu_{U}$-measure.

In particular, if $B = B_{1}$, then $\rho' = \chi_{ U \setminus B_{1} }$ equals $\rho$ $\mu_{U}$-almost everywhere, hence $\chi_{ U \setminus B_{1} }$ is a representative of $\rho_{\id_{U}}$. We proved the first part of the claim.

Next we prove the equivalence of the three statements. Since minimal upper gradients are unique up to $\mu_{U}$-measure zero, the measure $\nu = \rho_{ \id_{U} }\mu_{U}$ is independent of the chosen representative of the minimal upper gradient. We still consider the representative $\rho' = \chi_{ U \setminus B_{1} }$ constructed above. We deduce that $\nu( \widetilde{B} ) = 0$ if and only if $\mu_{U}( \widetilde{B} \setminus B_{1} ) = 0$.

We proved during the construction of $\rho'$ above that if $\modulus \Gamma^{+}_{ \widetilde{B} } = 0$, then
\begin{equation*}
	0
	=
	\mu_{U}\left(
		\widetilde{B} \setminus \left\{ \rho' = 0 \right\}
	\right)
	=
	\mu_{U}\left(
		\widetilde{B} \setminus B_{1}
	\right)
	=
	\nu( \widetilde{B} ).
\end{equation*}
Conversely, if $\nu( \widetilde{B} ) = 0$, then $\widetilde{B} \setminus B_{1}$ has zero $\mu_{U}$-measure and thus $\modulus \Gamma^{+}_{ \widetilde{B} \setminus B_{1} } = 0$ (property \ref{mod:5} of modulus). Then the subadditivity and monotonicity of modulus imply that $\modulus \Gamma^{+}_{ \widetilde{B} }
	\leq
	\modulus \Gamma^{+}_{ \widetilde{B} \cap B_{1} }
	\leq
	\modulus \Gamma^{+}_{ B_{1} }
	=
	0$. We have deduced that $\nu( \widetilde{B} ) = 0$ if and only if $\modulus \Gamma^{+}_{ \widetilde{B} } = 0$.
\end{proof}
\begin{proof}[Proof of \Cref{lem:Condition(N)}]
We first show the existence of $B_{0}$. \Cref{identity:minimal:upper:gradient} implies that $\rho_{ \id_{U} } = \chi_{ U \setminus B_{1} }$ and $\rho_{ \id_{V} } = \chi_{ V \setminus B_{2} }$ for some Borel sets $B_{1} \subset U$ and $B_{2} \subset V$. Then the same statement shows that for Borel sets $B \subset U$,
\begin{align}
	\label{eq:condition(n).proof.1}
	&\mu_{U}( B \setminus B_{1} ) = 0
	\quad\text{if and only if}\quad
	\modulus \Gamma^{+}_{ B } = 0 \text{ and}
	\\
	\label{eq:condition(n).proof.2}
	&\modulus \Gamma^{+}_{ \phi(B) } = 0
	\quad\text{if and only if}\quad
	\mu_{V}( \phi(B) \setminus B_{2} ) = 0.
\end{align}
Since the $\phi$- and $\phi^{-1}$-singular paths have zero modulus, \Cref{lem:QC:differentequivalent} and the quasiconformality of $\phi$ establish that
\begin{equation}
	\label{eq:condition(n).proof.3}
	\modulus \Gamma^{+}_{ B } = 0
	\quad\text{if and only if}\quad
	\modulus \Gamma^{+}_{ \phi(B) } = 0.
\end{equation}
We deduce from these statements that the symmetric difference of $B_{1}$ and $\phi^{-1}( B_{2} )$ has zero $\mu_{U}$-measure (similar result holds on the image side). Then $B_{0} = B_{1} \cup \phi^{-1}( B_{2} )$ is the desired Borel set $B_{0}$. The rest of the claim follows from \Cref{identity:minimal:upper:gradient} and the equivalences \eqref{eq:condition(n).proof.1}, \eqref{eq:condition(n).proof.2} and \eqref{eq:condition(n).proof.3}.
\end{proof}
\begin{corollary}\label{cor:minimal:upper:gradient}
Suppose that $\rho_{ \phi }$ and $\rho_{ \phi^{-1} }$ are representatives of minimal upper gradients of $\phi$ and $\phi^{-1}$, respectively. If $B = \rho_{ \phi }^{-1}( 0 ) \cup ( \rho_{ \phi^{-1} } \circ \phi )^{-1}( 0 )$, then we can set $B_{0} = B$ in \Cref{lem:Condition(N)}.

In particular, the restriction of $\phi$ to the complement of $B$ satisfies Conditions ($N$) and ($N^{-1}$).
\end{corollary}
\begin{proof}
Fix Borel representatives $\rho_{ \id_{U} } = \chi_{ U \setminus B' }$ of the minimal upper gradient of $\id_{U}$ and $\rho_{\phi}$ of $\phi$. By arguing as in the proof of \Cref{lem:QC:differentequivalent}, we deduce that
\begin{equation*}
	\rho_{1} = \left( \rho_{ \phi^{-1} } \circ \phi \right) \rho_{\phi}
\end{equation*}
is a weak upper gradient of $\id_{U}$. The function $\rho_{1}$ is an element of $L^{2}_{loc}( U )$ as a consequence of \eqref{eq:analytic:outer}. The minimality of $\rho_{ \id_{U} }$ implies that $\rho_{ \id_{U} } \leq \rho_{1}$ holds $\mu_{U}$-almost everywhere. From this fact and as $\rho_{1}$ is zero $\mu_{U}$-almost everywhere in $B$, it follows that
\begin{equation*}
	\mu_{U}
	\left(
		B \setminus B'
	\right)
	=
	0.
\end{equation*}
The modulus of $\Gamma^{+}_{B'}$ equals zero as a consequence of \Cref{identity:minimal:upper:gradient}. Then by arguing as in the proof of \Cref{identity:minimal:upper:gradient}, the map
\begin{equation*}
	\rho = \rho_{ \phi } \chi_{ U \setminus B' } \in L^{2}_{loc}( U )
\end{equation*}
is a minimal upper gradient of $\phi$. This implies that
\begin{equation*}
	\mu_{U}
	\left(
		B' \setminus B
	\right)
	\leq
	\mu_{U}\left(
		B' \setminus \rho_{\phi}^{-1}( 0 )
	\right)
	=
	0.
\end{equation*}
We have proved that the symmetric difference $B' \Delta B$ has zero $\mu_{U}$-measure and hence $\chi_{ U \setminus B }$ is a minimal upper gradient of $\id_{U}$.

By applying \Cref{lem:Condition(N)}, we can assume without loss of generality that $\chi_{ V \setminus \phi( B' ) }$ is a minimal upper gradient of $\id_{V}$. Then by arguing as above, we see that $\chi_{ V \setminus \phi(B) }$ is a minimal upper gradient of $\id_{V}$. We have proved that we can set $B_{0} = B$ in \Cref{lem:Condition(N)}.
\end{proof}
\subsection{Pointwise dilatations}\label{sec:pointwise:dilatation}
For this section, we fix a quasiconformal map $\phi \colon U \rightarrow V$.
\begin{definition}\label{def:pointwise:dilatation}
The \emph{pointwise outer dilatation $K_{O}( \phi )$} and \emph{pointwise inner dilatation $K_{I}( \phi )$} of $\phi$ are
\begin{equation*}
%	\label{eq:QC:outer:dilatation}
	K_{O}( \phi )( x )
	=
	\frac{ \rho_{\phi}^{2}( x ) }{ J_{\phi}(x) }
	\quad
	\text{and}
	\quad
	K_{I}( \phi )( x )
	=
	\rho_{\phi^{-1}}^{2} \circ \phi(x) J_{\phi}(x),
\end{equation*}
respectively, which are defined $\rho_{\id_{U}} \mu_{U}$-almost everywhere. The \emph{pointwise maximal dilatation $K( \phi )$} of $\phi$ is the maximum of the pointwise dilatations $K_{O}( \phi )$ and $K_{I}( \phi )$.
\end{definition}
\begin{remark}
With a slight abuse of notation, we denote by $K_{O}( \phi )$ both the outer dilatation of $\phi$ and the pointwise analog. Similarly for $K_{I}( \phi )$ and $K( \phi )$.
\end{remark}
\begin{lemma}[Composition Laws]\label{lem:outer:inner:characterization}
Suppose that $\phi \colon U \rightarrow V$ is quasiconformal. Then the pointwise dilatations satisfy
\begin{align}
	\label{eq:outer:dilatation:inverse}
	K_{O}( \phi )
	&=
	K_{I}( \phi^{-1} ) \circ \phi \quad \text{and} \quad
%	\\
%	\label{eq:inner:dilatation:inverse}
	K_{I}( \phi )
	=
	K_{O}( \phi^{-1} ) \circ \phi
\end{align}
$\rho_{\id_{U}} \mu_{U}$-almost everywhere. If $\psi \colon W \rightarrow U$ is quasiconformal, then the pointwise dilatations satisfy
\begin{align}
	\label{eq:outer:dilatation:composition}
	K_{O}( \phi \circ \psi )
	&\leq
	K_{O}( \phi ) \circ \psi \,
	K_{O}( \psi ) \text{ and}
	\\
	\label{eq:inner:dilatation:composition}
	K_{I}( \phi \circ \psi )
	&\leq
	K_{I}( \phi ) \circ \psi \,
	K_{I}( \psi )
\end{align}
$\rho_{\id_{W}} \mu_{W}$-almost everywhere.
\end{lemma}
\begin{proof}
\Cref{lem:Condition(N)} shows that the maps
\begin{align*}
	\psi
	&\colon
	( W, \, \rho_{\id_{W}} \mu_{W} )
	\rightarrow
	( U, \, \rho_{\id_{U}} \mu_{U} ),
	\\
	\phi
	&\colon
	( U, \, \rho_{\id_{U}} \mu_{U} )
	\rightarrow
	( V, \, \rho_{\id_{V}} \mu_{V} ), \text{ and}
	\\
	\phi \circ \psi
	&\colon
	( W, \, \rho_{\id_{W}} \mu_{W} )
	\rightarrow
	( V, \, \rho_{\id_{V}} \mu_{V} )
\end{align*}
satisfy Conditions ($N$) and ($N^{-1}$). This implies that
\begin{equation*}
	J_{ \phi \circ \psi }
	=
	J_{ \phi } \circ \psi
	J_{ \psi }
\end{equation*}
holds $\rho_{\id_{W}} \mu_{W}$-almost everywhere; here the Jacobians are taken with respect to the unweighted measures not with respect to the weighted measures $\rho_{\id_{W}} \mu_{W}$ et cetera. Nevertheless, we still have the above identity. Then \eqref{eq:outer:dilatation:composition} is equivalent to showing that
\begin{equation*}
	\rho_{ \phi \circ \psi }
	\leq
	\rho_{ \phi }\circ \psi
	\rho_{ \psi }
\end{equation*}
$\rho_{\id_{W}}\mu_{W}$-almost everywhere. But this follows from the fact that $\rho_{\phi}\circ\psi \rho_{\psi} \in L^{2}_{loc}( W )$ (recall \eqref{eq:analytic:outer}) is a weak upper gradient of $\phi \circ \psi$ (which follows by arguing as in the proof of \Cref{cor:minimal:upper:gradient}). Thus we obtain the inequality \eqref{eq:outer:dilatation:composition}.

We only prove the first identity in \eqref{eq:outer:dilatation:inverse} since the latter one is proved in a similar way. The equality $J_{\phi} = \frac{ 1 }{ J_{\phi^{-1}} } \circ \phi$ holds $\rho_{ \id_{U} } \mu_{U}$-almost everywhere, hence
\begin{equation*}
	K_{I}( \phi )
	=
	\rho_{\phi^{-1}}^{2} \circ \phi
	J_{\phi}
	=
	\frac{  \rho_{\phi^{-1}}^{2} \circ \phi }{ J_{\phi^{-1}} \circ \phi }
	=
	K_{O}( \phi^{-1} ) \circ \phi
\end{equation*}
holds $\rho_{\id_{U}} \mu_{U}$-almost everywhere.

Conditions ($N$) and ($N^{-1}$) reduce the inequality \eqref{eq:inner:dilatation:composition} to showing that
\begin{equation*}
	\rho_{ ( \phi \circ \psi )^{-1} }
	\leq
	\rho_{ \psi^{-1} } \circ \phi^{-1}
	\rho_{ \phi^{-1} }
\end{equation*}
$\rho_{\id_{V}}\mu_{V}$-almost everywhere. This follows by arguing as above as in the proof of \eqref{eq:outer:dilatation:composition}.
\end{proof}
\begin{corollary}[Global Dilatations]\label{cor:global:inner:and:outer:dilatation}
The outer and inner dilatations of a quasiconformal map $\phi \colon U \rightarrow V$ satisfy
\begin{align}
	\label{eq:outer:dilatation:global}
	K_{O}( \phi )
	&=
	\esssup\left\{
		K_{O}( \phi )(x)
		\mid
		x \in ( U, \, \rho_{\id_{U}} \mu_{U} )
	\right\} \text{ and}
	\\
	\label{eq:inner:dilatation:global}
	K_{I}( \phi )
	&=
	\esssup\left\{
		K_{I}( \phi )(x)
		\mid
		x \in ( U, \, \rho_{\id_{U}} \mu_{U} )
	\right\}.
\end{align}
\end{corollary}
\begin{proof}
\Cref{lem:Condition(N)} and the latter identity in \eqref{eq:outer:dilatation:inverse} imply that \eqref{eq:inner:dilatation:global} follows from the identity \eqref{eq:outer:dilatation:global}. Therefore we only prove \eqref{eq:outer:dilatation:global}.

Fix representatives of $\rho_{\phi}$ and $J_{\phi}$. Then \Cref{lem:Condition(N)} and \Cref{cor:minimal:upper:gradient} show that for every constant $C > 0$,
\begin{equation*}
	\mu_{U}\left( \left\{ \rho_{ \phi }^{2} > C J_{\phi} \right\} \right)
	=
	\left( \rho_{ \id_{U} } \mu_{U} \right)\left( \left\{ \rho_{ \phi }^{2} > C J_{\phi} \right\} \right).
\end{equation*}
This implies that the $L^{\infty}$-norms of $\frac{ \rho_{ \phi }^{2} }{ J_{\phi} }$ with respect to $\mu_{U}$ and with respect to $\rho_{ \id_{U} }\mu_{U}$ coincide; let $C'$ denote that norm. A simple application of \Cref{thm:QC:differentequivalent} for the cases $C > C'$ and $C' > C$ imply that $C \geq K_{O}( \phi )$ and $K_{O}( \phi ) \geq C$, respectively. We conclude that $C' = K_{O}( \phi )$.
\end{proof}
\subsection{Approximate metric differentials}\label{sec:approximate}
%\begin{definition}\label{def:approximate:metric}
Let $U \subset \mathbb{R}^{2}$ be open and $V$ a metric space. Then a map $\phi \colon U \rightarrow V$ is \emph{approximately metrically differentiable} at $x \in U$ if there exists a seminorm $s$ on $\mathbb{R}^{2}$ such that
\begin{equation}
	\label{eq:approximate:metric:differential}
	\aplim_{ y \rightarrow x }
	\frac{ d( \phi(x), \, \phi(y) ) - s[ x - y ] }{ \norm{ x - y }_{2} }
	=
	0.
\end{equation}
We recall the definition of approximate limits. In the following, the measure $m_{2}$ refers to the Lebesgue $2$-measure on $\mathbb{R}^{2}$.
\begin{definition}
If $f \colon U \rightarrow \left[-\infty, \, \infty\right]$ is an $m_{2}$-measurable function, then $\aplim_{ y \rightarrow x } f(y) = 0$ if there exists a Lebesgue measurable set $L \subset U$ that has $x \in L$ as a Lebesgue density point and
\begin{equation}
	\label{eq:aplim}
	\lim_{ L \ni y \rightarrow x }
		f(y)
	=
	0.
\end{equation}
\end{definition}
Equation \eqref{eq:aplim} implies that if the approximate limit exists, it is unique. If the approximate limit \eqref{eq:approximate:metric:differential} exists, the seminorm $s$ is the \emph{approximate metric differential} of $\phi$ at $x$ and is denoted by $\apmd[\phi]{x}$.

The map $\phi$ has an \emph{approximate metric differential $\apmd{\phi}$} if \eqref{eq:approximate:metric:differential} exists $m_{2}$-almost everywhere. In this case, we say that the approximate metric differential $\apmd{\phi}$ exists.
%\end{definition}

The space of seminorms on $\mathbb{R}^{2}$ is endowed with the topology of locally uniform convergence. The topology is equivalent to the convergence in the Banach space $\mathcal{C}( \mathbb{S}^{1}, \, \mathbb{R} )$, where the space $\mathcal{C}( \mathbb{S}^{1}, \, \mathbb{R} )$ uses the supremum norm. If $s$ is a seminorm on $\mathbb{R}^{2}$, the \emph{Jacobian} of $s$ is
\begin{equation}
	\label{eq:jacobian:seminorm}
	J_{2}( s )
	=
	\frac{ \pi }{ m_{2}\left( \left\{ s \leq 1 \right\} \right) }.
\end{equation}
The above definition is from \cite{kirchheim,ambrosio2000}. The Jacobian $J_{2}( s )$ is non-zero only if $s$ is a norm. The map $s \mapsto J_{2}( s )$ depends continuously on the seminorm $s$. The following proposition is essentially Proposition 4.3 of \cite{metricderivative}.
\begin{proposition}\label{newtonian:seminorm}
Let $\phi \colon U \rightarrow V$ be quasiconformal, where $U \subset \mathbb{R}^{2}$ is open and $V$ is an open subset of a metric surface. Then the approximate metric differential $\apmd{\phi}$ exists and is an element of $L^{2}_{loc}( U, \, \mathcal{C}( \mathbb{S}^{1}, \, \mathbb{R} ) )$.

Moreover, there exists a Borel set $N \subset U$ of zero $m_{2}$-measure and countably many pairwise disjoint compact sets $K_{i} \subset U$ partitioning $U \setminus N$ in such a way that for every $i = 1, \, 2, \dots$ the following properties hold:
\begin{enumerate}[label=(\alph*)]
\item\label{newtonian:seminorm:existence} the restriction of $\phi$ to $K_{i}$ is Lipschitz, the approximate metric differential $\apmd[\phi]{x}$ exists at every $x \in K_{i}$, and the restriction of $\apmd[\phi]{x}$ to $K_{i}$ is continuous;
\item\label{newtonian:seminorm:localproperties} for every $\epsilon > 0$ there exists a radius $r = r( i, \, \epsilon ) > 0$ such that for every $x \in K_{i}$ and every $x + v, x + w \in K_{i} \cap \overline{B}( x, \, r )$,
\begin{equation}
	\label{eq:approximate:metric:diff}
	\abs{
		d( \phi(x+v), \, \phi(x+w) )
		-
		\apmd[\phi]{x}[v-w]
	}
	\leq
	\epsilon \norm{ v - w }_{2}.
\end{equation}
\end{enumerate}
Additionally, the following properties hold:
\begin{enumerate}[label=(\alph*)]\setcounter{enumi}{2}
\item\label{newtonian:seminorm:changeofvariables} for every Borel measurable function $\rho \colon U \rightarrow \left[0, \, \infty\right]$,
\begin{equation}
	\label{eq:change:of:variables}
	\int_{ U }
		\rho(x)
		J_{2}( \apmd{\phi} )(x)
	\D m_{2}(x)
	=
	\int_{ V \setminus \phi(N) }
		\rho \circ \phi^{-1}(y)
	\D \mathcal{H}^{2}_{d}(y);
\end{equation}
\item\label{newtonian:seminorm:absolutecontinuity} the restriction of $\phi$ to $U \setminus N$ satisfies Condition $(N)$.
\end{enumerate}
\end{proposition}
\begin{remark}
There are metric surfaces where $\mathcal{H}^{2}_{d}( \phi(N) ) > 0$; see \cite[Proposition 17.1]{uniformization}. That is the main reason why we developed the theory of \Cref{sec:qc,sec:pointwise:dilatation}.
\end{remark}
\begin{proof}
The map $\phi$ is an element of the Newtonian--Sobolev space $N^{1, \, 2}_{loc}( U, \, V )$ as a consequence of \Cref{thm:QC:differentequivalent}. We cover $U$ by a countable union of open quadrilaterals $Q_{i} = \left( a_{i}, \, b_{i} \right) \times \left( c_{i}, \, d_{i} \right)$ in such a way that the closure $\overline{Q_{i}}$ is contained in $U$ for every integer $i = 1, \, 2, \, \dots$. Then for each $i$, the restriction of $\phi$ to $Q_{i}$ is an element of $N^{1, \, 2}( Q_{i}, \, X )$. Recall that quadrilaterals support a $(1, \, 1)$-Poincaré inequality and the Lebesgue measure is Ahlfors $2$-regular on $Q_{i}$. Therefore Corollary 10.2.9 of \cite{metricsobolev} shows that the Newtonian--Sobolev space $N^{1, \, 2}( Q_{i}, \, X )$ coincides with the Haj\l{}asz--Sobolev space $W^{1, \, 2}( Q_{i}, \, X )$.

Fix $i$. Recall that $f \in W^{1, \, 2}( Q_{i}, \, X )$ if $f \in L^{2}( Q_{i}, \, X )$ and there exists $N \subset Q_{i}$ with $m_{2}( N ) = 0$ and $g \in L^{2}( Q_{i} )$ such that
\begin{equation}
	\label{eq:hajlasz}
	d( f(x), \, f(y) )
	\leq
	\abs{ x - y }
	( g(x) + g(y) )
\end{equation}
for all $x$ and $y$ in $Q_{i} \setminus N$.

We apply the machinery developed in \cite{metricderivative}. We note that the definition of $W^{1, \, 2}( Q_{i}, \, X )$ stated here (essentially) coincides with the definition of Sobolev spaces used in \cite{metricderivative}; see \cite[Proposition 3.2]{metricderivative}. Therefore Proposition 4.3 of \cite{metricderivative} shows that $\apmd{\phi}(z)$ exists $m_{2}$-almost everywhere in $Q_{i}$; the fact that the map
\begin{equation*}
	z \mapsto \apmd{\phi}( z )
\end{equation*}
is an element of $L^{2}( Q_{i}, \, \mathcal{C}( \mathbb{S}^{1}, \, \mathbb{R} ) )$; and that there exist countably many pairwise disjoint compact sets $K^{i}_{j}$ which cover $m_{2}$-almost all of $Q_{i}$ in such a way that properties \ref{newtonian:seminorm:existence} and \ref{newtonian:seminorm:localproperties} hold for the compact sets $K^{i}_{j}$ (except the fact that the restriction of $\apmd{\phi}$ to each $K^{i}_{j}$ is continuous).

Since $\mathcal{C}( \mathbb{S}^{1}, \, \mathbb{R} )$ is separable, we apply Lusin's theorem \cite[Theorem 2.3.5]{federer} which implies the following: up to omitting a set of $m_{2}$-measure zero and passing to smaller compact sets, the restriction of $\apmd{\phi}$ to each $K^{i}_{j}$ is continuous. Therefore by passing to smaller compact sets, properties \ref{newtonian:seminorm:existence} and \ref{newtonian:seminorm:localproperties} hold.

By inspecting the proof of Proposition 4.3 of \cite{metricderivative}, we see that the restriction of $\phi$ to each $K^{i}_{j}$ is also Lipschitz. Thus the area formula \eqref{eq:change:of:variables} follows from the area formula for Lipschitz maps \cite{kirchheim,ambrosio2000}. Property \ref{newtonian:seminorm:absolutecontinuity} follows from property \ref{newtonian:seminorm:existence}.

We proved the claim for an arbitrary integer $i$. We repeat the argument for all integers. Then we consider the countable collection of compact sets $\left\{ K_{j}^{i} \right\}_{ i, \, j  \in \mathbb{N} }$. By omitting a set of $m_{2}$-measure zero and passing to smaller sets, we can assume that the compact sets are pairwise disjoint and that properties \ref{newtonian:seminorm:existence} to \ref{newtonian:seminorm:absolutecontinuity} hold.
\end{proof}
The following result is a key result in this paper. The proof is originally from Section 14 of \cite{uniformization}.
\begin{corollary}[Norm Field]\label{norm:almost:everywhere}
Let $\phi \colon \mathbb{R}^{2} \supset  U \rightarrow V$ be a quasiconformal homeomorphism. Then the approximate metric differential $\apmd{ \phi }$ is a norm $m_{2}$-almost everywhere.
\end{corollary}
\begin{proof}
Let $N$ denote the set of $m_{2}$-measure zero obtained from \Cref{newtonian:seminorm}. Let $\widetilde{N}$ denote the union of $N$ and the points for which $J_{2}( \apmd{\phi} ) = 0$. Then $\mathcal{H}^{2}_{d}( \phi( \widetilde{N} \setminus N ) ) = 0$ by the change of variables formula \eqref{eq:change:of:variables}. Since $\phi$ satisfies Condition ($N^{-1}$) (\Cref{lem:Condition(N):inverse}), the set $\widetilde{N} \setminus N$ has zero $m_{2}$-measure. Since $m_{2}( N ) = 0$, we deduce that $\widetilde{N}$ has zero $m_{2}$-measure. As a consequence, $J_{2}( \apmd{\phi} ) > 0$ $m_{2}$-almost everywhere, i.e., $\apmd{ \phi }$ is a norm $m_{2}$-almost everywhere.
\end{proof}
Suppose that $\phi \in N^{1, \, 2}_{loc}( U, \, V )$ is quasiconformal. Let $N \subset U$ and $\left\{ K_{i} \right\}_{ i = 1 }^{ \infty }$ be as in \Cref{newtonian:seminorm}.

Consider an open set $W \subset \mathbb{R}^{2}$ and suppose that $\psi \colon W \rightarrow U$ is a quasiconformal homeomorphism. Let $N_{0}$ denote the union of $\psi^{-1}( N )$, the collection of points where $\psi$ fails to be classically differentiable, and the collection of points of $\psi^{-1}( K_{i} )$ that are not Lebesgue density points of $\psi^{-1}( K_{i} )$ for $i = 1, \, 2, \, \dots$. Recall that $\psi$ satisfies Conditions ($N$) and ($N^{-1}$) and is classically differentiable $m_{2}$-almost everywhere \cite[Section 3.3]{astala}. As a consequence, the set $N_{0}$ has zero $m_{2}$-measure.
\begin{proposition}[Chain Rule]\label{prop:chainrule:QC}
The map $\phi \circ \psi$ is quasiconformal and $m_{2}( N_{0} ) = 0$. Moreover, for every $y \in W \setminus N_{0}$, the approximate metric differential $\apmd{ \phi \circ \psi }( y )$ exists and
\begin{equation}
	\label{eq:chainrule:approximate}
	\apmd{ \phi \circ \psi }(y)
	=
	\apmd{ \phi } \circ D\psi(y).
\end{equation}
\end{proposition}
\begin{remark}\label{diffeo}
If $\psi$ is a diffeomorphism, it can be proved directly using \eqref{eq:aplim} that $\apmd{ \phi \circ \psi }( y ) = \apmd{ \phi } \circ D\psi(y)$ whenever $\apmd{ \phi \circ \psi }( y )$ or $\apmd{ \phi }( \psi(y) )$ exists.
\end{remark}
\begin{proof}[Proof of \Cref{prop:chainrule:QC}]
Suppose that $y \in W \setminus N_{0}$ and define $s = \apmd{\phi} \circ D\psi(y)$. The point $y$ is a density point of some $L = \psi^{-1}( K_{i} )$ by definition of $N_{0}$. Fix $\epsilon > 0$. Since $\apmd{\phi}( \psi(y) )$ is a seminorm and as $\psi$ is classically differentiable at $y$,
\begin{equation}
	\label{eq:triangle:inequality}
	\abs{
		\apmd{\phi}( \psi(y) )[ \psi( y + v ) - \psi( y ) ]
		-
		s[v]
	}
	\leq
	o( \norm{ v }_{2} ).
\end{equation}
Since $\psi$ is classically differentiable at $y$,
\begin{equation}
	\label{eq:classic:differentiability}
	\frac{ \norm{ \psi(y+v) - \psi(y) }_{2} }{ \norm{ v }_{2} }
	\leq
	\norm{ D\psi }( y )
	+
	\frac{ o( \norm{ v }_{2} ) }{ \norm{ v }_{2} },
\end{equation}
where $\norm{ D\psi }$ is the operator norm of $D\psi$. Fix $\epsilon > 0$. By combining the inequalities \eqref{eq:triangle:inequality} and \eqref{eq:classic:differentiability} with Property \ref{newtonian:seminorm:localproperties} of \Cref{newtonian:seminorm}, we find that
\begin{gather*}
	\limsup_{ L \ni z \rightarrow y }
	\abs{
		\frac{ d( \phi \circ \psi(z), \, \phi \circ \psi(y) ) - s[ z - y ] }{ \norm{ z - y }_{2} }
	}
	\leq
	\epsilon\norm{ D\psi }( y ).
\end{gather*}
Since this holds for any $\epsilon > 0$, we obtain the claim.
\end{proof}
\section{Uniformization}\label{sec:uniformization}
From this points onwards, we only consider metric surfaces. We introduce the definition and some results about isothermal parametrizations in \Cref{sec:minimal:chart}. We prove them in \Cref{sec:isothermal:properties}. As an intermediate step, we study more general quasiconformal parametrizations in \Cref{sec:uniformization:charts}.

The isothermal parametrizations are used to construct a conformal atlas on a locally reciprocal surface, and the atlas is used to construct a Riemannian distance on the surface; this is the topic of \Cref{sec:uniformization:structure}.
\subsection{Isothermal parametrizations}\label{sec:minimal:chart}
Suppose that $U$ is an open subset of $\mathbb{R}^{2}$, $V$ is a metric surface, and $\phi \colon U \rightarrow V$ is a quasiconformal homeomorphism.
\begin{definition}\label{def:minimal:chart}
The quasiconformal map $\phi$ is an \emph{isothermal parametrization of $V$} if for every quasiconformal map $\psi \colon W \rightarrow U$, where $W \subset \mathbb{R}^{2}$, the pointwise dilatations satisfy
\begin{equation}
	\label{eq:minimal}
	(K_{O}( \phi )K_{I}( \phi ))(x)
	\leq
	\left[ K_{O}( \phi \circ \psi )K_{I}( \phi \circ \psi ) \right] \circ \psi^{-1}(x)
\end{equation}
$m_{2}$-almost everywhere. Such a $\phi$ is said to be \emph{isothermal}.
\end{definition}
The pointwise dilatations were defined in \Cref{def:pointwise:dilatation}. Since the minimal upper gradient $\rho_{ \id_{U} }$ of $\id_{U}$ equals $\chi_{U}$ $m_{2}$-almost everywhere, the quantities in \eqref{eq:minimal} are well-defined $m_{2}$-almost everywhere in $U$.

\Cref{prop:minimal:TFAE} provides several equivalent definitions of isothermal parametrizations. Isothermal parametrizations appear in the work of Romney in \cite{upper:modulus:bound}, though they are not explicitly defined; see \Cref{prop:minimal:TFAE} and the proof of Theorem 1.2 of \cite{upper:modulus:bound}. The main contribution of \Cref{sec:minimal:chart} (and \Cref{sec:isothermal:properties}) is \Cref{thm:minimal:chart:uniqueness} (and \Cref{prop:minimal:TFAE}).
%
%Generalizations of classical isothermal coordinates appear in the work of \cite{compactqs} under the name of \emph{quasi-isothermal} coordinates. Wildrick and Geyer construct charts that are quasisymmetries instead of quasiconformal maps and the charts satisfy a property analogous to \Cref{thm:minimal:chart:uniqueness}.
%
%Generalized isothermal coordinates also appear in the theory of Alexandrov surfaces; see \cite{isothermalalexandrov} and references therein.

\Cref{isothermal:compatible} proves that if $V$ is an open subset of a Riemannian surface, then \Cref{def:minimal:chart} coincides with the classical definition of isothermal parametrizations, i.e., in that case every isothermal parametrization is a diffeomorphism and conformal.

The following proposition establishes the existence of isothermal parametrizations.
\begin{restatable}{proposition}{thmminimalchartexistence}\label{thm:minimal:chart:existence}
Let $V$ be a metric surface and $\phi \colon \mathbb{R}^{2} \supset U \rightarrow V$ be quasiconformal. Then there exists a $\frac{ 4 }{ \pi }K_{I}( \phi )$-quasiconformal map $\psi \colon \mathbb{R}^{2} \supset W \rightarrow U$ such that $\phi \circ \psi$ is isothermal.
\end{restatable}
%\begin{theorem}\label{thm:minimal:chart:existence}
%There exists a $\frac{ \pi }{ 2 }K( \phi )$-quasiconformal map $\psi \colon W \rightarrow U$ so that $\phi \circ \psi$ is isothermal.
%\end{theorem}
\Cref{thm:minimal:chart:uniqueness} shows that isothermal parametrizations are unique up to conformal maps. Before stating the theorem we recall some terminology. When we say that a homeomorphism $\psi$ between planar domains is \emph{conformal in the classical sense}, we mean that $\psi$ is a bijective holomorphic or antiholomorphic map.

Recall that a homeomorphism between planar domains is conformal in the geometric sense if and only if it is conformal in the classical sense (Weyl's lemma -- Lemma A.6.10 of \cite{astala}).
\begin{restatable}{theorem}{thmminimalchartuniqueness}\label{thm:minimal:chart:uniqueness}
Suppose that $\phi$ is isothermal and that $\widetilde{\psi} \colon \widetilde{W} \rightarrow V$ is a quasiconformal map for $\widetilde{W} \subset \mathbb{R}^{2}$. Then $\widetilde{\psi}$ is isothermal if and only if $\psi = \phi^{-1} \circ \widetilde{\psi} \colon \widetilde{W} \rightarrow U$ is conformal in the classical sense.
\end{restatable}
%\begin{theorem}\label{thm:minimal:chart:uniqueness}
%Suppose that $\phi$ is isothermal and that $\widetilde{\psi} \colon \widetilde{W} \rightarrow V$ is a quasiconformal map for $\widetilde{W} \subset \mathbb{R}^{2}$. Then $\widetilde{\psi}$ is isothermal if and only if $\psi = \phi^{-1} \circ \widetilde{\psi} \colon \widetilde{W} \rightarrow U$ is conformal.
%\end{theorem}
\begin{definition}\label{def:BM:distance}
Let $M$ and $N$ be norms on $\mathbb{R}^{2}$. Then $GL_{2}[ M, \, N ]$ is the collection of invertible linear maps $S \colon ( \mathbb{R}^{2}, \, M ) \rightarrow ( \mathbb{R}^{2}, \, N )$.

The \emph{multiplicative Banach--Mazur distance from $M$ to $N$} is
\begin{equation}
	\label{eq:BM-distance:definition}
	\rho( M, \, N )
	=
	\inf\left\{
		\left[
			K_{O}( S )
			K_{I}( S )
		\right]^{1/2}
		\mid
		S \in GL_{2}[ M, \, N ]
	\right\}.
\end{equation}
\end{definition}
Recall from \eqref{def:pointwise:dilatation} that the distortion $\left[ K_{O}( S )K_{I}( S ) \right]^{1/2}$ equals the product of the minimal upper gradients of $S$ and $S^{-1}$. The minimal upper gradients coincide with the operator norms of $S$ and $S^{-1}$, respectively. This means that \eqref{eq:BM-distance:definition} agrees with the standard definition of the (multiplicative) Banach--Mazur distance (see \cite[Chapter 37]{enoughsymmetries}). Moreover, by John's theorem, $\rho( M, \, N ) \leq 2$ and by a compactness argument, the infimum in \eqref{eq:BM-distance:definition} is a minimum.
\begin{definition}\label{def:BM:distance:min}
Let $M$ and $N$ be norms on $\mathbb{R}^{2}$. Then $S \in GL_{2}[ M, \, N ]$ is a \emph{Banach--Mazur minimizer from $M$ to $N$} if $S$ attains the infimum in \eqref{eq:BM-distance:definition}. If the domain and codomain of the linear map $S$ is clear from the context, we say that $S$ is a \emph{Banach--Mazur minimizer}.
\end{definition}
\begin{restatable}{proposition}{propdilatationbounds}\label{prop:dilatationbounds}
Suppose that $\phi \colon U \rightarrow V$ is quasiconformal. Then $\phi$ is isothermal if and only if 
\begin{align}
	\label{eq:pointwise:outer:inner:optimal}
	\left[
		K_{O}( \phi )
		K_{I}( \phi )
	\right](x)
	=
	\rho^{2}( \norm{ \cdot }_{2}, \, \apmd{\phi} )(x)
\end{align}
$m_{2}$-almost everywhere.
\end{restatable}
%\begin{proposition}\label{prop:dilatationbounds}
%Suppose that $\phi$ is isothermal. Then at $m_{2}$-almost every $x \in U$,
%\begin{align}
%	\label{eq:pointwise:outer:inner:optimal}
%	\left[
%		K_{O}( \phi )
%		K_{I}( \phi )
%	\right]^{1/2}(x)
%	=
%	\rho( \norm{ \cdot }_{2}, \, F_{\phi} )( x ).
%\end{align}
%\end{proposition}
Our method of proving \eqref{eq:pointwise:outer:inner:optimal} actually shows that
\begin{equation*}
	D\id
	\colon
	( TU, \, \norm{ \cdot }_{2} )
	\rightarrow
	( TU, \, \apmd{\phi} )
\end{equation*}
is a Banach--Mazur minimizer $m_{2}$-almost everywhere. Here $TU$ is the tangent bundle of $U$, where the notation $( TU, \, \apmd{\phi} )$ is used to emphasize that the approximate metric differential $\apmd{\phi}$ of $\phi$ depends on the basepoint. We show the following.
\begin{restatable}{corollary}{corminimalcharts}\label{cor:minimal:charts}
Suppose that $\phi \colon U \rightarrow V$ is isothermal. Then $\phi$ satisfies
\begin{align}
	\label{eq:pointwise:dilatation:outer:optimal}
	\frac{ 2 }{ \pi }
	\rho^{2}( \norm{ \cdot }_{2}, \, \apmd{\phi} )( x )
	&\leq
	K_{O}( \phi )( x )
	\leq
	\frac{ 4 }{ \pi } \text{ and }
	\\
	\label{eq:pointwise:dilatation:inner:optimal}
	\frac{ \pi }{ 4 }
	\rho^{2}(  \norm{ \cdot }_{2}, \, \apmd{\phi} )( x )
	&\leq
	K_{I}( \phi )( x )
	\leq
	\frac{ \pi }{ 2 }
\end{align}
for $m_{2}$-almost every $x \in U$. In particular, $K_{O}( \phi )(x) \leq \frac{ 4 }{ \pi }$, $K_{I}( \phi )(x) \leq \frac{ \pi }{ 2 }$ and $K_{O}( \phi )(x)K_{I}( \phi )(x) \leq 2$ for $m_{2}$-almost every $x \in U$.
\end{restatable}
%\begin{corollary}\label{cor:minimal:charts}
%Suppose that $\phi$ is isothermal. Then $\phi$ satisfies
%\begin{align}
%	\label{eq:pointwise:dilatation:outer:optimal}
%	\frac{ 2 }{ \pi }
%	\rho^{2}( \norm{ \cdot }_{2}, \, F_{\phi} )( x )
%	&\leq
%	K_{O}( \phi )( x )
%	\leq
%	\frac{ 4 }{ \pi } \text{ and}
%	\\
%	\label{eq:pointwise:dilatation:inner:optimal}
%	\frac{ \pi }{ 4 }
%	\rho^{2}(  \norm{ \cdot }_{2}, \, F_{\phi} )( x )
%	&\leq
%	K_{I}( \phi )( x )
%	\leq
%	\frac{ \pi }{ 2 }.
%\end{align}
%\end{corollary}
\Cref{cor:minimal:charts} is a reformulation of \cite[Theorem 1.2]{upper:modulus:bound}. \Cref{thm:minimal:chart:uniqueness} is essential for the rest of the paper.

Before proving the results stated in \Cref{sec:minimal:chart}, we need to understand the pointwise outer and inner dilatations appearing in \Cref{def:minimal:chart}. This is the content of \Cref{sec:uniformization:charts}.
\subsection{Analytic properties of parametrizations}\label{sec:uniformization:charts}
Let $\apmd{\phi}$ denote the approximate metric differential of a quasiconformal map $\phi \colon U \rightarrow V$ (as defined in \Cref{sec:approximate}). We study the map
\begin{equation}
	\label{eq:identitymap:differential}
	D\id
	\colon
	( TU, \, \norm{ \cdot }_{2} )
	\rightarrow
	( TU, \, \apmd{\phi} )
\end{equation}
in this section. For each basepoint $x \in U \subset \mathbb{R}^{2}$, we identify $T_{x}U$ with $\mathbb{R}^{2}$ in the standard way.

Since $\phi$ may fail to satisfy Condition ($N$), we have to be very careful with measurability considerations. For this reason, let $\widetilde{N}$ be a Borel set of $m_{2}$-measure zero for which $\apmd{\phi}\chi_{ U \setminus \widetilde{N} }$ is a norm everywhere in $U \setminus \widetilde{N}$ and such that
\begin{equation*}
	x
	\mapsto
	\apmd{\phi}\chi_{ U \setminus \widetilde{N} }(x)
\end{equation*}
is Borel measurable. \Cref{newtonian:seminorm} Property \ref{newtonian:seminorm:existence} and \Cref{norm:almost:everywhere} imply that such a set $\widetilde{N}$ exists. Let $N$ denote the Borel set of zero $m_{2}$-measure appearing in \Cref{newtonian:seminorm}. Then by \Cref{newtonian:seminorm} Property \ref{newtonian:seminorm:absolutecontinuity}, the set $\phi( \widetilde{N} ) \setminus \phi( N )$ has $\mathcal{H}^{2}_{d}$-measure zero. We assume without loss of generality that $\widetilde{N} \supset N$.

A crucial point about the set $\widetilde{N}$ is that $\phi$ satisfies Conditions ($N$) and ($N^{-1}$) in the complement of $\widetilde{N}$ (\Cref{newtonian:seminorm} and \Cref{lem:Condition(N):inverse}, respectively). Fix such a Borel set $\widetilde{N}$.

Let $I_{\phi} \colon U \rightarrow \left[0, \, \infty\right]$ denote the operator norm of \eqref{eq:identitymap:differential} in $U \setminus \widetilde{N}$ and zero in $\widetilde{N}$. Let
\begin{equation*}
	I_{\phi^{-1}} \colon V \rightarrow \left[0, \, \infty\right]
\end{equation*}
be defined similarly: In the complement of $\widetilde{N}$, pointwise $I_{\phi^{-1}} \circ \phi$ equals the operator norm of the inverse of \eqref{eq:identitymap:differential}. In $\widetilde{N}$, $I_{\phi^{-1}} \circ \phi$ equals zero.

The functions $I_{\phi}$ and $I_{\phi^{-1}}$ are Borel measurable, since the operator norm depends continuously on the norms on the domain and codomain of $D\id$. The Jacobian $J_{2}( D\id )$ of $D\id$ from \eqref{eq:identitymap:differential} is (defined to be) the Jacobian $J_{2}( \apmd{\phi} )$ of the norm $\apmd{\phi}$ (recall \eqref{eq:jacobian:seminorm}). The dilatations of $D\id$ from \eqref{eq:identitymap:differential} are defined pointwise.
\begin{proposition}\label{prop:phi:characterization}
The Borel functions $I_{\phi}$ and $I_{\phi^{-1}}$ are minimal upper gradients of $\phi$ and $\phi^{-1}$, respectively. The Jacobian of $\phi$ equals the Jacobian of $D\id$ from \eqref{eq:identitymap:differential} $m_{2}$-almost everywhere.

Moreover, the pointwise outer and inner dilatations of $\phi$ satisfy
\begin{align*}
	K_{O}( \phi )
	&=
	K_{O}( D\id )
	\quad
	\text{ and }
	\quad
	K_{I}( \phi )
	=
	K_{I}( D\id )
\end{align*}
$m_{2}$-almost everywhere in $U$, where the domain and codomain of $D\id$ are determined by \eqref{eq:identitymap:differential}.
\end{proposition}
We express the minimal upper gradients of $\phi$ and $\phi^{-1}$ in \Cref{prop:upper:gradient} and the corresponding Jacobians in \Cref{change:of:variables:formula} using the norm field $\apmd{\phi}$. \Cref{prop:phi:characterization} follows immediately from these two results, \Cref{cor:prop:upper:gradient}, and from \Cref{lem:outer:inner:characterization}.
\subsubsection{Minimal upper gradients}\label{sec:uniformization:charts:gradient}
Recall that $\widetilde{N}$ is a Borel set of zero $m_{2}$-measure and the approximate metric differential $\apmd{\phi}$ is well-defined and a norm everywhere in the complement of $\widetilde{N}$.

Let $\Gamma_{0}$ denote the family of paths that are $\phi$-singular (recall \Cref{def:QC:singularset;goodpaths}) or paths that have positive length in $\widetilde{N}$. \Cref{lem:QC:differentequivalent,identity:minimal:upper:gradient} imply that $\Gamma_{0}$ and $\phi( \Gamma_{0} )$ have zero modulus.

By construction, the absolutely continuous paths of $AC_{+}(U) \setminus \Gamma_{0}$ are $\phi$-good (recall \Cref{def:QC:singularset;goodpaths}) and they have zero length in the Borel set $\widetilde{N}$. The key property for the paths in $AC_{+}( U ) \setminus \Gamma_{0}$ is that the variation measures $v_{ \gamma } \cdot m_{1}$ and $v_{ \phi \circ \gamma } \cdot m_{1}$ are absolutely continous with respect to one another. Therefore the path $\phi \circ \gamma$ has zero length in $\phi( \widetilde{N} )$ and it makes sense to talk about the differential $D\gamma$ of $\gamma$ for $v_{ \phi \circ \gamma} \cdot m_{1}$-almost every point. This means that the expression $\apmd{ \phi } \circ D\gamma$ makes sense $v_{ \phi \circ \gamma} \cdot m_{1}$-almost everywhere.
\begin{proposition}\label{prop:upper:gradient}
The maps $I_{\phi}$ and $I_{\phi^{-1}}$ are minimal upper gradients of $\phi$ and $\phi^{-1}$, respectively.

Additionally, if $\gamma \in AC( U ) \setminus \Gamma_{0}$, the triple $( \phi, \,  I_{\phi}, \, \gamma )$ satisfies the upper gradient inequality. Moreover, for any such $\gamma$, $\phi \circ \gamma \in AC( V ) \setminus \phi( \Gamma_{0} )$ and the triple $( \phi^{-1}, \, I_{\phi^{-1}}, \, \phi \circ \gamma )$ satisfies the upper gradient inequality.
\end{proposition}
\begin{corollary}\label{cor:prop:upper:gradient}
The characteristic functions $\chi_{ U \setminus \widetilde{N} }$ and $\chi_{ V \setminus \phi( \widetilde{N} ) }$ are minimal upper gradients of $\id_{U}$ and $\id_{V}$, respectively.
\end{corollary}
\begin{proof}
This is an immediate consequence of \Cref{cor:minimal:upper:gradient} and the fact that $I_{\phi}$ and $I_{\phi^{-1}}$ are minimal upper gradients of $\phi$ and $\phi^{-1}$, respectively.
\end{proof}
We prove \Cref{prop:upper:gradient} using the following lemma.
\begin{lemma}\label{metric:speed:almosteverycurve}
For every path $\gamma \in AC(U) \setminus \Gamma_{0}$, in particular, almost every absolutely continuous path, the metric speed $v_{ \phi \circ \gamma }$ of $\phi \circ \gamma$ satisfies 
\begin{equation}
	\label{eq:metric:speed}
	v_{ \phi \circ \gamma }
	=
	\apmd{\phi} \circ D\gamma
	\in
	L^{1}( \mathrm{dom}( \gamma ) )
\end{equation}
$m_{1}$-almost everywhere in $\mathrm{dom}( \gamma )$.
\end{lemma}
\begin{proof}
If $\gamma$ is a constant path, then there is nothing to prove. Therefore it suffices to consider paths $\gamma \in AC_{+}( U ) \setminus \Gamma_{0}$. Since every path $\gamma \in AC_{+}( U ) \setminus \Gamma_{0}$ is $\phi$-good, it suffices to show \eqref{eq:metric:speed} when $\gamma$ has unit speed with respect to the Euclidean norm.

By \Cref{newtonian:seminorm} Property \ref{newtonian:seminorm:localproperties} and the construction of $\widetilde{N}$, there exists a partition $\left\{ B_{i} \right\}_{ i = 1 }^{ \infty }$ of $U \setminus \widetilde{N}$ by Borel sets with the following property: For every $\epsilon > 0$, there exists a radius $r = r( i, \, \epsilon ) > 0$ such that for every $x \in B_{i}$ and every $x + v \in B_{i} \cap \overline{B}( x, \, r )$,
\begin{equation}
	\label{eq:proof:approximate:metric:diff}
	\abs{
		d( \phi(x+v), \, \phi(x) )
		-
		\apmd{\phi}(x)[v]
	}
	\leq
	\epsilon \norm{ v }_{2}.
\end{equation}
Recall that $\apmd{\phi}(x)$ is a norm for every $x \in U \setminus \widetilde{N}$.

The set $\gamma^{-1}( \widetilde{N} )$ has zero $m_{1}$-measure since the path $\gamma$ has zero length in $\widetilde{N}$ and $\gamma$ has unit speed. As a consequence, $m_{1}$-almost every $t$ in the domain of $\gamma$ is a density point of some $\gamma^{-1}( B_{i} )$.

The absolute continuity of $\gamma$ yields that for $m_{1}$-almost every $t \in \mathrm{dom}( \gamma )$, the differential $D\gamma(t)$ exists and
\begin{align}
	\label{eq:taylor:series}
	\gamma( s )
	&=
	\gamma( t )
	+
	D\gamma(t)( s - t )
	+
	o( \abs{s-t} ).
\end{align}
Also, the path $\phi \circ \gamma$ is absolutely continuous hence at $m_{1}$-almost every $t \in \mathrm{dom}( \gamma )$, the metric speed $v_{ \phi \circ \gamma }(t)$ exists and
\begin{align}
	\label{eq:metric:speed:good:point}
	d( \phi(\gamma(s)), \, \phi(\gamma(t)) )
	&=
	v_{ \phi \circ \gamma }(t) \abs{ s - t }
	+
	o( \abs{s-t} ).
\end{align}
In conclusion, $m_{1}$-almost every $t \in \mathrm{dom}( \gamma )$ is a Lebesgue density point of some $\gamma^{-1}( B_{i} )$ for which \eqref{eq:taylor:series} and \eqref{eq:metric:speed:good:point} hold. Fix such a $t$. By the continuity of $\gamma$, the Lebesgue density of $\gamma^{-1}( B_{i} )$ at $t$, and \eqref{eq:proof:approximate:metric:diff}, for every positive integer $n$ there exists $t_{n} \in \gamma^{-1}( B_{i} )$ with $0 < \abs{ t_{n} - t } \leq 2^{-n}$ and
\begin{gather}
	\label{eq:approximate:metric:diff:alongpath}
	\abs{
		d( \phi( \gamma(t_{n}), \, \phi( \gamma(t) ) )
		-
		\apmd{\phi}(\gamma(t))[ \gamma(t_{n}) - \gamma(t) ]
	}
	\\ \notag
	\leq \,
	2^{-n} \norm{ \gamma(t_{n}) - \gamma(t) }_{2}.
\end{gather}
By combining \eqref{eq:taylor:series}, \eqref{eq:metric:speed:good:point}, \eqref{eq:approximate:metric:diff:alongpath}, and using the fact that $\apmd{\phi}( \gamma(t) )$ is a norm, we obtain that
\begin{equation}
	\abs{
		v_{ \phi \circ \gamma }(t)
		-
		\apmd{\phi} \circ D\gamma(t)
	}
	\leq
	\frac{ o( \abs{ t_{n} - t } ) }{ \abs{ t_{n} - t } }
	+
	2^{-n}\norm{ D\gamma(t) }_{2}
\end{equation}
for every $n$. The identity \eqref{eq:metric:speed} at $t$ follows by passing to the limit $n \rightarrow \infty$. Since $m_{1}$-almost every $t$ in $\mathrm{dom}( \gamma )$ is of this form, the proof is complete.
\end{proof}
\begin{proof}[Proof of \Cref{prop:upper:gradient}]
\Cref{metric:speed:almosteverycurve} and the definition of the operator norm of linear maps yield that for every $\gamma \in AC( U ) \setminus \Gamma_{0}$,
\begin{align}
	\label{eq:phi:upper}
	v_{ \phi \circ \gamma }
	&\leq
	(I_{\phi} \circ \gamma)
	v_{ \gamma }
%	\quad
	\text{ and }
%	\\
%	\label{eq:phi:inverse:upper}
	v_{ \gamma }
	\leq
	I_{ \phi^{-1} } \circ ( \phi \circ \gamma )
	v_{ \phi \circ \gamma }
\end{align}
for $m_{1}$-almost every $t \in \mathrm{dom}( \gamma )$; here we use the fact that $\gamma$ has zero length in $\widetilde{N}$. Integrating \eqref{eq:phi:upper} over the domain of $\gamma$ implies the claimed upper gradient inequalities. Since $\Gamma_{0}$ has zero modulus in $U$, the map $I_{\phi}$ is a weak upper gradient of $\phi$. Similar conclusions hold for the path family $\phi( \Gamma_{0} )$ in $V$ and the maps $I_{\phi^{-1}}$ and $\phi^{-1}$.

We fix representatives of the minimal upper gradients $\rho_{\phi}$ and $\rho_{\phi^{-1}}$ for the rest of the proof.

Part (1): We show that $\rho_{\phi} \geq I_{\phi}$ $m_{2}$-almost everywhere. This immediately yields that $\rho_{\phi} = I_{\phi}$ $m_{2}$-almost everywhere by the minimality of $\rho_{\phi}$.

Proof of Part (1): Let $\Gamma_{1}$ denote the family of absolutely continuous paths $\gamma \in AC_{+}( U ) \setminus \Gamma_{0}$ for which
\begin{equation*}
	m_{1}\left(
		\left\{
			v_{ \phi \circ \gamma }
			>
			\left( \rho_{ \phi } \circ \gamma \right)
			v_{ \gamma }
		\right\}
	\right)
	>
	0.
\end{equation*}
Then the path family $\Gamma = \Gamma_{0} \cup \Gamma_{1}$ has zero modulus.

In the following, we use the exponential map $\mathbb{R} \ni \theta \mapsto \Exp{ i \theta } \in \mathbb{S}^{1}$. Fix $\theta \in \mathbb{R}$, $x \in U$ and a small $r > 0$. Consider the square
\begin{equation*}
	Q_{r}
	=
	\left\{
		x
		+
		\Exp{ i \theta } t
		+
		i \Exp{ i \theta } s
		\mid
		t, \, s \in \left[ -r, \, r \right]
	\right\}.
\end{equation*}
For $-r \leq s \leq r$, consider the arcs
\begin{equation*}
	\left[-r, \, r\right]
	\ni
	t
	\mapsto
	\gamma_{s}( t )
	=
	x
	+
	\Exp{ i \theta } t
	+
	i \Exp{ i \theta } s.
\end{equation*}
Observe that $D\gamma_{s}(t) = \Exp{ i \theta }$ and for $m_{1}$-almost every $-r \leq s \leq r$, the path $\gamma_{s}$ is in the complement of $\Gamma$. Therefore for $m_{1}$-almost every $-r \leq s \leq r$ for $m_{1}$-almost every $-r \leq t \leq r$, we have that
\begin{align}
	\label{eq:proof:upper:gradient:phi}
	\apmd{\phi}( \gamma_{s}(t) )[ \Exp{ i \theta } ]
	&=
	v_{ \phi \circ \gamma_{s} }(t)
%	\\ \notag
	\leq
	\rho_{\phi} \circ \gamma_{s}(t)
	v_{ \gamma_{s} }(t)
	\\ \notag
	&=
	\rho_{\phi} \circ \gamma_{s}(t).
\end{align}
Then Fubini's theorem implies that for $m_{2}$-almost every $y \in Q_{r}$,
\begin{equation}
	\label{eq:the:end:result}
	\apmd{\phi}( y )\left[ \Exp{ i \theta } \right]
	\leq
	\rho_{\phi}( y ).
\end{equation}
Consider a countable dense subset $D \subset \mathbb{R}$. Let $G$ denote the intersection of the Lebesgue points of $\rho_{\phi} \in L^{2}_{loc}( U )$ and the functions 
\begin{equation*}
	\apmd{\phi}\left[ \Exp{ i \theta } \right] \in L^{2}_{loc}( U )
\end{equation*}
as $\theta \in D$ varies. The complement of $G$ has zero $m_{2}$-measure. For every $( x, \, \theta ) \in G \times D$, we have that
\begin{align}
	\label{eq:density}
	\apmd{\phi}(x)\left[ \Exp{ i \theta } \right]
	&=
	\lim_{ r \rightarrow 0^{+} }
	\aint{ Q_{r} }
		\apmd{\phi}( y )\left[ \Exp{ i \theta } \right]
	\D m_{2}(y) \text{ and }
	\\
	\label{eq:density:phi}
	\rho_{\phi}(x)
	&=
	\lim_{ r \rightarrow 0^{+} }
	\aint{ Q_{r} }
		\rho_{\phi}(y)
	\D m_{2}(y).
\end{align}
The horizontal bar in the integral sign refers to the integral average over the domain of integration. We combine \eqref{eq:the:end:result} with the identities \eqref{eq:density} and \eqref{eq:density:phi} in order to deduce that for every $( x, \, \theta ) \in G \times D$,
\begin{equation}
	\label{eq:essential:estimate}
	\apmd{\phi}( x )\left[ \Exp{ i \theta } \right]
	\leq
	\rho_{ \phi }( x ).
\end{equation}
Fix $x \in G \setminus \widetilde{N}$. Since \eqref{eq:essential:estimate} holds independently of $\theta \in D$, the $\rho_{\phi}(x)$ must be bounded from below by the number
\begin{equation*}
	I_{\phi}(x)
	=
	\sup\left\{
		\apmd{\phi}(x)[ v ]
		\mid
		\norm{ v }_{2} = 1
	\right\}.
\end{equation*}
We have deduced that $I_{\phi} \leq \rho_{\phi}$ $m_{2}$-almost everywhere and hence Part (1) is proved.

Part (2): We show that $\rho_{ \phi^{-1} } \geq I_{\phi^{-1}}$ for $\mathcal{H}^{2}_{d}$-almost every point in $V$. Since $I_{ \phi^{-1} }$ equals zero in $\phi\left( \widetilde{N} \right)$, it is sufficient to show the claimed inequality in $V \setminus \phi( \widetilde{N} )$, or equivalently that
\begin{equation}
	\label{eq:proof:essential:estimate:upper:inverse}
	\rho_{ \phi^{-1} } \circ \phi
	\geq
	I_{ \phi^{-1} } \circ \phi
\end{equation}
for $m_{2}$-almost every point in $U \setminus \widetilde{N}$; this equivalence follows from the fact that $\phi$ satisfies Conditions ($N$) and ($N^{-1}$) in the complement of $\widetilde{N}$ (\Cref{newtonian:seminorm} and \Cref{lem:Condition(N):inverse}, respectively). Once again, after this is shown, we conclude that $I_{\phi^{-1}} = \rho_{\phi^{-1}}$ $\mathcal{H}^{2}_{d}$-almost everywhere in $V$ by the minimality of $\rho_{\phi^{-1}}$.

The definition of $I_{ \phi^{-1} }$ yields that for every $x \in U \setminus \widetilde{N}$, we have that
\begin{equation*}
	I_{ \phi^{-1} }( \phi(x) )
	=
	\frac{ 1 }{ \inf\left\{ \apmd{\phi}(x)[ v ] \mid \norm{ v }_{2} = 1 \right\} }.
\end{equation*}
Proof of Part (2): Fix $\theta \in \mathbb{R}$, $x \in U$, and a small $r > 0$. We define
\begin{equation*}
	Q_{r}
	=
	\left\{
		x
		+
		\Exp{ i \theta } t
		+
		i \Exp{ i \theta } s
		\mid
		t, \, s \in \left[ -r, \, r \right]
	\right\}.
\end{equation*}
Let $\gamma_{s}$ be as in Part (1). Then for $m_{1}$-almost every $-r \leq s \leq r$ for $m_{1}$-almost every $-r \leq t \leq r$, we have that
\begin{equation*}
	1
	=
	v_{ \gamma_{s} }(t)
	\leq
	\rho_{ \phi^{-1} }( \phi \circ \gamma_{s}(t) )
	\apmd{\phi}( \gamma_{s}(t) )[ \Exp{ i \theta } ].
\end{equation*}
We divide both sides with $\apmd{\phi}( \gamma_{s}(t) )[ \Exp{ i \theta } ]$ and multiply the inequality with the Jacobian $J_{2}( \apmd{\phi}( \gamma_{s}(t) ) )$. Then we apply Fubini's theorem and the change of variables formula \eqref{eq:change:of:variables}: we obtain that
\begin{align}
	\label{eq:proof:change:of:variables}
	\infty
	&>
	\aint{ \phi( Q_{r} ) }
		\rho_{ \phi^{-1} }
	\D \mathcal{H}^{2}_{d}
	\geq
	\aint{ Q_{r} }
		\rho_{ \phi^{-1} } \circ \phi
		J_{2}( \apmd{\phi} )
	\D m_{2}
	\\ \notag
	&\geq
	\aint{ Q_{r} }
		\frac{ 1 }{ \apmd{\phi}\left[ \Exp{ i \theta } \right] }
		J_{2}( \apmd{\phi} )
	\D m_{2}.
\end{align}
We deduce from \eqref{eq:proof:change:of:variables} that $\rho_{ \phi^{-1} } \circ \phi J_{2}( \apmd{\phi} )$ and $\frac{ 1 }{ \apmd{\phi}[ \Exp{ i \theta } ] } J_{2}( \apmd{\phi} )$ are elements of $L^{2}_{ loc }( U )$.

Fix a countable dense subset $D \subset \mathbb{R}$. By arguing as in Part (1), we deduce that for $m_{2}$-almost every $x \in U \setminus \widetilde{N}$, for every $\theta \in D \subset \mathbb{R}$,
\begin{equation}
	\label{eq:inverse:uppergradient:last:part}
	\rho_{ \phi^{-1} }( \phi(x) ) J_{2}( \apmd{ \phi }(x) )
	\geq
	\frac{ 1 }{ \apmd{ \phi }( x )[ \Exp{ i \theta } ] }
	J_{2}( \apmd{\phi}(x) ).
\end{equation}
Fix such an $x \in U \setminus \widetilde{N}$. Since $x \in U \setminus \widetilde{N}$, we can divide both sides of \eqref{eq:inverse:uppergradient:last:part} with the Jacobian $\infty > J_{2}( \apmd{\phi}(x) ) > 0$. We take the supremum over $\theta \in D$ and deduce that
\begin{equation*}
	\rho_{ \phi^{-1} }( \phi(x) )
	\geq
	\frac{ 1 }{ \inf\left\{ \apmd{\phi}(x)[ v ] \mid \norm{ v }_{2} = 1 \right\} }.
\end{equation*}
The right-hand side coincides with $I_{ \phi^{-1}}( \phi(x) )$ in $U \setminus \widetilde{N}$. The proof is complete.
\end{proof}
\subsubsection{Jacobians}\label{sec:uniformization:charts:jacobian}
Recall that $\phi \colon \mathbb{R}^{2} \supset U \rightarrow V$ is a quasiconformal map, where $V$ is a metric surface. We express the Jacobians $J_{\phi}$ and $J_{\phi^{-1}}$ in the complement of the Borel set $\widetilde{N} \subset U$ of $m_{2}$-measure zero.

The approximate metric differential of $\phi$ is denoted by $\apmd{\phi}$. The Borel set $\widetilde{N} \subset U$ has $m_{2}$-measure zero in such a way that the approximate metric differential $\apmd{\phi}$ exists and is a norm everywhere in $U \setminus \widetilde{N}$ (see the beginning of \Cref{sec:uniformization:charts}).

Recall that the Jacobian of $\apmd{\phi}$ is
\begin{equation*}
	J_{2}( \apmd{\phi} )
	=
	\frac{ \pi }{ m_{2}\left( \left\{ \apmd{\phi} \leq 1 \right\} \right) }
\end{equation*}
whenever $\apmd{\phi}$ is well-defined. At every $x \in U \setminus \widetilde{N}$, the number $J_{2}( \apmd{\phi} )(x)$ is the Jacobian of the map $D\id \colon ( TU, \, \norm{ \cdot }_{2} ) \rightarrow ( TU, \, \apmd{\phi} )$ from \eqref{eq:identitymap:differential}.
\begin{proposition}\label{change:of:variables:formula}
The Jacobian $J_{\phi}$ of the quasiconformal map $\phi$ coincides with $J_{2}( D\id )$ $m_{2}$-almost everywhere. Moreover,
\begin{equation}
	J_{ \phi^{-1} }
	=
	\frac{ 1 }{ J_{2}( D\id ) } \circ \phi^{-1}
	=
	J_{2}( (D\id)^{-1} ) \circ \phi^{-1}
\end{equation}
$\mathcal{H}^{2}_{d}$-almost everywhere in $V \setminus \phi( \widetilde{N} )$.
\end{proposition}
\begin{proof}
The change of variables formula \eqref{eq:change:of:variables} and Lebesgue differentiation theorem yield that $J_{\phi} = J_{2}( \apmd{\phi} )$ $m_{2}$-almost everywhere in $U \setminus \widetilde{N}$.

\Cref{lem:Condition(N)} and \Cref{cor:prop:upper:gradient} show that the set where the $\phi$-pullback measure $\phi^{*}\mathcal{H}^{2}_{d}$ and $m_{2}$ fail to be absolutely continuous with respect to one another is concentrated on the set $\widetilde{N}$. This implies that $\left( J_{\phi} \circ \phi^{-1} \right)
	J_{\phi^{-1}}
	=
	1$ for $\mathcal{H}^{2}_{d}$-almost every point in $V \setminus \phi( \widetilde{N} )$. The proof is complete.
\end{proof}
\subsection{Existence of isothermal parametrizations}\label{sec:isothermal:properties}
In this section, we study a quasiconformal map $\phi \colon \mathbb{R}^{2} \supset U \rightarrow V$, where $V$ is a subset of a metric surface $Y_{d}$. Let $\apmd{\phi}$ denote the approximate metric differential of $\phi$ obtained from \Cref{norm:almost:everywhere}. Following \cite{upper:modulus:bound}, we show that we can associate a Beltrami differential $\mu_{\phi}$ to $\apmd{\phi}$ and thus to $\phi$ in a natural way. Instead of defining $\mu_{\phi}$ using the John ellipses of the norms of $\apmd{\phi}$ as in \cite{uniformization}, the approach of \cite{upper:modulus:bound} uses the Banach--Mazur distance (\Cref{def:BM:distance}) and the associated Beltrami differential (see \Cref{beltrami:diff:norm}).
%
%We construct the isothermal parametrizations by considering a solution $\psi^{-1} \colon U \rightarrow W$ to the Beltrami equation $\mu_{f} = \mu_{\phi}$. Then the quasiconformal map $\phi \circ \psi$ is an isothermal parametrization of $V$.

Some of the results of this section are proved in \cite{upper:modulus:bound}. Our main contribution is \Cref{prop:minimal:TFAE} and its corollaries (excluding \Cref{cor:minimal:charts} which corresponds to Theorem 1.2 of \cite{upper:modulus:bound}).

We recall some notations. The group $O_{2}$ is the group of isometries of $\mathbb{R}^{2}$ and $\mathbb{R}_{+} \cdot O_{2}$ refers to the invertible linear maps $L = \lambda \cdot S$, where $\lambda > 0$ and $S \in O_{2}$. The group $SO_{2}$ are the elements of $O_{2}$ with determinant equal to $1$. The group $\mathbb{R}_{+} \cdot O_{2}$ is the conformal automorphism group of $\mathbb{R}^{2}$ and $\mathbb{R}_{+} \cdot SO_{2}$ the subgroup of $\mathbb{R}_{+} \cdot O_{2}$ whose elements have positive determinant.

The result we state next corresponds to Lemmas 2.1 and 2.2 of \cite{upper:modulus:bound}. Recall the definition of Banach--Mazur minimizer from \Cref{def:BM:distance:min}.
\begin{lemma}\label{pointwise:BM:properties}
Let $M$ be a norm on $\mathbb{R}^{2}$ and $L \colon \left( \mathbb{R}^{2}, \, M \right) \rightarrow \left( \mathbb{R}^{2}, \, \norm{ \cdot }_{2} \right)$ a Banach--Mazur minimizer. Then
\begin{align}
	\label{eq:BM:minimizer:outer}
	\frac{ \pi }{ 4 }
	\rho^{2}( M, \, \norm{ \cdot }_{2} )
	&\leq
	K_{O}\left( L \right)
	\leq
	\frac{ \pi }{ 2 } \text{ and}
	\\
	\label{eq:BM:minimizer:inner}
	\frac{ 2 }{ \pi }
	\rho^{2}( M, \, \norm{ \cdot }_{2} )
	&\leq
	K_{I}\left( L \right)
	\leq
	\frac{ 4 }{ \pi }.
\end{align}
Moreover, $L' \in GL_{2}[ M, \, \norm{ \cdot }_{2} ]$ is a Banach--Mazur minimizer if and only if $L' \circ L^{-1} \in \mathbb{R}_{+} \cdot O_{2}$.
\end{lemma}
\begin{proof}
The inequalities \eqref{eq:BM:minimizer:outer} and \eqref{eq:BM:minimizer:inner} are slight reformulations of Lemma 2.1 of \cite{upper:modulus:bound}. Let $L$ be the minimizer as in the claim. We claim that
\begin{equation*}
	L' \circ L^{-1} \in \mathbb{R}_{+} \cdot O_{2}
\end{equation*}
if and only if $K_{O}( L )K_{I}( L ) = K_{O}( L' )K_{I}( L' )$. If $L' \circ L^{-1} \in \mathbb{R}_{+} \cdot O_{2}$, then $K_{O}( L )K_{I}( L ) = K_{O}( L' )K_{I}( L' )$ since the group $\mathbb{R}_{+} \cdot O_{2}$ is the conformal automorphism group of $\mathbb{R}^{2}$ (recall the composition laws of dilatations from \Cref{lem:outer:inner:characterization}). The converse is harder to show and corresponds to Lemma 2.2 of \cite{upper:modulus:bound}.
\end{proof}
\begin{lemma}\label{beltrami:diff:norm}
Suppose that $M$ is a norm on $\mathbb{R}^{2}$. Then there exists a unique complex number $\mu_{M}$ in the Euclidean ball $\mathbb{D}$ such that
\begin{equation}
	T_{M}
	=
	\id
	+
	\mu_{M}
	\cdot
	\overline{\id}
	\colon
	\left( \mathbb{R}^{2}, \, M \right)
	\rightarrow
	\left( \mathbb{R}^{2}, \, \norm{ \cdot }_{2} \right)
\end{equation}
is a Banach--Mazur minimizer from $M$ to $\norm{ \cdot }_{2}$. Moreover, $\mu_{M}$ and $T_{M}$ depend continuously on the norm $M$.
\end{lemma}
\begin{proof}
Since the group $\mathbb{R}_{+} \cdot O_{2}$ contains the conjugate map $z \mapsto \overline{z}$, there are orientation-preserving Banach--Mazur minimizers $L \in GL_{2}[ M, \, \norm{ \cdot }_{2} ]$ (\Cref{pointwise:BM:properties}). Then for any orientation-preserving Banach--Mazur minimizer $L$, there exists $S \in \mathbb{R}_{+} \cdot SO_{2}$ such that $L' = S \circ L$ is of the form $T_{M}$. Any such $L'$ is also a Banach--Mazur minimizer by \Cref{pointwise:BM:properties} and $\mu_{L'} = \mu_{L}$. The existence and uniqueness of $T_{M}$ and $\mu_{M}$ follow immediately.

The continuity of $M \mapsto \mu_{M}$ and $M \mapsto T_{M}$ are equivalent, therefore it suffices to focus on the continuity of $M \mapsto T_{M}$. We only sketch the argument. If $\left( M_{n} \right)$ is a sequence of norms converging to a norm $M$, then every subsequence of $\left( T_{M_{n}} \right)$ has a convergent subsequence due to Arzelà--Ascoli theorem. It suffices to show that an arbitrary convergent subsequence $\left( T_{M_{n_{j}}} \right)_{j}$ converges to $T_{M}$. To that end, suppose that
\begin{equation*}
	T_{j} = T_{M_{n_{j}}} \rightarrow T
\end{equation*}
locally uniformly. Such a limit is of the form $T = \id + \mu \cdot \overline{\id}$ for some $\mu$ in the Euclidean unit ball.

Let $\rho_{j}$ denote the Banach--Mazur distance $\rho( M_{n_{j}}, \, \norm{ \cdot }_{2} )$ and $D$ the distortion of $T \in GL_{2}\left[ M, \, \norm{ \cdot }_{2} \right]$ (the distortion is the product of the operator norms of $T$ and $T^{-1}$). Since the norms $M_{n_{j}}$ converge to $M$ and the minimizers $T_{j}$ to $T$ locally uniformly, the distortions $\rho_{j}$ of $T_{j}$ converge to $D$ (since the $\rho_{j}$ are the distortions of $T_{j}$). Let $\rho$ denote the Banach--Mazur distance $\rho( M, \, \norm{ \cdot }_{2} )$. The distances $\rho_{j}$ must converge to $\rho$ since the $M_{n_{j}}$ converge to $M$ locally uniformly. In conclusion, $D = \rho$. We deduce that $T$ is a Banach--Mazur minimizer from $M$ to $\norm{ \cdot }_{2}$, and by uniqueness of $\mu_{M}$, $T = T_{M}$.
\end{proof}
\begin{definition}\label{Beltrami:differential}
The Beltrami differential of $\phi$ is $\mu_{ \apmd{\phi} }$ and it is denoted by $\mu_{\phi}$.
\end{definition}
The measurability of $x \mapsto \mu_{\phi}(x)$ follows from the measurability of $x \mapsto \apmd{\phi}(x)$ and the continuity of $M \mapsto \mu_{M}$ shown in \Cref{beltrami:diff:norm}. We observe in the proof of \Cref{prop:minimal:TFAE} that the $L^{\infty}$-norm of $\mu_{\phi}$ is bounded from above by a constant $C = C( K_{I}( \phi ) ) < 1$. The following theorem is the main result of this section.
\begin{theorem}\label{prop:minimal:TFAE}
Let $\psi \colon \mathbb{R}^{2} \supset W \rightarrow U$ be a quasiconformal map, possibly orientation-reversing. Then the following are equivalent:
\begin{enumerate}[label=(\alph*)]
\item\label{prop:minimal:TFAE:a} Either $\psi^{-1}$ or $\overline{ \psi^{-1} }$ is an orientation-preserving solution of the Beltrami equation $\mu_{f} = \mu_{\phi}$;
\item\label{prop:minimal:TFAE:b} The map $D( \psi^{-1} ) \colon ( TU, \, \apmd{\phi} ) \rightarrow ( TW, \, \norm{ \cdot }_{2} )$ is a Banach--Mazur minimizer pointwise $m_{2}$-almost everywhere;
\item\label{prop:minimal:TFAE:c} The map $D \id_{W} \colon ( TW, \, \apmd{\phi \circ \psi} ) \rightarrow ( TW, \, \norm{ \cdot }_{2} )$ is a Banach--Mazur minimizer pointwise $m_{2}$-almost everywhere;
\item\label{prop:minimal:TFAE:d} The pointwise dilatations satisfy the equality
\begin{equation}
	\label{eq:outerinnerproduct:optimal}
	K_{O}( \phi \circ \psi )K_{I}( \phi \circ \psi )
	=
	\rho^{2}( \norm{ \cdot }_{2}, \, \apmd{\phi \circ \psi} )
\end{equation}
$m_{2}$-almost everywhere in $W$;
\item\label{prop:minimal:TFAE:e} The composition $\phi \circ \psi$ is isothermal;
\item\label{prop:minimal:TFAE:f} The equality $\mu_{ \phi \circ \psi } = 0$ holds $m_{2}$-almost everywhere.
\end{enumerate}
\end{theorem}
\begin{proof}[Proof of \Cref{prop:minimal:TFAE}]
Consider the equivalence of Claims \ref{prop:minimal:TFAE:a} and \ref{prop:minimal:TFAE:b}. \Cref{pointwise:BM:properties} yields that
\begin{equation*}
	D( \psi^{-1} ) \colon ( TU, \, \apmd{\phi} ) \rightarrow ( TW, \, \norm{ \cdot }_{2} )
\end{equation*}
is a Banach--Mazur minimizer $m_{2}$-almost everywhere if and only if there exists a measurable map $x \mapsto S(x) \in \mathbb{R}_{+} \cdot O_{2}$ such that $D( \psi^{-1} ) = S \circ T_{ \apmd{\phi} }$ pointwise $m_{2}$-almost everywhere. The map $\psi$ is orientation-preserving if and only if $S$ is orientation-preserving $m_{2}$-almost everywhere. In that case $\mu_{ \psi^{-1} } = \mu_{ \phi }$ holds $m_{2}$-almost everywhere. Otherwise $\overline{ \psi^{-1} }$ is orientation-preserving and $\mu_{ \overline{ \psi^{-1} } } = \mu_{ \phi }$ holds $m_{2}$-almost everywhere. We conclude that Claims \ref{prop:minimal:TFAE:a} and \ref{prop:minimal:TFAE:b} are equivalent.

The equivalence of Claims \ref{prop:minimal:TFAE:b} and \ref{prop:minimal:TFAE:c} follows from the fact that $D\psi \colon ( TW, \, \apmd{\phi \circ \psi} ) \rightarrow ( TU, \, \apmd{\phi} )$ is an isometry pointwise $m_{2}$-almost everywhere due to the chain rule \Cref{prop:chainrule:QC}.

\Cref{prop:phi:characterization} yields that if
\begin{equation*}
	D\id_{W} \colon ( TW, \, \norm{ \cdot }_{2} ) \rightarrow ( TW, \, \apmd{ \phi \circ \psi } ),
\end{equation*}
then the pointwise dilatations satisfy
\begin{align*}
	K_{O}( \phi \circ \psi )
	&=
	K_{O}( D\id_{W}  )
	\quad \text{ and}
	\quad
	K_{I}( \phi \circ \psi )
	=
	K_{I}( D\id_{W} )
\end{align*}
$m_{2}$-almost everywhere. The equivalence of \ref{prop:minimal:TFAE:c} and \ref{prop:minimal:TFAE:d} is immediate. We also deduce from these identities that $m_{2}$-almost everywhere, by the definition of the Banach--Mazur distance \eqref{eq:BM-distance:definition}, that the pointwise dilatations satisfy
\begin{align}
	\label{eq:pointwise:maximal:dilatation}
	K_{O}( \phi \circ \psi ) K_{I}( \phi \circ \psi )
	\geq
	\rho^{2}( \norm{ \cdot }_{2}, \, \apmd{\phi \circ \psi} ).
\end{align}
Also if $\psi_{1}$ and $\psi_{2}$ are two maps for which $\phi \circ \psi_{1}$ and $\phi \circ \psi_{2}$ are isothermal (\Cref{def:minimal:chart}), then $m_{2}$-almost everywhere
\begin{align}
	\label{eq:two:minimal:existence}
	K_{O}( \phi \circ \psi_{1} )K_{ I }( \phi \circ \psi_{1} )
	=
	\left[
		K_{O}( \phi \circ \psi_{2} )K_{I}( \phi \circ \psi_{2} )
	\right]
	\circ 
	( \psi_{2}^{-1} \circ \psi_{1} ).
\end{align}
By applying \eqref{eq:pointwise:maximal:dilatation} and \eqref{eq:two:minimal:existence}, the equivalence of Claims \ref{prop:minimal:TFAE:a} to \ref{prop:minimal:TFAE:d} and Claim \ref{prop:minimal:TFAE:e} follows if it can be shown that there exists a quasiconformal map $\psi$ such that the equality in \eqref{eq:pointwise:maximal:dilatation} holds $m_{2}$-almost everywhere. By Claim \ref{prop:minimal:TFAE:a}, it suffices to solve the Beltrami equation $\mu_{f} = \mu_{\phi}$ induced by $\phi$.

Suppose that we know that the $L^{\infty}$-norm of $\mu_{\phi}$ is bounded from above by some constant $C < 1$ (depending only on the inner dilatation of $\phi$). Then we extend $\mu_{\phi}$ as zero to the Euclidean plane and let $f$ be the normalized solution to the corresponding Beltrami equation. The existence of $f$ is guaranteed by the measurable Riemann mapping theorem; see for example \cite{astala}. The restriction of $f^{-1}$ to the appropriate domain is the desired map $\psi$.

Thus we only need to find such a bound $C$. To that end, consider the maps
\begin{align*}
	L
	&=
	T_{\apmd{\phi}} \colon ( TU, \, \apmd{\phi} ) \rightarrow ( TU, \, \norm{ \cdot }_{2} ),
	\\
	L_{1}
	&=
	T_{ \apmd{\phi} } \colon ( TU, \,\norm{ \cdot }_{2} ) \rightarrow ( TU, \, \norm{ \cdot }_{2} ), \text{ and}
	\\
	L_{2}
	&=
	D\id \colon ( TU, \, \norm{ \cdot }_{2} ) \rightarrow ( TU, \, \apmd{\phi} ).
\end{align*}
By construction of $T_{\apmd{\phi}}$ and $\mu_{\phi}$,
\begin{equation*}
	\norm{ \mu_{\phi} }_{ L^{\infty}( U ) }
	=
	\esssup_{ x \in U }
		\frac{ K_{I}( L_{1} )(x) - 1 }{ K_{I}( L_{1} )(x) + 1 }.
\end{equation*}
Therefore it suffices to bound the pointwise dilatation $K_{I}( L_{1} )$ from above. Also $m_{2}$-almost everywhere
\begin{align*}
%	K_{O}( \phi )
%	&=
%	K_{O}( L_2 )
%	\quad \text{ and}
%	\quad
	K_{I}( \phi )
	=
	K_{I}( L_2 ).
\end{align*}
The composition laws for dilatations (\Cref{lem:outer:inner:characterization}) yield that
\begin{align*}
	K_{I}( L_{1} )
	=
	K_{I}( L \circ L_2 )
	\leq
	K_{I}( L ) K_{I}( \phi )
\end{align*}
$m_{2}$-almost everywhere. \Cref{pointwise:BM:properties} shows that $K_{I}( L ) \leq \frac{ 4 }{ \pi }$ at $m_{2}$-almost every point. We conclude that $m_{2}$-almost everywhere
\begin{align*}
	K_{I}( L_{1} )
	\leq
	\frac{ 4 }{ \pi }
	K_{I}( \phi ).
\end{align*}
By \Cref{cor:global:inner:and:outer:dilatation}, the pointwise dilatations $K_{I}( \phi )$ are less than the (global) inner dilatation of $\phi$ $m_{2}$-almost everywhere, therefore
\begin{equation*}
	C
	=
	\frac{ \frac{ 4 }{ \pi } K_{I}( \phi ) - 1 }{ \frac{ 4 }{ \pi } K_{I}( \phi ) + 1 }
\end{equation*}
is a bound of the desired form.

The equivalence of \ref{prop:minimal:TFAE:f} to the rest of the claims follows from the observation that $\phi \circ \psi$ is isothermal if and only if $( \phi \circ \psi ) \circ \id_{W}$ is isothermal.
\end{proof}
Recall \Cref{thm:minimal:chart:existence}. \thmminimalchartexistence*
\begin{proof}[Proof of \Cref{thm:minimal:chart:existence}]
Such a $\psi$ was constructed during the proof of \Cref{prop:minimal:TFAE}.
\end{proof}
Recall \Cref{thm:minimal:chart:uniqueness}. \thmminimalchartuniqueness*
\begin{proof}[Proof of \Cref{thm:minimal:chart:uniqueness}]
Since $\phi$ is isothermal, we know that $\mu_{ \phi } = 0$ by \Cref{prop:minimal:TFAE}. Moreover, since $\widetilde{\psi} = \phi \circ \psi$, we know that $\widetilde{\psi}$ is isothermal if and only if $\psi^{-1}$ or $\overline{ \psi^{-1} }$ solves the Beltrami equation $\mu_{f} = \mu_{\phi} = 0$ (\Cref{prop:minimal:TFAE}) if and only if $\psi^{-1}$ is conformal in the classical sense (Weyl's lemma -- Lemma A.6.10 of \cite{astala}).
\end{proof}
Recall \Cref{prop:dilatationbounds}. \propdilatationbounds*
\begin{proof}[Proof of \Cref{prop:dilatationbounds}]
This is a special case of \Cref{prop:minimal:TFAE}.
\end{proof}
Recall \Cref{cor:minimal:charts}. \corminimalcharts*
\begin{proof}[Proof of \Cref{cor:minimal:charts}]
\Cref{prop:minimal:TFAE} shows that if $\phi$ is isothermal, then $L = D\id_{U} \colon ( TU, \, \apmd{\phi} ) \rightarrow ( TU, \, \norm{ \cdot }_{2} )$ is a Banach--Mazur minimizer pointwise $m_{2}$-almost everywhere. Recall from \Cref{prop:phi:characterization} that the pointwise dilatations satisfy
\begin{align*}
	K_{O}( \phi )
	&=
	K_{I}( L  )
	\quad \text{ and}
	\quad
	K_{I}( \phi )
	=
	K_{O}( L ).
\end{align*}
Then the claim follows from \Cref{pointwise:BM:properties}.
\end{proof}
\subsection{Conformal surfaces}\label{sec:uniformization:structure}
We study a locally reciprocal surface $Y_{d}$ (\Cref{def:local:reciprocality}) in this section. Recall that a locally reciprocal surface can be covered by quasiconformal images of Euclidean disks.

Let $\mathcal{I}_{d} = \left\{ (V_{i}, \, f_{i}) \right\}_{ i \in I }$ denote the collection of \emph{isothermal charts} (also called \emph{isothermal coordinates}). This means that $\phi_{i} = f_{i}^{-1} \colon U_{i} \rightarrow V_{i}$ is an isothermal parametrization of $V_{i}$. The subscript $d$ refers to the dependence of the atlas on the distance of $Y$.

In \Cref{maximal:atlas}, we prove that the pair $( Y, \, \mathcal{I}_{d} )$ is a \emph{conformal surface} (\Cref{def:conformal:atlas}). Then following along a proof of the classical uniformization theorem, we construct a Riemannian metric on $( Y, \, \mathcal{I}_{d} )$ (\Cref{coveringspaces} and \Cref{maximal:atlas:metric}). The main content of this section is \Cref{maximal:atlas,maximal:atlas:metric}.
\begin{definition}\label{def:conformal:atlas}
A \emph{conformal atlas} $\mathcal{D}$ is an atlas whose transition maps are conformal in the classical sense.

A conformal atlas $\mathcal{D}$ is \emph{maximal} if for every other conformal atlas $\mathcal{D}'$ with $\mathcal{D} \cap \mathcal{D}' \neq \emptyset$, we have $\mathcal{D}' \subset \mathcal{D}$.

If $\mathcal{D}$ is a maximal conformal atlas, the pair $( Y, \, \mathcal{D} )$ is a \emph{conformal surface}. A smooth surface is defined analogously.
\end{definition}
The transition maps of a conformal atlas are either holomorphic or antiholomorphic maps. Conformal surfaces are sometimes called Klein or dianalytic surfaces.
\begin{proposition}\label{maximal:atlas}
The pair $( Y, \, \mathcal{I}_{d} )$ is a conformal surface.
\end{proposition}
\begin{proof}
Since $Y_{d}$ is locally reciprocal, it can be covered by disks from which there are quasiconformal charts into $\mathbb{R}^{2}$. \Cref{thm:minimal:chart:existence} implies that without loss of generality the charts are isothermal charts. This means that the open sets $V_{i}$, $i \in I$ cover $Y_{d}$.

\Cref{prop:minimal:TFAE} Part \ref{prop:minimal:TFAE:f} shows that the restrictions of isothermal charts to open subsets of their domains are isothermal charts. \Cref{thm:minimal:chart:uniqueness} implies that $\mathcal{I}_{d}$ is a conformal atlas of $Y_{d}$ and that $\mathcal{I}_{d}$ is maximal.
\end{proof}
We define and recall some terminology from Riemannian geometry. A \emph{Riemannian norm (field)} $G$ on a conformal (or a smooth) surface $( Y, \, \mathcal{A} )$ is a map $G \colon TY \rightarrow \mathbb{R}$ for which there exists a smooth Riemannian metric $g$ such that $G( v ) = \left[ g(v, \, v) \right]^{1/2}$ for $v \in TY$. Here $TY$ is the tangent bundle of $Y$.

The length distance induced by $g$ is denoted by $d_{G}$. We say that $d_{G}$ is the \emph{Riemannian distance induced by $G$}. The metric space $Y_{d_{G}}$ has \emph{constant curvature $k$} if the corresponding Riemannian metric $g$ has constant curvature $k$. The curvature refers to Gaussian curvature.

A \emph{Riemannian surface} is a conformal (or smooth) surface with a Riemannian norm field.
\begin{definition}\label{def:riemanniansense}
A map $\nu \colon ( Y_{1}, \, G_{1} ) \rightarrow ( Y_{2}, \, G_{2} )$ between Riemannian surfaces is \emph{conformal in the Riemannian sense} if $\nu$ is a diffeomorphism and there exists a positive smooth function $h \colon Y_{2} \rightarrow \left( 0, \, \infty \right)$ such that the pushforward Riemannian norm field $\nu_{*}G_{1}$ equals $h \cdot G_{2}$.

A Riemannian norm $G$ is \emph{compatible} with a conformal atlas $\mathcal{I}$ if every chart $( V, \, f ) \in \mathcal{I}$ is conformal in the Riemannian sense.
\end{definition}
\begin{proposition}\label{maximal:atlas:metric}
The conformal surface $( Y, \, \mathcal{I}_{d} )$ has a Riemannian distance $d_{G}$ such that $G$ is compatible with the isothermal charts $\mathcal{I}_{d}$ of $Y_{d}$ and $Y_{d_{G}}$ is complete and has constant curvature $-1$, $0$ or $1$. Additionally, $\mathcal{I}_{d} = \mathcal{I}_{d_{G}}$ and the charts $( V, \, f ) \in \mathcal{I}_{d_{G}}$ are conformal in the Riemannian sense.
\end{proposition}
The equality $\mathcal{I}_{d} = \mathcal{I}_{d_{G}}$ means that $( V, \, f )$ is an isothermal chart of $Y_{d}$ if and only if it is an isothermal chart of $Y_{d_{G}}$.

Before proving \Cref{maximal:atlas:metric}, we show a couple of auxiliary results. \Cref{isothermal:compatible} proves that on Riemannian surfaces, our isothermal charts coincides with the classical ones.
\begin{lemma}\label{isothermal:compatible}
Let $Y_{d_{G}}$ be a smooth Riemannian surface and $\mathcal{I}$ a maximal conformal atlas on $Y$.

The Riemannian norm field $G$ is compatible with the conformal atlas $\mathcal{I}$ if and only if $\mathcal{I} = \mathcal{I}_{d_{G}}$. This happens if and only if $\mathcal{I} \cap \mathcal{I}_{ d_{G} }$ contains a conformal atlas of $Y$.

In particular, the isothermal charts of $Y_{d_{G}}$ are conformal in the Riemannian sense.
\end{lemma}
\begin{proof}
The intersection property follows from the fact that every conformal atlas is contained in a unique maximal conformal atlas. The rest of the proof is split into two subclaims.

Claim (1): The Riemannian surface $( Y, \, G )$ has a maximal atlas $\mathcal{A}$ of charts that are conformal in the Riemannian sense. Also $\mathcal{A} \subset \mathcal{I}_{d_{G}}$.

Proof of Claim (1): It suffices to construct an atlas $\mathcal{A}'$ of charts that are conformal in the Riemannian sense and isothermal. To that end, we consider a chart $( V, \, f )$ compatible with the smooth structure of $Y$. By restricting to a small domain $V' \subset V$, we can and do assume that $f$ is quasiconformal.

Let $f_{*}G$ denote the pushforward Riemannian norm $G \circ D( f^{-1} )$, $T_{ f_{*}G }$ the map defined in \Cref{beltrami:diff:norm}, and $W = f(V)$. Then
\begin{equation}
	\label{eq:relevant:norm}
	T_{ f_{*}G }
	=
	\id
	+
	\mu_{f_{*}G}
	\overline{\id}
	\colon
	( TW, \, f_{*}G )
	\rightarrow
	( TW, \, \norm{ \cdot }_{2} )
\end{equation}
is conformal since $G$ is a Riemannian norm. Consequently, the inner product
\begin{equation*}
	H
	=
	\frac{ 1 }{ \left( \det T_{ f_{*}G } \right)^{2} }
	T_{ f_{*}G }^{T} T_{ f_{*}G } 
\end{equation*}
is comparable to the pushforward metric $f_{*}g$. Section 2.3 of \cite{astala} shows that there is a smooth map
\begin{equation*}
	W
	\ni
	z
	\mapsto
	\mu_{f_{*}g}(z)
	\in
	\mathbb{D}
\end{equation*}
The explicit expression (2.19) of \cite{astala} for $\mu_{f_{*}g}$ and a direct computation (using the fact that the determinant of $H$ equals one and that $H$ is comparable to $f_{*}g$) shows that
\begin{equation*}
	\mu_{ f_{*}g }
	=
	\frac{ 2 \overline{ \mu_{f_{*}G} } }{ 2 - \norm{ \mu_{f_{*}G} }^{2} + \norm{ \mu_{f_{*}G} }^{4} }.
\end{equation*}
The map
\begin{equation*}
	\mathbb{D}
	\ni
	z
	\mapsto
	\frac{ 2 \overline{z} }{ 2 - \norm{ z }_{2}^{2} + \norm{ z }_{2}^{4} }
	\in
	\mathbb{D}
\end{equation*}
is a diffeomorphism of the Euclidean ball onto itself. Therefore $W \ni z \mapsto \mu_{ f_{*}G }(z)$ is smooth.

By considering a subdomain of $V$, we can and will assume that $z \mapsto \mu_{ f_{*}G }(z)$ is the restriction of a Beltrami differential $\mu' \in \mathcal{C}^{\infty}_{0}( \mathbb{R}^{2} )$ that satisfies $\norm{ \mu' }_{ L^{\infty}( \mathbb{R}^{2} ) } = k < 1$. Let $\nu$ denote the normalized solution to the Beltrami equation induced by $\mu'$ whose existence is guaranteed by Theorem 5.2.4 of \cite{astala}. The same theorem guarantees that $\nu$ is a diffeomorphism mapping $\mathbb{R}^{2}$ onto itself.

The map $\widetilde{f} = \nu \circ f \colon V \rightarrow \widetilde{W} \subset \mathbb{R}^{2}$ is a diffeomorphism that is isothermal (as $\apmd{ f^{-1} } = f_{*}G$ everywhere and since \Cref{prop:minimal:TFAE} holds). The Banach--Mazur distance between $f_{*}G$ and $\norm{ \cdot }_{2}$ equals one, therefore \Cref{prop:minimal:TFAE} implies that $\widetilde{f}$ is conformal in the geometric sense. Since $\widetilde{f}$ is also a diffeomorphism, the map $\widetilde{f}$ is conformal in the Riemannian sense. The existence of the atlas $\mathcal{A'}$ follows by considering the charts of the form $( V, \, \widetilde{f} )$.

Claim (2): The atlas $\mathcal{I}_{d_{G}}$ coincides with the maximal atlas $\mathcal{A}$ of charts that are conformal in the Riemannian sense.

Proof of Claim (2): Claim (1) shows that the atlas $\mathcal{A}$ is a subset of $\mathcal{I}_{d_{G}}$. Since the transition maps between $\mathcal{A}$ and $\mathcal{I}_{d_{G}}$ are conformal in the classical sense (\Cref{thm:minimal:chart:uniqueness}), every chart $( V, \, f ) \in \mathcal{I}_{d_{G}}$ is a diffeomorphism and conformal in the geometric sense. The proof of Claim (2) is complete.

Claim (2) proves that the atlas $\mathcal{I}_{d_{G}}$ of isothermal charts coincide with the atlas $\mathcal{A}$ of charts that are conformal in the Riemannian sense. As a consequence, a maximal conformal atlas $\mathcal{I}$ is compatible with the Riemannian norm $G$ if and only if $\mathcal{I} = \mathcal{I}_{d_{G}}$.
\end{proof}
\begin{lemma}\label{coveringspaces}
Suppose that $( Y, \, \mathcal{I} )$ is a conformal surface. Then there exists a universal cover $( \Omega, \, \pi )$ of $Y$, where $\Omega$ is either the hyperbolic plane $\mathbb{H}^{2}$, the Riemann sphere $\mathbb{S}^{2}$, or the Euclidean plane $\mathbb{R}^{2}$; the map $\pi \colon \Omega \rightarrow ( Y, \, \mathcal{I} )$ is locally conformal; and the covering group of $\pi$ acts by isometries on $\Omega$.
\end{lemma}
\begin{proof}
First of all, the existence of a universal cover of $( Y, \, \mathcal{I} )$ follows by pulling back the conformal structure $\mathcal{I}$ to a topological universal cover. The classical uniformization theorem allows us to assume that the universal cover is $\mu_{1} \colon \Omega \rightarrow ( Y, \, \mathcal{I} )$, where $\Omega$ is one of the listed spaces. Clearly, the covering group $G_{1}$ of $\mu_{1}$ acts by conformal maps on $\Omega$. In the following, we identify $\mathbb{S}^{2}$ with the extended plane $\mathbb{R}^{2} \cup \left\{ \infty \right\}$ by using the stereographic projection that fixes the unit circle $\mathbb{S}^{1} = \mathbb{S}^{1} \times \left\{ 0 \right\} \subset \mathbb{R}^{3}$ and that maps the south pole $( 0, \, 0, \, -1 )$ to the origin $0 \in \mathbb{R}^{2}$. The claim splits into three cases.
\begin{enumerate}[label=(\alph*)]
\item\label{coveringspaces:case1} $\Omega = \mathbb{H}^{2}$ or $\mathbb{R}^{2}$ and we can take $\pi = \mu_{1}$.
\item\label{coveringspaces:case2} $\Omega = \mathbb{S}^{2}$ and the covering group $G_{1}$ only contains the identity element. Then we can take $\pi = \mu_{1}$.
\item\label{coveringspaces:case3} $\Omega = \mathbb{S}^{2}$ and the covering group $G_{1}$ has two elements. Then there exists a(n orientation-preserving linear) conformal map $L \colon \Omega \rightarrow \Omega$ such that the covering group of $\mu_{2} = \mu_{1} \circ L$ is generated by the antipodal map
\begin{equation*}
	z
	\mapsto
	-
	\frac{ 1 }{ \overline{z} }.
\end{equation*}
In particular, we can set $\pi = \mu_{2}$.
\end{enumerate}
Proof of Case \ref{coveringspaces:case1}: If $\Omega = \mathbb{H}^{2}$, then every conformal map from $\Omega$ onto itself is an isometry of $\Omega$ (by the standard construction of the Riemannian structure on $\Omega$) therefore $G$ consists of isometries of $\Omega$.

Consider the case $\Omega = \mathbb{R}^{2}$. Then for every $h \in G$, we have two cases. If $h$ is orientation-preserving, then it is a translation or the identity map (otherwise it would have fixed points). If $h$ is orientation-reversing, then $h^{2}$ is orientation-preserving and hence it is either a translation or the identity map. This fact implies that $h(z) = a \overline{z} + c$ for some $a \in \mathbb{S}^{1}$ and $c \in \mathbb{R}^{2}$. The proof of Case \ref{coveringspaces:case1} is complete.

If $\Omega = \mathbb{S}^{2}$, the group of orientation-preserving elements of $G_{1}$ is a singleton; recall that orientation-preserving Möbius transformations have fixed points. Thus we are in the setting of Case \ref{coveringspaces:case2} or \ref{coveringspaces:case3}. In Case \ref{coveringspaces:case2}, the covering group is trivial and so is the claim. Otherwise the group $G_{1}$ is generated by an orientation-reversing Möbius involution $\sigma \colon \Omega \rightarrow \Omega$ which has no fixed points.

Proof of Case \ref{coveringspaces:case3}: The characterization of Möbius transformations of $\mathbb{S}^{2}$ implies that
\begin{equation*}
	\sigma(z)
	=
	\frac{ A \overline{z} + B }{ C \overline{z} + D },
\end{equation*}
where $A$, $B$, $C$ and $D$ are complex numbers. We assume without loss of generality that $A D - B C = 1$. Since $\sigma$ does not have any fixed points, we observe that $B \neq 0 \neq C$. Consider the constants $\mu_{1} = \frac{A}{C}$, $\mu_{2} = \frac{B}{C} \neq 0$, and $\mu_{3} = \frac{D}{C}$. Since $\sigma^{2} = \id$, a direct computation shows that
\begin{enumerate}[label=(\alph*)]
\item $\mu_{3} = - \overline{\mu_{1}}$ ($\sigma^{2}$ fixes $\infty$);
\item $\mu_{1}\overline{\mu_{2}} + \mu_{2}\overline{\mu_{3}} = 0$ ($\sigma^{2}$ fixes $0$);
\item $\norm{ \mu_{1} }^{2}_2 + \mu_{2} = \overline{\mu_{2}} + \norm{ \mu_{3} }^{2}_2$ ($\sigma^{2}$ fixes $0$, $1$, and $\infty$);
\item $C^{-2} = \mu_{1}\mu_{3} - \mu_{2}$ (the normalization $AD - BC = 1$).
\end{enumerate}
The first and third conditions imply that $\mu_{2} \in \mathbb{R} \setminus \left\{ 0 \right\}$, and hence $C^{-2} = - \mu_{2} - \norm{ \mu_{1} }^{2}_2 \in \mathbb{R}$ by the fourth condition.

We show that $C \in \mathbb{R} \setminus \left\{ 0 \right\}$. First, the fact that $C^{-2}$ is real implies that either the real or imaginary part of $C$ equals zero. Secondly, for $z \neq \infty$, the equation $z = \sigma(z)$ is equivalent to $\norm{ z - \mu_{1} }_{2}^{2} = - C^{-2}$. Since $\sigma$ has no fixed points, $C$ cannot be imaginary. Therefore $C \in \mathbb{R} \setminus \left\{ 0 \right\}$.

A direct computation shows that given the linear map
\begin{equation*}
	L(z)
	=
	C^{-1} z + \mu_{1},
\end{equation*}
the Möbius transformation $\tilde{\sigma} = L^{-1} \circ \sigma \circ L$ is the antipodal map. Then the covering group of $\mu_{2} = \mu_{1} \circ L$ is generated by the antipodal map $\tilde{\sigma}$ which is an isometry of $\mathbb{S}^{2}$.
\end{proof}
\begin{proof}[Proof of \Cref{maximal:atlas:metric}]
For convenience, we suppress the notation for $\mathcal{I} = \mathcal{I}_{d}$ in the following argument. Fix a universal cover
\begin{equation*}
	\pi
	\colon
	\Omega
	\rightarrow
	Y,
\end{equation*}
where $\pi$ is locally conformal, its covering group $G$ acts by isometries on $\Omega$, and $\Omega$ is either the hyperbolic plane $\mathbb{H}^{2}$, the Euclidean plane $\mathbb{R}^{2}$, or the Riemann sphere $\mathbb{S}^{2}$. The existence of $\pi$ follows from \Cref{coveringspaces}. As the covering group $G$ acts by isometries on $\Omega$, the Riemannian inner product $g_{\Omega}$ of $\Omega$ can be pushed forward to an inner product $g$ on $Y$ in such a way that
\begin{equation*}
	\pi
	\colon
	( \Omega, \, g_{\Omega} )
	\rightarrow
	( Y, \, g )
\end{equation*}
is a local isometry in the Riemannian sense; see for example \cite[Chapter 2]{leeRiemannian}. Let $G$ denote the corresponding Riemannian norm on $Y$. Then the metric surface $Y_{d_{G}}$ has the same curvature as $\Omega$ (since $\pi$ is a local isometry) and complete (as a consequence of Hopf--Rinow theorem and completeness of $\Omega$).

Let $\mathcal{K}$ denote the atlas consisting of the inverses of the restrictions of $\pi$ to small open subsets of $\mathbb{R}^{2} \cap \Omega$. Since the elements of $\mathcal{K}$ are local isometries and $g_{\Omega}$ is comparable to the Euclidean inner product in $\mathbb{R}^{2} \cap \Omega$, the elements of $\mathcal{K}$ are isothermal charts that are compatible with $G$. Therefore \Cref{isothermal:compatible} implies that $G$ is compatible with $\mathcal{I}$.
\end{proof}
\section{Quasiconformal maps between reciprocal surfaces}\label{sec:uniformization:main:theorem}
We generalize the analysis of quasiconformal parametrizations (\Cref{sec:uniformization:charts}) to analysis of quasiconformal maps between locally reciprocal surfaces. Let $d_{G}$ denote the Riemannian distance obtained from \Cref{maximal:atlas:metric}. Given two locally reciprocal surfaces $Y_{d_{1}}$ and $Y_{d_{2}}$, we make the following simplifications for $i = 1, \, 2$.
\begin{enumerate}[label=(\alph*)]
\item\label{simple:surface} The triple $Y_{d_{i}} = ( Y_{i}, \, d_{i}, \, \mathcal{H}^{2}_{d_{i}} )$ is denoted by $Y_{i} = ( Y_{i}, \, d_{i}, \, \mathcal{H}^{2}_{i} )$.
\item\label{simple:surface:riemannian} The triple $Y_{G_{i}} = ( Y_{i}, \, d_{G_{i}}, \, \mathcal{H}^{2}_{G_{i}} )$ refers to the triple $Y_{d_{G_{i}}} = ( Y_{i}, \, d_{ G_{i} }, \, \mathcal{H}^{2}_{ d_{ G_{i} } } )$.
\item\label{simple:surface:id} The minimal upper gradient of $\id_{Y_{d_{i}}}$ is denoted by $\rho_{i}$.
\item\label{simple:surface:weight} The weighted measure $\rho_{i} \mathcal{H}^{2}_{i}$ (introduced in \Cref{sec:analysisonQCmaps}) is denoted by $\nu_{i}$.
\item\label{simple:surface:uniformizationmap} The map $u_{i} = \id_{Y_{i}} \colon Y_{G_{i}} \rightarrow Y_{i}$ is called the \emph{uniformization map}.
\item\label{simple:surface:isothermal} \Cref{maximal:atlas:metric} implies that every isothermal parametrization of a subdomain of $Y_{i}$ (the metric surface $Y_{d_{i}}$) is of the following form: There exists a homeomorphism
\begin{equation*}
	\phi \colon \mathbb{R}^{2} \supset U \rightarrow V \subset Y_{G_{i}}
\end{equation*}
conformal in the Riemannian sense for which $u_{i} \circ \phi$ is an isothermal parametrization of $u_{i}( V ) \subset Y_{i}$. Unless it is otherwise stated, we express isothermal parametrizations of $Y_{i}$ in this way.
\end{enumerate}
If we are studying a single locally reciprocal surface, we omit the subscript $i$.
\subsection{Norm fields}\label{sec:normfield:surface}
\begin{lemma}\label{def:derivative}
Let $Y$ be a locally reciprocal surface. Then there is a measurable section $\apmd{} = \apmd{d}$ of the norm bundle of $Y$, where for every isothermal parametrization
\begin{equation*}
	\label{eq:isothermal:parametrization}
	u \circ \phi
	\colon
	\mathbb{R}^{2} \supset U
	\rightarrow
	V \subset Y,
\end{equation*}
and its approximate metric differential $\apmd{u \circ \phi}$, we have that
\begin{equation}
	\label{eq:chain:rule:surface}
	\apmd{ u \circ \phi }
	=
	\apmd{d} \circ D\phi.
\end{equation}
\end{lemma}
\begin{proof}
Since the transition maps of $\mathcal{I}_{d_{G}}$ are diffeomorphisms, the expression
\begin{equation*}
	\apmd{ u \circ \phi } \circ D( \phi^{-1} )
\end{equation*}
depends only on the image point of $\phi$, not on the isothermal parametrization $\phi$ (recall the chain rule \Cref{diffeo}). \Cref{norm:almost:everywhere} and \Cref{newtonian:seminorm} show that $\apmd{ u \circ \phi }$ exists and is a norm $m_{2}$-almost everywhere in $U$. Thus $\apmd{d} \colon TY \rightarrow \mathbb{R}$ is well-defined $\mathcal{H}^{2}_{G}$-almost everywhere in $V$.

The norm bundle of $Y$ is defined using the isothermal charts of $Y_{G}$: A real-valued map $N \colon TY \rightarrow \mathbb{R}$ is a measurable section of the norm bundle of $Y$ if for any isothermal parametrization $\phi$ of $V \subset Y_{G}$, we have that $N \circ D\phi$ is measurable as a map into $\mathcal{C}( \mathbb{S}^{1}, \, \mathbb{R} )$ that is also a norm at $m_{2}$-almost every point. This holds for $\apmd{}$ as a consequence of \Cref{newtonian:seminorm} and \Cref{norm:almost:everywhere}.
\end{proof}
\begin{lemma}\label{lem:borelrepresentative}
There exists a Borel set $B \subset Y$ of $\mathcal{H}^{2}_{G}$-measure zero for which $y \mapsto ( \apmd{} \chi_{ Y \setminus B } + G \chi_{ B } )( y )$ is a Borel measurable section of the norm bundle of $Y$.
\end{lemma}
\begin{proof}
We fix a countable subatlas $\mathcal{J} \subset \mathcal{I}_{d}$. By applying \Cref{newtonian:seminorm} and \Cref{norm:almost:everywhere}, given any parametrization $\phi \colon U \rightarrow V$ with $( V, \, f = \phi^{-1} ) \in \mathcal{J}$, there exists a Borel set $B' \subset U$ such that $x \mapsto \left( \apmd{u\circ\phi} \chi_{ U \setminus B' } + G \circ D\phi \chi_{ B' } \right)(x)$ is Borel measurable as a map into $\mathcal{C}( \mathbb{S}^{1}, \, \mathbb{R} )$ and a norm everywhere. Define $B_{f} = \phi( B' )$ and
\begin{equation*}
	B
	=
	\bigcup_{ (V, \, f) \in \mathcal{J} }
		B_{f}.
\end{equation*}
Then $B$ has zero $\mathcal{H}^{2}_{G}$-measure. Also, the map
\begin{equation*}
	y \mapsto \left( \apmd{} \chi_{ Y \setminus B } + G \chi_{ B } \right)(y)
\end{equation*}
is Borel measurable (section of the norm bundle) and a norm everywhere.
\end{proof}
The uniformization map $u = \id_{Y} \colon Y_{G} \rightarrow Y$ is defined on \ref{simple:surface:uniformizationmap} at the beginning of this section.
\begin{lemma}\label{lem:uniformization:pointwisedilatation}
The uniformization map $u \colon Y_{G} \rightarrow Y$ is $\frac{ \pi }{ 2 }$-quasiconformal and if $\phi \colon U \rightarrow V \subset Y_{G}$ is an isothermal parametrization of $V \subset Y_{G}$, the pointwise dilatations of $u$ satisfy
\begin{equation}
	\label{eq:pointwise:dilatations}
	K_{O}( u )
	=
	K_{O}( u \circ \phi ) \circ \phi^{-1}
	\text{ and }
	K_{I}( u )
	=
	K_{I}( u \circ \phi ) \circ \phi^{-1}.
\end{equation}
\end{lemma}
\begin{proof}
To see that $u$ is $\frac{\pi}{2}$-quasiconformal, we recall that $u \circ \phi$ is $\frac{ \pi }{ 2 }$-quasiconformal (\Cref{cor:minimal:charts}) and $\phi$ is conformal (\Cref{maximal:atlas:metric}). The pointwise identities for dilatations of $u$ follow from the composition laws for dilatations (\Cref{lem:outer:inner:characterization}) and the fact that $\phi$ is conformal.
\end{proof}
\begin{corollary}\label{cor:uniformization:pointwisedilatation:minimal}
Let $u$ be the uniformization map. Then $\mathcal{H}^{2}_{G}$-almost everywhere the pointwise dilatations of $u$ satisfy
\begin{equation}
	\label{eq:pointwise:dilatation:uniformization}
	\rho^{2}( G, \, \apmd{} )
	=
	K_{O}( u )K_{I}( u ),
\end{equation}
where $\rho( G, \, \apmd{} )$ is the Banach--Mazur distance between the Banach spaces $( TY, \, G )$ and $( TY, \, \apmd{} )$.
\end{corollary}
The Banach--Mazur distance in \Cref{def:BM:distance} generalizes to distances between Banach spaces, however, we can also apply \Cref{def:BM:distance} directly in the following equivalent way: Given an isothermal parametrization $\phi \colon U \rightarrow V \subset Y_{G}$, we define $\rho( \apmd{d}, \, G ) \circ \phi$ to be equal to $\rho( \apmd{d} \circ D\phi, \, G \circ D\phi )$.
\begin{proof}
Since $\apmd{d} \circ D\phi = \apmd{ u \circ \phi }$ (\Cref{def:derivative}) and $G \circ D\phi$ is comparable to the Euclidean norm (\Cref{maximal:atlas:metric}), the corresponding claim about $u \circ \phi$ (\Cref{prop:dilatationbounds}) establishes \eqref{eq:pointwise:dilatation:uniformization}.
\end{proof}
\begin{lemma}\label{lem:abscontinuity}
The map $u \colon ( Y, \, \mathcal{H}^{2}_{G} ) \rightarrow ( Y, \, \nu )$ satisfies Conditions ($N$) and ($N^{-1}$).
\end{lemma}
\begin{proof}
The claim is an immediate consequence of the fact that $u$ is quasiconformal (\Cref{lem:uniformization:pointwisedilatation}), the fact that the minimal upper gradient of $\id_{Y_{G}}$ equals $\chi_{Y_{G}}$ (\Cref{lem:Condition(N):inverse}), and \Cref{lem:Condition(N)}.
\end{proof}
In the following sections, it is sometimes convenient to consider the differential
\begin{equation*}
	Du
	\colon
	( TY, \, G )
	\rightarrow
	( TY, \, \apmd{} ),
\end{equation*}
where we consider the norm field $\apmd{}$ to be well-defined $\nu$-almost everywhere in $Y$. This makes sense due to \Cref{lem:abscontinuity}.
\subsection{Quasiconformal maps}\label{sec:qcmaps:surfaces}
Our goal is to understand the analogs of \Cref{lem:uniformization:pointwisedilatation} and \Cref{cor:uniformization:pointwisedilatation:minimal} for two locally reciprocal surfaces $Y_{1}$ and $Y_{2}$ and for an arbitrary quasiconformal map
\begin{equation}
	\label{eq:arbitrary:QC:map}
	\Psi
	\colon
	Y_{1}
	\rightarrow
	Y_{2}.
\end{equation}
To that end, since the map
\begin{equation*}
	\widetilde{\Psi}
	=
	u_{2}^{-1} \circ \Psi \circ u_{1}
	\colon
	Y_{G_{1}}
	\rightarrow
	Y_{G_{2}}
\end{equation*}
is quasiconformal as a map between two Riemannian surfaces, it is classically differentiable $\mathcal{H}^{2}_{G_{1}}$-almost everywhere and it satisfies Conditions ($N$) and ($N^{-1}$); see Section 3.3 of \cite{astala}.
\begin{lemma}\label{lem:differential:QC:map}
The differential
\begin{equation}
	\label{eq:differential:QC:map}
	D\Psi
	\colon
	( TY_{1}, \, \apmd{1} )
	\rightarrow
	( TY_{2}, \, \apmd{2} )
\end{equation}
is well-defined $\nu_{1}$-almost everywhere. Moreover,
\begin{equation*}
	D( \Psi^{-1} ) \circ D\Psi
	=
	D\id_{Y_{1}}
\end{equation*}
$\nu_{1}$-almost everywhere.
\end{lemma}
\begin{proof}
Since \Cref{lem:abscontinuity} holds and $\widetilde{\Psi}$ satisfies Conditions ($N$) and ($N^{-1}$), the claim follows from \Cref{lem:borelrepresentative} and the fact that $D( \widetilde{\Psi}^{-1} ) \circ D\widetilde{\Psi} = D\id_{G_{1}}$ $\mathcal{H}^{2}_{G_{1}}$-almost everywhere.
\end{proof}
Fix a point $x$, where $\apmd{1}( x )$ and $\apmd{2}( \Psi(x) )$ are norms and where $D\Psi(x)$ is an invertible linear map, i.e., $\nu_{1}$-almost any point $x \in Y_{1}$ (\Cref{lem:differential:QC:map}). Then the operator norm of
\begin{equation}
	\label{eq:differential:QC:map:pointwise}
	D\Psi(x)
	\colon
	( T_{x}Y_{1}, \, \apmd{1}(x) )
	\rightarrow
	( T_{\Psi(x)}Y_{2}, \, \apmd{2}( \Psi(x) ) ).
\end{equation}
is a minimal upper gradient of $D\Psi(x)$. We denote it by $\norm{ D\Psi }( x )$. The Jacobian of $D\Psi(x)$ is defined as in Definition 4.1 of \cite{ambrosio2000} and is denoted by $J_{2}( D\Psi )(x)$. The operator norm of the inverse of $D\Psi(x)$ is denoted by $\norm{ D\Psi^{-1} }( \Psi(x) )$. It is the minimal upper gradient of $D( \Psi^{-1} )( \Psi(x) )$. The outer and inner dilatation of $D\Psi(x)$ are denoted by $K_{O}( D\Psi )(x)$ and $K_{I}( D\Psi )(x)$, respectively. The maximal dilatation of $D\Psi(x)$ is denoted by $K( D\Psi )(x)$. The following result is a consequence of \Cref{lem:differential:QC:map}.
\begin{lemma}\label{def:differential}
The operator norm $\norm{ D\Psi }$ of $D\Psi$ and its Jacobian $J_{2}( D\Psi )$ are well-defined $\nu_{1}$-almost everywhere. Also the pointwise outer dilatation $K_{O}( D\Psi )$, inner dilatation $K_{I}( D\Psi )$ and maximal dilatation $K( D\Psi )$ of $D\Psi$ are well-defined $\nu_{1}$-almost everywhere.
\end{lemma}
\begin{definition}\label{def:conformality}
We say that the differential $D\Psi$ is \emph{conformal $\nu_{1}$-almost everywhere} if $K_{O}( D\Psi ) K_{I}( D\Psi ) = \chi_{ Y_{1} }$ $\nu_{1}$-almost everywhere. Equivalently, the maximal dilatation $K( D\Psi ) = \chi_{ Y_{1} }$ $\nu_{1}$-almost everywhere.
\end{definition}
\begin{remark}\label{def:conformality:remark}
Note that $K_{O}( D\Psi )$ and $K_{I}( D\Psi )$ are greater than $\chi_{ Y_{1} }$ $\nu_{1}$-almost everywhere. This is a consequence of the following argument. Given two two-dimensional Banach spaces $E_{1}$ and $E_{2}$ and an invertible linear map $L \colon E_{1} \rightarrow E_{2}$, the metric balls satisfy the following inclusions:
\begin{equation*}
	B_{ E_{2} }
	\left(
		0, \, \frac{ 1 }{ \norm{ L^{-1} } }
	\right)
	\subset
	L( B_{ E_{1} }( 0, \, 1 ) )
	\subset
	B_{ E_{2} }
	\left(
		0, \, \norm{ L }
	\right).
\end{equation*}
Given these inclusions, using the monotonicity of the Hausdorff $2$-measure, and by applying the change of variables formula for $L$ (see \cite{ambrosio2000}), we deduce that $\min\left\{ K_{O}( L ), \, K_{I}( L ) \right\} \geq 1$. We conclude from this fact that $K( L ) = 1$ if and only if $K_{O}( L ) K_{I}( L ) = 1$. As a consequence, the conformality of $D\Psi$ is well-defined.
\end{remark}
Let $B_{0}$ be a Borel set of $\mathcal{H}^{2}_{G_{1}}$-measure zero such that the restrictions of $u_{1}$ and $u_{2} \circ \widetilde{\Psi}$ to $Y_{G_{1}} \setminus B_{0}$ satisfy Conditions ($N$) and ($N^{-1}$). The existence of such a set is guaranteed by \Cref{lem:abscontinuity}, by the fact that $\widetilde{\Psi}$ satisfies Conditions ($N$) and ($N^{-1}$), and by \Cref{lem:Condition(N)}. We fix such a set for the rest of this section.
\begin{lemma}\label{jacobian:surfaces}
The Jacobian $J_{\Psi}$ of $\Psi$ equals $J_{2}( D\Psi )$ $\mathcal{H}^{2}_{1}$-almost everywhere in $Y_{1} \setminus u_{1}( B_{0} )$. In particular, this identity holds $\nu_{1}$-almost everywhere.
\end{lemma}
\begin{proof}
The claim is local so it suffices to consider the claim using isothermal charts of $Y_{1}$ and of $Y_{2}$. The isothermal charts satisfy Conditions ($N$) and ($N^{-1}$) when restricted to the complement of $u_{1}( B_{0} )$ and $\Psi \circ u_{1}( B_{0} )$, respectively. Then the claim follows from the change of variables formula stated in \Cref{newtonian:seminorm}, the chain rule of approximate metric differentials (\Cref{prop:chainrule:QC}), and the chain rule of Jacobians of linear maps between Banach spaces \cite[Lemma 4.2]{ambrosio2000}.
\end{proof}
We fix a Borel set $B_{1} \supset B_{0}$ of zero $\mathcal{H}^{2}_{G_{1}}$-measure for which the following properties hold:
\begin{enumerate}[label=(\alph*)]
\item\label{necessaryevil:1} The maps $Y_{1} \setminus u_{1}( B_{1} ) \ni y \mapsto \apmd{1}( y )$ and $Y_{2} \setminus \Psi( u_{1}( B_{1} ) ) \ni y \mapsto \apmd{2}( y )$ are norms everywhere and also Borel measurable;
\item\label{necessaryevil:2} The maps $Y_{1} \setminus u_{1}( B_{1} ) \ni y \mapsto D\Psi( y )$ and $Y_{2} \setminus \Psi( u_{1}( B_{1} ) ) \ni y \mapsto D( \Psi^{-1} )( y )$ are Borel measurable and the chain rule $D( \Psi^{-1} ) \circ D\Psi = D\id_{Y_{1}}$ holds everywhere in $Y_{1} \setminus u_{1}( B_{1} )$.
\end{enumerate}
Recall that $\widetilde{\Psi}$ satisfies Conditions ($N$) and ($N^{-1}$). The existence of $B_{1}$ satisfying property \ref{necessaryevil:1} is guaranteed by \Cref{lem:borelrepresentative}. The existence of $B_{1}$ satisfying property \ref{necessaryevil:2} follows by a similar reasoning.

The main point of $B_{1}$ is that the operator norm of $D\Psi$ from \eqref{eq:differential:QC:map:pointwise} and of its inverse $D( \Psi^{-1} )$ are well-defined everywhere in the complement of $u_{1}( B_{1} )$ and $\Psi( u_{1}( B_{1} ) )$, respectively. The restriction of $\Psi$ to the complement of $u_{1}( B_{1} )$ satisfies Conditions ($N$) and ($N^{-1}$).
\begin{definition}\label{def:minimaluppergradient:representative}
We define $I_{\Psi}$ as the operator norm $\norm{ D\Psi }$ in $Y_{1} \setminus u_{1}( B_{1} )$ and zero otherwise. The function $I_{\Psi^{-1}}$ is defined as the operator norm $\norm{ D( \Psi^{-1} ) }$ in $Y_{2} \setminus \Psi( u_{1}( B_{1} ) )$ and zero otherwise.
\end{definition}
\begin{proposition}\label{prop:characterizations}
The Borel functions $I_{\Psi}$ and $I_{\Psi^{-1}}$ are minimal upper gradients of $\Psi$ and $\Psi^{-1}$, respectively.
\end{proposition}
We state some consequences of \Cref{prop:characterizations} before proving the statement itself. Recall the definition of Banach--Mazur minimizers from \Cref{def:BM:distance:min}.
\begin{corollary}\label{cor:dilatations}
The equalities $K_{O}( \Psi ) = K_{O}( D\Psi )$ and $K_{I}( \Psi ) = K_{I}( D\Psi )$ hold $\nu_{1}$-almost everywhere. In particular, the pointwise dilatations satisfy
\begin{equation}
	\label{eq:BM:distance:generalQC}
	K_{O}( D\Psi )K_{I}( D\Psi )
	\geq
	\rho^{2}( \apmd{1}, \, \apmd{2} \circ D\Psi )
\end{equation}
$\nu_{1}$-almost everywhere. The equality \eqref{eq:BM:distance:generalQC} holds $\nu_{1}$-almost everywhere if and only if the differential
\begin{equation*}
	D\Psi
	\colon
	( TY_{1}, \, \apmd{1} )
	\rightarrow
	( TY_{2}, \, \apmd{2} )
\end{equation*}
is a Banach--Mazur minimizer $\nu_{1}$-almost everywhere.
\end{corollary}
\begin{proof}
\Cref{lem:abscontinuity} implies that $u_{1}( B_{1} )$ has zero $\nu_{1}$-measure. Therefore it suffices to show the claimed identities in the complement of $u_{1}( B_{1} )$.

The identities of the pointwise dilatations of $\Psi$ are immediate consequences of the expression of the Jacobian $J_{\Psi}$ of $\Psi$ (\Cref{jacobian:surfaces}), the expressions of minimal upper gradients of $\Psi$ and $\Psi^{-1}$ (\Cref{prop:characterizations}), and the definitions of the pointwise dilatations of $\Psi$ (\Cref{def:pointwise:dilatation}) and of $D\Psi$ (\Cref{def:differential}).

The inequality \eqref{eq:BM:distance:generalQC} follows from the definition of Banach--Mazur minimizers (\Cref{def:BM:distance:min}). The same definition implies that the equality holds as claimed.
\end{proof}
\begin{corollary}\label{cor:conformality}
A quasiconformal map $\Psi \colon Y_{1} \rightarrow Y_{2}$ between locally reciprocal surfaces is conformal if and only if $D\Psi \colon ( TY_{1}, \, \apmd{1} ) \rightarrow ( TY_{2}, \, \apmd{2} )$ is conformal $\nu_{1}$-almost everywhere. This happens if and only if
\begin{equation*}
	K_{O}( D\Psi )K_{I}( D\Psi )
	=
	\rho^{2}( \apmd{1}, \, \apmd{2} \circ D\Psi )
	=
	\chi_{ Y_{1} }
\end{equation*}
$\nu_{1}$-almost everywhere.
\end{corollary}
\begin{proof}
First, we consider the claim about the conformality of $\Psi$ and $D\Psi$ being equivalent.

As we defined in \Cref{sec:intro:definitions}, the map $\Psi$ is conformal if its (global) maximal dilatation equals one. The pointwise dilatations of $\Psi$ can be expressed using the corresponding dilatations of $D\Psi$ (\Cref{cor:dilatations}). The latter dilatations are greater than $\chi_{ Y_{1} }$ (\Cref{def:conformality:remark} and \Cref{def:conformality}). Therefore the maximal dilatation $\Psi$ equals one if and only if the pointwise dilatations of $\Psi$ equal $\chi_{ Y_{1} }$ $\nu_{1}$-almost everywhere. This happens if and only if the product of the pointwise dilatations equals $\chi_{ Y_{1} }$ $\nu_{1}$-almost everywhere.

We conclude that $\Psi$ is conformal if and only if the differential $D\Psi$ is conformal $\nu_{1}$-almost everywhere.

The Banach--Mazur distance between two Banach spaces is always greater than one, therefore \eqref{eq:BM:distance:generalQC} implies that $\Psi$ is conformal if and only if
\begin{equation*}
	K_{O}( D\Psi )K_{I}( D\Psi )
	=
	\rho^{2}( \apmd{1}, \, \apmd{2} \circ D\Psi )
	=
	\chi_{ Y_{1} }
\end{equation*}
$\nu_{1}$-almost everywhere. The proof is complete.
\end{proof}
\begin{proof}[Proof of \Cref{prop:characterizations}]
We fix the Borel set $B_{1}$ as before. We recall the relevant properties. The set $u_{1}(  B_{1} )$ has zero $\nu_{1}$-measure (\Cref{lem:abscontinuity}) and the restriction of $\Psi$ to the complement of $u_{1}( B_{1} )$ satisfies Conditions ($N$) and ($N^{-1}$). The restriction of $u_{1}$ to the complement of $B_{1}$ satisfies Conditions ($N$) and ($N^{-1}$). Also properties \ref{necessaryevil:1} and \ref{necessaryevil:2} in the definition of $B_{1}$ yield the following --- they state that the differential $D\Psi$ is well-defined and invertible everywhere in the complement of $u_{1}( B_{1} )$. Also the norm field $\apmd{1}$ is a norm everywhere in the complement of $u_{1}( B_{1} )$.

Similar properties hold for the set $u_{2}^{-1}\left( \Psi \circ u_{1}( B_{1} ) \right)$ and the maps $u_{2}$ and $\Psi^{-1}$. Also $D( \Psi^{-1} )$ and $\apmd{2}$ are well-defined in the complement of $\Psi( u_{1}(B_{1}) )$.

Let $\Gamma$ denote the $\Psi$-, $u_{1}^{-1}$-, and $u_{2}^{-1} \circ \Psi$-singular paths (recall \Cref{def:QC:singularset;goodpaths}). The family $\Gamma$ has zero modulus by \Cref{lem:QC:differentequivalent}. Let $\Gamma^{+}_{ u_{1}( B_{1} ) }$ denote the paths that have positive length in $u_{1}( B_{1} ) \subset Y_{1}$. That family has zero modulus by \Cref{identity:minimal:upper:gradient}. Let $\Gamma_{1}$ denote the union of $\Gamma$ and $\Gamma^{+}_{ u_{1}( B_{1} ) }$. This family has zero modulus.

We prove two claims. Claim (1): The operator norm
\begin{equation*}
	I_{\Psi} = \norm{ D\Psi } \chi_{ Y_{1} \setminus u_{1}( B_{1} ) }
\end{equation*}
is a weak upper gradient of $\Psi$.

Claim (2): If $\rho_{\Psi}$ is a minimal upper gradient of $\Psi$, then $I_{\Psi} \leq \rho_{\Psi}$ $\mathcal{H}^{2}_{1}$-almost everywhere.

These claims imply that $I_{\Psi} = \rho_{ \Psi }$ $\mathcal{H}^{2}_{1}$-almost everywhere which is what we wanted to prove. The argument for $I_{\Psi^{-1}}$ is similar so we only prove Claims (1) and (2). Since $Y_{1}$ can be covered by isothermal charts, it suffices to prove Claims (1) and (2) in the image of an arbitrary isothermal parametrization.

Consider an isothermal parametrization $u_{1} \circ \phi_{1} \colon U_{1} \rightarrow V_{1} \subset Y_{1}$ and let $V_{2} = \Psi( V_{1} )$. Fix an isothermal parametrization $u_{2} \circ \phi_{2} \colon U_{2} \rightarrow V_{2}$. (The existence of $\phi_{2}$ follows from \Cref{thm:minimal:chart:existence}.) Define a quasiconformal map $\psi \colon U_{1} \rightarrow U_{2}$ by the formula $\Psi \circ ( u_{1} \circ \phi_{1} ) = ( u_{2} \circ \phi_{2} ) \circ \psi$.

Proof of Claim (1): Let $\widetilde{N}_{1} \subset U_{1}$ denote a Borel set of zero $m_{2}$-measure that contains $\phi_{1}^{-1}( B_{1} )$ and the points where the chain rule
\begin{equation}
	\label{eq:chain:rule:intermediate:proof}
	\apmd{ ( u_{2} \circ \phi_{2} ) \circ \psi }
	=
	\apmd{ u_{2} \circ \phi_{2} } \circ D\psi
\end{equation}
fails to hold. (The existence of $\widetilde{N}_{1}$ is implied by the chain rule \Cref{prop:chainrule:QC}, the Condition ($N^{-1}$) of $\phi_{1}$, and the fact that $B_{1}$ has zero $\mathcal{H}^{2}_{G_{1}}$-measure.)

Let $\Gamma_{0}$ denote the union of $\Gamma^{+}_{ \widetilde{N}_{1} }$ and $( u_{1} \circ \phi_{1} )^{-1}( \Gamma_{1} )$. This path family has zero modulus. The paths in $AC_{+}( U_{1} ) \setminus \Gamma_{0}$ are $(u_{1} \circ \phi_{1})$-, $\Psi \circ ( u_{1} \circ \phi_{1} )$-, and $\psi$-good by the construction of $\Gamma_{1}$ and the fact that $\phi_{1}$ and $\phi_{2}$ are diffeomorphisms.

The classical chain rule for differentials and \eqref{eq:chain:rule:intermediate:proof} imply that everywhere in $U_{1} \setminus \widetilde{N}_{1}$,
\begin{align}
	\label{eq:metricderivatives:needed}
	\apmd{ \Psi \circ ( u_{1} \circ \phi_{1} ) }
	&=
	\apmd{2} \circ D\Psi \circ D( u_{1} \circ \phi_{1} )
	\text{ and }
	\\
	\label{eq:metricderivatives:needed.1}
	\apmd{ u_{1} \circ \phi_{1} }
	&=
	\apmd{1} \circ D( u_{1} \circ \phi_{1} ).
\end{align}
Then \Cref{metric:speed:almosteverycurve}, together with \eqref{eq:metricderivatives:needed} and \eqref{eq:metricderivatives:needed.1}, imply that for every $\gamma \in AC( U_{1} ) \setminus \Gamma_{0}$,
\begin{align}
	\label{eq:speed:psiimage}
	v_{ \Psi \circ ( u_{1} \circ \phi_{1} \circ \gamma ) }
	&=
	\apmd{2} \circ D\Psi \circ D( u_{1} \circ \phi_{1} ) \circ D\gamma \text{ and }
	\\
	\label{eq:speed:psidomain}
	v_{ u_{1} \circ \phi_{1} \circ \gamma }
	&=
	\apmd{1} \circ D( u_{1} \circ \phi_{1} ) \circ D\gamma
\end{align}
$m_{1}$-almost everywhere in $\mathrm{dom}( \gamma )$. By definition of the operator norm $\norm{ D\Psi }$ and of $I_{\Psi}$, \eqref{eq:speed:psiimage} and \eqref{eq:speed:psidomain} imply that
\begin{equation}
	\label{eq:uppergradient:estimate}
	v_{ \Psi \circ ( u_{1} \circ \phi_{1} \circ \gamma ) }
	\leq
	I_{\Psi} \circ ( u_{1} \circ \phi_{1} \circ \gamma )
	v_{ u_{1} \circ \phi_{1} \circ \gamma }
\end{equation}
$m_{1}$-almost everywhere in $\mathrm{dom}( \gamma )$; here we use the fact that the path $\gamma$ has zero length in $\widetilde{N}_{1}$ and that the path is $\Psi \circ ( u_{1} \circ \phi_{1} )$- and $u_{1} \circ \phi_{1}$-good. Since the path family $u_{1} \circ \phi_{1}( \Gamma_{0} )$ has zero modulus, we deduce from \eqref{eq:uppergradient:estimate} that Claim (1) holds.

Proof of Claim (2): Fix a representative of the minimal upper gradient of $\rho_{ \Psi }$. The claim follows from a slight modifications of the proof of \Cref{prop:upper:gradient}, therefore we only point out the important details. The modifications require the observations that if $\gamma$ is an arc segment in $U$ and it is not in the path family $\Gamma_{0}$, then \eqref{eq:speed:psiimage} and \eqref{eq:speed:psidomain} hold and such path segments are $\Psi \circ ( u_{1} \circ \phi_{1} )$- and $( u_{1} \circ \phi_{1} )$-good paths. Using these observations, Fubini's theorem, and the change of variables formula for $u_{1} \circ \phi_{1}$, it is possible to prove that for any $\theta \in \mathbb{R}$ for $m_{2}$-almost every $x \in U_{1} \setminus \widetilde{N}_{1}$,
\begin{align}
	\label{eq:lastestimatefornorms}
	\infty
	&>
	\rho_{ \Psi } \circ ( u_{1} \circ \phi_{1} )( x )
	J_{2}( \apmd{ u_{1} \circ \phi_{1} } )(x)
	\\ \notag
	&\geq
	\frac{ \apmd{2} \circ D\Psi \circ D( u_{1} \circ \phi_{1} )(x)[ \Exp{ i \theta } ] }{ \apmd{1} \circ D( u_{1} \circ \phi_{1} )(x)[ \Exp{ i \theta } ] }
	J_{2}( \apmd{ u_{1} \circ \phi_{1} } )(x),
\end{align}
where $\mathbb{R} \ni \theta \mapsto \Exp{ i \theta } \in \mathbb{S}^{1}$ is the exponential map. By considering a countable dense set $D \subset \mathbb{R}$, we deduce from \eqref{eq:lastestimatefornorms} that
\begin{equation*}
	\rho_{ \Psi } \circ ( u_{1} \circ \phi_{1} )(x)
	\geq
	\norm{ D\Psi } \circ ( u_{1} \circ \phi_{1} )(x)
\end{equation*}
$m_{2}$-almost every $x \in U_{1} \setminus \widetilde{N}_{1}$, i.e., $m_{2}$-almost every $x \in U_{1} \setminus \phi_{1}^{-1}( B_{1} )$. This inequality combined with the fact that $\phi_{1}$ and the restriction of $u_{1}$ to the complement of $B_{1}$ satisfy Conditions ($N$) and ($N^{-1}$) yield that
\begin{equation}
	\label{eq:Psi:gradient}
	\rho_{ \Psi }
	\geq
	\norm{ D\Psi }
	=
	I_{\Psi}
\end{equation}
$\mathcal{H}^{2}_{1}$-almost everywhere in $V_{1} \setminus u_{1}( B_{1} )$. Since $I_{\Psi}$ equals zero in $u_{1}( B_{1} )$, we deduce that the inequality \eqref{eq:Psi:gradient} holds $\mathcal{H}^{2}_{1}$-almost everywhere in $V_{1}$. The proof of Claim (2) is complete.
\end{proof}
\begin{remark}\label{computing:dilatations:charts}
The proof of \Cref{prop:characterizations} implies the following. Let $\phi_{1} \colon U_{1} \rightarrow V_{1}$ and $\phi_{2} \colon U_{2} \rightarrow V_{2}$ be isothermal parametrizations and suppose that $\Psi \colon u_{1}( V_{1} ) \rightarrow u_{2}( V_{2} )$ is quasiconformal. Define $\psi \colon U_{1} \rightarrow U_{2}$ by the identity $\left( u_{2} \circ \phi_{2} \right) \circ \psi = \Psi \circ \left( u_{1} \circ \phi_{1} \right)$ and $L$ by
\begin{equation*}
	L
	=
	D\psi
	\colon
	( TU_{1}, \, \apmd{1} \circ D\phi_{1} ) \rightarrow ( TU_{2}, \, \apmd{2} \circ D\phi_{2} ).
\end{equation*}
Then for $m_{2}$-almost every $x \in U_{1}$ and $y = u_{1} \circ \phi_{1}(x)$,
\begin{align}
	\label{eq:outer:usingcharts}
	K_{O}( \Psi )( y )
	&=
	K_{O}\left(
		L
	\right)(x) \text{ and }
	K_{I}( \Psi )(y)
	=
	K_{I}\left(
		L
	\right)(x).
\end{align}
In conclusion, the pointwise dilatations of $\Psi$ can be computed using isothermal charts.
\end{remark}
\section{Applications}\label{sec:applications}
\subsection{Isothermal parametrizations using Riemannian surfaces}\label{sec:isothermal:Riemannian}
We start this section by generalizing local isothermal parametrizations of locally reciprocal surfaces by planar domains (\Cref{def:minimal:chart}) with global isothermal parametrizations by Riemannian surfaces (\Cref{def:minimal:chart:surface}). The main result of this section is \Cref{thm:isothermal:surface}. We conclude the section with the proof of \Cref{thm:uniformization:general}.
\begin{definition}[Isothermal parametrizations]\label{def:minimal:chart:surface}
Let $Z_{d_{G}}$ be a Riemannian surface and $\Psi \colon Z_{d_{G}} \rightarrow Y_{d}$ a quasiconformal map.

We say that the pair $( Z_{d_{G}}, \, \Psi )$ is an \emph{isothermal parametrization of $Y_{d}$} if for every other Riemannian surface $\widetilde{Z}_{d_{\widetilde{G}}}$ and quasiconformal map $\widetilde{\Psi} \colon \widetilde{Z}_{d_{\widetilde{G}}} \rightarrow Y_{d}$ we have that
\begin{equation}
	\label{eq:minimal:surface}
	(K_{O}( \Psi )K_{I}( \Psi ))(x)
	\leq
	\left[ K_{O}( \widetilde{\Psi} )K_{I}( \widetilde{\Psi} ) \right] \circ ( \widetilde{\Psi}^{-1} \circ \Psi )(x)
\end{equation}
at $\mathcal{H}^{2}_{d_{G}}$-almost every $x \in Z_{d_{G}}$. If the image of the map $( Z_{d_{G}}, \, \Psi )$ is clear from the context, we say that $( Z_{d_{G}}, \, \Psi )$ is \emph{isothermal}. If also the domain is clear, we simply say that $\Psi$ is \emph{isothermal}.
\end{definition}
Recall that $u = \id_{Y} \colon Y_{d_{G}} \rightarrow Y_{d}$ is the uniformization map defined in the beginning of \Cref{sec:uniformization:main:theorem}. The following theorem is analogous to \Cref{prop:minimal:TFAE}.
\begin{theorem}\label{thm:isothermal:surface}
The uniformization map $u$ is isothermal. Isothermal parametrizations are unique in the following sense.

For a quasiconformal map $\Psi \colon Z_{d_{G}} \rightarrow Y_{d}$, the following are equivalent.
\begin{enumerate}[label=(\alph*)]
\item\label{thm:isothermal:surface.1} The map $\Psi$ is isothermal;
\item\label{thm:isothermal:surface.2} The composition $u^{-1} \circ \Psi$ is conformal in the Riemannian sense.
\end{enumerate}
Moreover, \ref{thm:isothermal:surface.1} and \ref{thm:isothermal:surface.2} are equivalent to any one of the following properties.
\begin{enumerate}[label=(\alph*)]\setcounter{enumi}{2}
\item\label{thm:isothermal:surface.3} The map $D\Psi \colon ( TZ, \, G ) \rightarrow ( TY, \, \apmd{ d } )$ is a Banach--Mazur minimizer $\mathcal{H}^{2}_{d_{G}}$-almost everywhere;
\item\label{thm:isothermal:surface.4} The identity
\begin{equation}
	\label{eq:characterization:isothermal:surfaces}
	\rho^{2}( G, \, \apmd{ d } \circ D\Psi )
	=
	K_{O}( \Psi )K_{I}( \Psi )
\end{equation}
holds $\mathcal{H}^{2}_{d_{G}}$-almost everywhere in $Z_{d_{G}}$.
\item\label{thm:isothermal:surface.5} The pointwise dilatations satisfy
\begin{equation}
	\label{eq:characterization:isothermal:surfaces:uniformizationmap}
	K_{O}( \Psi )K_{I}( \Psi ) \circ ( \Psi^{-1} \circ u )
	=
	K_{I}( u )K_{I}( u )
\end{equation}
$\mathcal{H}^{2}_{d_{G}}$-almost everywhere in $Y_{d_{G}}$.
\end{enumerate}
\end{theorem}
\begin{proof}
For the duration of the proof, we denote $Z_{d_{G}}$ by $Z$ and the Riemannian norm of $Z$ by $G_{Z}$ to avoid confusion with the Riemannian norm $G$ of $Y_{d_{G}} = Y_{G}$. The Hausdorff $2$-measure on $Z$ is denoted by $\mathcal{H}^{2}_{Z}$ and the Hausdorff $2$-measure of $Y_{G}$ by $\mathcal{H}^{2}_{G}$.

First, the equivalence between Claim \ref{thm:isothermal:surface.3} and \ref{thm:isothermal:surface.4} is established by \Cref{cor:dilatations}.

Secondly, since $\Psi \colon Z_{G_{Z}} \rightarrow Y$ is quasiconformal, \Cref{cor:dilatations} shows that
\begin{align}
	\label{eq:second:observation}
	K_{O}( \Psi )K_{I}( \Psi )
	&\geq
	\rho^{2}( G_{Z}, \, \apmd{ d } \circ D\Psi )
	\\ \notag
	=
	\rho^{2}( G_{Z} \circ D( \Psi^{-1} ) &\circ Du, \, \apmd{ d } \circ Du ) \circ ( u^{-1} \circ \Psi ),
\end{align}
where the terms in \eqref{eq:second:observation} are well-defined and the chain of inequalities holds $\mathcal{H}^{2}_{Z}$-almost everywhere in $Z$.

The composition $G_{Z} \circ D( \Psi^{-1} ) \circ Du$ is a norm induced by a Riemannian norm $\mathcal{H}^{2}_{G}$-almost everywhere in $Y$. Therefore $\mathcal{H}^{2}_{Z}$-almost everywhere, the identity
\begin{gather*}
	\rho^{2}( G_{Z} \circ D( \Psi^{-1} ) \circ Du, \, \apmd{ d } \circ Du ) \circ ( u^{-1} \circ \Psi )
	\\
	=
	\rho^{2}( G, \, \apmd{ d } \circ Du ) \circ ( u^{-1} \circ \Psi )
\end{gather*}
holds. Applying \Cref{cor:uniformization:pointwisedilatation:minimal} to the latter term shows that
\begin{gather}
	\label{eq:third:observation}
	\rho^{2}( G_{Z} \circ D( \Psi^{-1} ) \circ Du, \, \apmd{ d } \circ Du ) \circ ( u^{-1} \circ \Psi )
%	\rho^{2}( G_{Z} \circ D( \Psi^{-1} ) \circ Du, \, \apmd{ d } \circ Du )
%	\\ \notag
%	=
%	\rho^{2}( G, \, \apmd{ d } \circ Du ) \circ ( u^{-1} \circ \Psi )
	\\ \notag
	=
	K_{O}( u )K_{I}( u ) \circ ( u^{-1} \circ \Psi )
\end{gather}
$\mathcal{H}^{2}_{Z}$-almost everywhere in $Z$. Now \eqref{eq:second:observation} and \eqref{eq:third:observation} show that
\begin{equation}
	\label{eq:essentialobservation}
	K_{O}( \Psi )K_{I}( \Psi ) \circ ( \Psi^{-1} \circ u )
	\geq
	K_{O}( u )K_{I}( u )
\end{equation}
$\mathcal{H}^{2}_{G}$-almost everywhere in $Y_{G}$. We deduce from \eqref{eq:essentialobservation} that the uniformization map $u$ is isothermal.

The map $\Psi$ is isothermal if and only if the inequality in \eqref{eq:essentialobservation} is an equality $\mathcal{H}^{2}_{G}$-almost everywhere, and by \eqref{eq:second:observation} and \eqref{eq:third:observation}, this happens if and only if
\begin{equation}
	\label{eq:essentialobservation.1}
	D\Psi
	\colon
	( TZ, \, G_{Z} )
	\rightarrow
	( TY, \, \apmd{ d } )
\end{equation}
is a Banach--Mazur minimizer $\mathcal{H}^{2}_{Z}$-almost everywhere. An immediate consequence of this observation is that Claims \ref{thm:isothermal:surface.1}, \ref{thm:isothermal:surface.3}, \ref{thm:isothermal:surface.4}, and \ref{thm:isothermal:surface.5} are equivalent.

To prove the equivalence of Claims \ref{thm:isothermal:surface.3} and \ref{thm:isothermal:surface.2}, it suffices to show the equivalence in the domain of an arbitrary isothermal chart of $Z$. After this is shown, the proof is complete.

Fix isothermal parametrizations $u_{1} \circ \phi_{1} \colon U_{1} \rightarrow V_{1} \subset Z$ and $u \circ \phi_{2} \colon U_{2} \rightarrow V_{2} = \Psi( V_{1} ) \subset Y$. Define $\psi \colon U_{1} \rightarrow U_{2}$ by the identity $\left( u \circ \phi_{2} \right) \circ \psi = \Psi \circ \left( u_{1} \circ \phi_{1} \right)$ and $L$ by
\begin{equation*}
	L = D\psi
	\colon
	( TU_{1}, \, \apmd{1} \circ D\phi_{1} ) \rightarrow ( TU_{2}, \, \apmd{d} \circ D\phi_{2} ).
\end{equation*}
Then \Cref{computing:dilatations:charts} (pointwise dilatations of $\Psi$ can be computed using $L$) and \Cref{cor:dilatations} (pointwise dilatations of $\Psi$ coincide with the dilatations of $D\Psi$) yield that pointwise $m_{2}$-almost everywhere in $U_{1}$,
\begin{equation*} 
	K_{O}( D\Psi ) \circ ( u_{1} \circ \phi_{1} )
	=
	K_{O}( L )
	\text{ and }
	K_{I}( D\Psi ) \circ ( u_{1} \circ \phi_{1} )
	=
	K_{I}( L ).
\end{equation*}
Recall that $\apmd{d} \circ D\phi_{2} = \apmd{ u \circ \phi_{2} }$ by construction (\Cref{def:derivative}). The domain of $\Psi$ is the Riemannian surface $Z$ therefore \Cref{isothermal:compatible} yields that $\apmd{1} \circ D\phi_{1} = \apmd{ u_{1} \circ \phi_{1} } = G_{Z} \circ D( u_{1} \circ \phi_{1} )$ is comparable to the Euclidean norm $\norm{ \cdot }_{2}$. Therefore
\begin{equation*}
	\widetilde{L} = D\psi
	\colon
	( TU_{1}, \, \norm{ \cdot }_{2} )
	\rightarrow
	( TU_{2}, \, \apmd{ u \circ \phi_{2} } )
\end{equation*}
satisfies
\begin{equation*}
	K_{O}( L )
	=
	K_{O}( \widetilde{L} )
	\text{ and }
	K_{I}( L )
	=
	K_{I}( \widetilde{L} ).
\end{equation*}
We conclude that the differential $D\Psi$ from \eqref{eq:essentialobservation.1} is a Banach--Mazur minimizer $\mathcal{H}^{2}_{Z}$-almost everywhere in the image of $u_{1} \circ \phi_{1}$ if and only if $\widetilde{L}$ is a Banach--Mazur minimizer $m_{2}$-almost everywhere.

The differential $\widetilde{L}$ is a Banach--Mazur minimizer $m_{2}$-almost everywhere if and only if $( u \circ \phi_{2} ) \circ \psi$ is isothermal (\Cref{prop:minimal:TFAE}). This happens if and only if $\phi_{2} \circ \psi$ is conformal in the Riemannian sense (\Cref{maximal:atlas:metric}). Since $u_{1} \circ \phi_{1}$ is conformal in the Riemannian sense (\Cref{isothermal:compatible}), the map $\phi_{2} \circ \psi$ is conformal in the Riemannian sense if and only if
\begin{equation}
	u^{-1} \circ \Psi
	\colon
	Z \supset V_{1}
	\rightarrow
	u^{-1}( V_{2} ) \subset Y_{G_{2}}
\end{equation}
is conformal in the Riemannian sense.

In conclusion, the differential $D\Psi$ from \eqref{eq:essentialobservation.1} is a Banach--Mazur minimizer $\mathcal{H}^{2}_{Z}$-almost everywhere in the image $V_{1}$ of $u_{1} \circ \phi_{1}$ if and only if the restriction of $u^{-1} \circ \Psi$ to $V_{1}$ is conformal in the Riemannian sense. Since $u_{1} \circ \phi_{1}$ was an arbitrary isothermal chart of $Z$, the proof is complete.
\end{proof}
\begin{corollary}\label{cor:essentiallymaintheorem}
Let $Y_{d}$ be locally reciprocal. Then the complete constant curvature Riemannian surface $Y_{d_{G}}$ from \Cref{maximal:atlas:metric} and the uniformization map $u = \id_{Y} \colon Y_{d_{G}} \rightarrow Y_{d}$ provide an isothermal parametrization of $Y_{d}$. Moreover,
\begin{align*}
%	\label{eq:pointwise:dilatation:outer:optimal}
	\frac{ 2 }{ \pi }
	\rho^{2}( G, \, F )( x )
	&\leq
	K_{O}( u )( x )
	\leq
	\frac{ 4 }{ \pi } \text{ and }
	\\
%	\label{eq:pointwise:dilatation:inner:optimal}
	\frac{ \pi }{ 4 }
	\rho^{2}(  G, \, F )( x )
	&\leq
	K_{I}( u )( x )
	\leq
	\frac{ \pi }{ 2 }
\end{align*}
for $\mathcal{H}^{2}_{d_{G}}$-almost every $x \in Y_{d_{G}}$. In particular, the pointwise dilatations satisfy $K_{O}( u ) \leq \frac{ 4 }{ \pi }$, $K_{I}( u ) \leq \frac{ \pi }{ 2 }$ and $K_{O}( u )K_{I}( u ) \leq 2$ $\mathcal{H}^{2}_{d_{G}}$-almost everywhere.
\end{corollary}
\begin{proof}
The uniformization map is isothermal by \Cref{thm:isothermal:surface}. The dilatation bounds for $u$ follow from \Cref{cor:uniformization:pointwisedilatation:minimal} and \Cref{lem:uniformization:pointwisedilatation,pointwise:BM:properties}.
\end{proof}
Recall the statement of \Cref{thm:uniformization:general}. We are ready to prove it.
\thmuniformizationgeneral*
\begin{proof}[Proof of \Cref{thm:uniformization:general}]
The "if" direction is clear since local reciprocality is a quasiconformal invariant.

The "only if" direction follows from \Cref{cor:essentiallymaintheorem}. That result proves that the uniformization map $u \colon Y_{G} \rightarrow Y$ and the related Riemannian surface $Y_{G}$ satisfy the claimed dilatation bounds.

The first part of the claim is proved. Next we prove that the upper bounds on the outer dilatation $K_{O}( \phi )$, inner dilatation $K_{I}( \phi )$ and the distortion $\left[ K_{O}( \phi )K_{I}( \phi ) \right]^{1/2}$ are optimal. The following construction implies that every quasiconformal equivalence class of locally reciprocal surfaces has a representative where the equalities are achieved.

The following construction is motivated by the simply connected example of $( \mathbb{R}^{2}, \, \norm{ \cdot }_{\infty} )$, where $\norm{ \cdot }_{\infty}$ is the supremum norm; see Example 2.2 of \cite{uniformization}.

Fix a complete Riemannian surface $Z = Z_{d_{G}}$ of curvature $-1$, $0$, or $1$. Then for $x \in Z$ and small $r > 0$, there exists an orthonormal frame $\left\{ E_{1}, \, E_{2} \right\}$ with respect to $G$ defined in the metric ball $\overline{B}( x, \, 3 r )$. Fix a smooth function $0 \leq \phi \leq 1$ that equals zero in the complement of $B( x, \, 2r )$ and one in $\overline{B}( x, \, r )$. Consider a smooth vector field $X = X_{1} E_{1} + X_{2} E_{2}$ defined in $B( x, \, 3r )$. We define $H[X] = \sup_{ i = 1, \, 2 } \abs{ X_{i} }$ and
\begin{equation*}
	F[X]
	=
	\phi H[X]
	+
	( 1 - \phi )G[X].
\end{equation*}
The continuous norm field $F$ is initially defined in $B( x, \, 3r )$ but it extends to $Z$ due to the choice of $\phi$. By minimizing the length functional induced by $F$ along absolutely continuous paths, we obtain a length distance $d_{F}$ on $Z$ that is $\sqrt{2}$-biLipschitz equivalent to $d_{G}$.

Since $F$ is continuous, it is not difficult to see that $\apmd{ d_{F} } = F$ everywhere (see for example \cite{lipmanifolds}). Then we construct
\begin{equation*}
	T_{F} \colon ( TZ, \, F ) \rightarrow ( TZ, \, G )
\end{equation*}
as in \Cref{beltrami:diff:norm}; due to potential issues with the non-orientability of $Z$, the map $T_{F}$ is only defined locally, say in a single isothermal chart. The following argument is local so we ignore this issue.

The definition of the Banach--Mazur distance implies that fiberwise, the pullback norm $G \circ T_{F}$ is induced by an inner product that is invariant under the isometry group of the corresponding fiber of $(TZ, \, F)$. We claim that this implies that $G \circ T_{F}$ must be comparable to $G$ and thus $T_{F}$ equals $D\id_Z$. This requires an argument.

The construction of $F$ implies that the John's ellipse of the unit ball of $F$ is the unit ball of $G$ everywhere. Also the fibers of $( TZ, \, F )$ have at least as many isometries as the corresponding fibers of $( TZ, \, G )$. Therefore, using the terminology of \cite{enoughsymmetries}, the fibers of $( TZ, \, F )$ have \emph{enough symmetries}. Then Proposition 16.5 of \cite{enoughsymmetries} states that the inner product induced by $T_{F}$ is comparable to the inner product induced by the John's ellipse of the unit ball of $F$, i.e., comparable to the Riemannian metric associated to $G$. Therefore the construction of $T_{F}$ implies that $T_{F} = D\id_{Z}$.

Since $T_{F}$ equals $D\id_{Z}$, we deduce that the map $\Psi = \id_{Z} \colon Z_{d_{G}} \rightarrow Z_{d_{F}}$ is isothermal (\Cref{thm:isothermal:surface} Part \ref{thm:isothermal:surface.3}). Additionally,
\begin{equation}
	\label{BM-distance.bad}
	\rho^{2}( G, \, F \circ D\Psi )
	=
	2
\end{equation}
everywhere in $\overline{B}( x, \, r )$ (see for example Proposition 37.4 of \cite{enoughsymmetries}). Then \eqref{BM-distance.bad} and the lower bounds stated in \Cref{cor:essentiallymaintheorem} imply that in $\overline{B}( x, \, r )$,
\begin{align*}
	K_{O}( \Psi )
	=
	\frac{ 4 }{ \pi },
	\quad
	K_{I}( \Psi )
	=
	\frac{ \pi }{ 2 }
	\quad
	\text{and}
	\quad
	K_{O}( \Psi )K_{I}( \Psi )
	=
	2.
\end{align*}
This shows the sharpness of \Cref{thm:uniformization:general}.
\end{proof}
\subsection{Conformal automorphisms}\label{sec:stuff}
We study when a locally reciprocal surface is conformally equivalent to a Riemannian surface. We start the discussion with trying to understand conformal maps between reciprocal surfaces.

As an application, we show that the conformal automorphism groups of reciprocal disks characterize when a metric surface is conformally equivalent to a Riemannian surfaces (see \Cref{prop:Möbiusgroup}). We apply \Cref{prop:Möbiusgroup} to Alexandrov surfaces in order to show that the uniformization map of an Alexandrov surface is conformal thereby proving \Cref{thm:uniformization:alexandrov}. See the end of this section.

Let $Y_{1}$ and $Y_{2}$ be two locally reciprocal surfaces and $\phi_{1} \colon Z_{ G_{1} } \rightarrow Y_{ 1 }$ and $\phi_{2} \colon Z_{ G_{2} } \rightarrow Y_{ 2 }$ isothermal parametrizations by Riemannian surfaces (the existence of which is guaranteed by \Cref{thm:isothermal:surface}).
\begin{proposition}\label{conformalmaps:betweenreciprocal}
Suppose that $\Psi \colon Y_{1} \rightarrow Y_{2}$ is quasiconformal. If $\Psi$ is conformal, the map
\begin{equation}
	\widetilde{\Psi}
	=
	\phi_{2}^{-1} \circ \Psi \circ \phi_{1}
	\colon
	Z_{G_{1}} \rightarrow Z_{G_{2}}
\end{equation}
is conformal in the Riemannian sense. In particular, $\widetilde{\Psi}$ is a diffeomorphism.
\end{proposition}
Let $Y_{G}$ denote the Riemannian surface obtained from \Cref{maximal:atlas:metric}. We obtain the following result from \Cref{cor:essentiallymaintheorem} and \Cref{conformalmaps:betweenreciprocal}.
\begin{corollary}\label{conformal:automorphisms:characterization}
The conformal automorphisms of $Y$ are conformal automorphisms of $Y_{G}$.
\end{corollary}
\begin{proof}[Proof of \Cref{conformalmaps:betweenreciprocal}]
If $\Psi$ is conformal, then the pointwise dilatations satisfy
\begin{equation}
	\label{eq:conformalmaps:1}
	K_{O}( \Psi \circ \phi_{1} )K_{I}( \Psi \circ \phi_{1} )
	=
	K_{O}( \phi_{1} )K_{I}( \phi_{1} )
\end{equation}
by \Cref{lem:outer:inner:characterization}.

Let $u_{2} = \id_{Y_{2}} \colon Y_{G_{2}} \rightarrow Y_{2}$ be the uniformization map (see \Cref{sec:uniformization:main:theorem}). An application of \eqref{eq:conformalmaps:1} and \Cref{thm:isothermal:surface} shows that $\Psi \circ \phi_{1} = \phi_{2} \circ \widetilde{\Psi}$ is isothermal. The same theorem shows that this happens if and only if $u_{2}^{-1} \circ ( \phi_{2} \circ \widetilde{\Psi} )$ is conformal in the Riemannian sense. Since $u_{2}^{-1} \circ \phi_{2}$ is also conformal in the Riemannian sense by the same theorem, so is the map $\widetilde{\Psi}$.
\end{proof}
Let $u = \id_{Y} \colon Y_{G} \rightarrow Y$ denote the uniformization map. For each disk $D \subset Y_{G}$, let $G$ denote the conformal automorphism group of $D$ and $H$ the conformal automorphism group of $u(D) \subset Y$. \Cref{thm:isothermal:surface} implies that the restriction of $u$ to $D$ is an isothermal parametrization of $u(D)$. \Cref{conformalmaps:betweenreciprocal} implies that the conjugation map
\begin{equation}
	\label{eq:conjugation:uniformization}
	C_{u^{-1}}
	\colon
	H 
	\rightarrow
	G, \,
	h
	\mapsto
	u^{-1}
	\circ
	h
	\circ
	\restr{ u }{ D }
\end{equation}
is a well-defined homomorphism. It is obviously injective. Recall that $Y$ has a norm field $\apmd{} = \apmd{d}$ associated to it by \Cref{def:derivative}.
\begin{proposition}\label{prop:Möbiusgroup}
The conjugation map \eqref{eq:conjugation:uniformization} is an isomorphism if and only if the restriction of $u$ to $D$ is conformal. This happens if and only if the norm field $\apmd{}$ of $Y$ is induced by an inner product $\mathcal{H}^{2}_{G}$-almost everywhere in $D$.
\end{proposition}
\begin{proof}
The fact that the restriction of $u$ to $D$ is conformal if and only if the norm field $\apmd{}$ of $Y$ is induced by an inner product $\mathcal{H}^{2}_{G}$-almost everywhere in $D$ follows from \Cref{cor:conformality,cor:uniformization:pointwisedilatation:minimal}. If the restriction of $u$ to $D$ is conformal, the surjectivity of the conjugation map is clear. Thus it suffices to consider the case where the conjugation map is surjective.

To that end, the classical uniformization theorem implies that there exists a conformal map $\phi \colon U \rightarrow D \subset Y_{G}$, where $U$ is the Euclidean plane $\mathbb{R}^{2}$ or the Euclidean disk $\mathbb{D}$. Let $\psi = u \circ \phi$ and let $M$ denote the conformal automorphism group of $U$. Then by construction, the conjugation map
\begin{equation}
	\label{eq:conjugation:uniformization:proof}
	C_{ \psi^{-1} }
	\colon
	H
	\rightarrow
	M, \,
	h
	\mapsto
	\psi^{-1}
	\circ
	h
	\circ
	\psi
\end{equation}
is an isomorphism.

We apply \Cref{newtonian:seminorm} and \Cref{norm:almost:everywhere} to the map $\psi$. We obtain an $m_{2}$-null set $N_{0} \subset \mathbb{D}$ and countably many Borel sets $\left\{ B_{i} \right\}_{ i = 1 }^{ \infty }$ partitioning $U \setminus N_{0}$ in such a way that for each $i$, $\restr{ \apmd{\psi} }{ B_{i} }$ is a norm pointwise everywhere and continuous as a map into $\mathcal{C}( \mathbb{S}^{1}, \, \mathbb{R} )$.

Consider a Lebesgue density point $x \in B_{i}$ for some $i$ and fix $\lambda \in \mathbb{S}^{1}$. Then there exists an orientation-preserving Möbius transformation $m \in M$ that fixes $x$ and for which $Dm(x) = \lambda$. (The existence of such an $m$ is readily verified by first considering the special case $x = 0$ and then reducing the general case to $x = 0$ by considering a Möbius transformation of $M$ that maps $0$ to $x$.)

Consider the Borel set $B = B_{i} \cap m^{-1}( B_{i} )$. Since $m$ is a diffeomorphism that fixes $x$, the point $x$ is a density point of $B$ as well. By recalling \Cref{diffeo}, $\apmd{ \psi \circ m }$ exists at $x \in B$ if and only if $\apmd{ \psi }$ exists at $m(x) \in B$, and in either case,
\begin{equation}
	\label{eq:essential:step:continuity}
	\apmd{ \psi \circ m }(x) = \apmd{ \psi } \circ Dm(x).
\end{equation}
Consequently, the restriction of $\apmd{ \psi \circ m }$ to $B$ is continuous.

Since $h = \psi \circ m \circ \psi^{-1}$ is conformal, \Cref{computing:dilatations:charts} implies that the set $K \subset U$ of points where the differential
\begin{equation}
	\label{eq:differential:m}
	Dm
	\colon
	( TU, \, \apmd{ \psi } )
	\rightarrow
	( TU, \, \apmd{ \psi } )
\end{equation}
is conformal has the full measure of $U$. At every $y \in K$, we find that
\begin{equation}
	\label{eq:conformal:inB}
	\apmd{ \psi } \circ Dm(y)
	=
	C( y )\apmd{ \psi }( y )
\end{equation}
for the operator norm $C( y )$ of \eqref{eq:differential:m}. The chain rule \eqref{eq:essential:step:continuity} implies that
\begin{equation}
	\label{eq:operator:norm:chainrule}
	C( y )
	=
	\norm{ D\id \colon ( TU, \, \apmd{ \psi } ) \rightarrow ( TU, \, \apmd{ \psi \circ m } ) } ( y ).
\end{equation}
The norms $\apmd{ \psi \circ m }$ and $\apmd{ \psi }$ are continuous in $B$ by construction of $B$ thus $C$ from \eqref{eq:operator:norm:chainrule} is continuous in $B$. We deduce that \eqref{eq:conformal:inB} holds at Lebesgue density points of $B$ (which are limit points of $B \cap K$) and especially at $y = x$.

We conclude that for $y = x$ and for $m \in M$ with $m(x) = x$ and $Dm(x) = \lambda \in \mathbb{S}^{1}$, \eqref{eq:conformal:inB} shows that for every $v \in \mathbb{R}^{2}$,
\begin{equation}
	\label{eq:iteratethis}
	\apmd{ \psi }(x)[ \lambda v ]
	=
	C( x )\apmd{ \psi }(x)[ v ].
\end{equation}
Iterating \eqref{eq:iteratethis} for $v_{n} = \lambda^{n}$ yields that $C( x ) = 1$. Thus for every $v \in \mathbb{R}^{2}$,
\begin{equation}
	\label{eq:C=1}
	\apmd{ \psi }(x)[ \lambda v ]
	=
	\apmd{ \psi }(x)[ v ].
\end{equation}
Since the conjugation map \eqref{eq:conjugation:uniformization:proof} is an isomorphism, \eqref{eq:C=1} holds for every $\lambda \in \mathbb{S}^{1}$. We deduce that $\apmd{ \psi }( x ) = c(x) \norm{ \cdot }_{2}$ for some constant $0 < c(x) < \infty$. The argument goes through for an arbitrary density point of $B_{i}$ for any $i$, hence such a constant $c(x)$ exists for $m_{2}$-almost every $x \in U$. This conclusion combined with \Cref{prop:dilatationbounds} and \Cref{cor:conformality} imply that $\psi$ is conformal. This means that
\begin{equation*}
	\restr{ u }{ D }
	=
	\psi \circ \phi^{-1}
\end{equation*}
is conformal as well. The proof is complete.
\end{proof}
Next we prove \Cref{thm:uniformization:alexandrov}. Recall the statement.
\thmuniformizationalexandrov*
We elaborate on the precise statement of \Cref{thm:uniformization:alexandrov} a bit more. We recall that a two-dimensional Alexandrov space $X$ can be written as a union of a surface $Y$ and topological boundary $\partial Y = X \setminus Y$ in such a way that every point of $Y$ has a neighbourhood biLipschitz homeomorphic to a planar domain. Also recall that at $\mathcal{H}^{2}_{d}$-almost every point of $Y$, the Gromov--Hausdorff tangents are isometric to $\mathbb{R}^{2}$. See Sections 10.8 and 10.10 of \cite{lengthspace} for details. We prove the theorem for the surface part $Y = Y_{d}$.

The uniformization map $u$ is the map refered as "the map in \Cref{thm:uniformization:general}"; recall the proof of \Cref{thm:uniformization:general}. When we say that the uniformization map is absolutely continuous, we mean that it satisfies Condition ($N$).
\begin{proof}
The existence of biLipschitz charts of $Y$ imply that the minimal upper gradient of $\id_{Y_{d}}$ has $\chi_{ Y_{d} }$ as its representative, therefore the uniformization map $u$ satisfies Conditions ($N$) and ($N^{-1}$) (\Cref{lem:Condition(N):inverse}).

The chain rule \eqref{eq:chain:rule:surface} and \Cref{prop:Möbiusgroup} reduce the claim to showing that for an arbitrary isothermal parametrization $\phi \colon \mathbb{D} \rightarrow V \subset Y_{d}$ of a disk $V \subset Y_{d}$, the approximate metric differential $\apmd{ \phi }$ is induced by an inner product $m_{2}$-almost everywhere.

We use two facts. First, the Gromov--Hausdorff tangents of $Y_{d}$ are unique and isometric to $\mathbb{R}^{2}$ $\mathcal{H}^{2}_{d}$-almost everywhere in $Y_{d}$. Secondly, an application of \Cref{newtonian:seminorm} and a standard blow-up argument at density points along the lines of \cite[Proposition 3.1]{uniquetangents} imply that for $m_{2}$-almost every $x \in \mathbb{D}$, the tangents of $Y_{d}$ at $\phi(x)$ are isometric to $( \mathbb{R}^{2}, \, \apmd{ \phi }(x) )$. We deduce that $\apmd{ \phi }$ must be induced by an inner product $m_{2}$-almost everywhere. The proof is complete.
\end{proof}
\subsection{Reciprocality and quasisymmetric maps}\label{sec:recip:qs}
We consider the connections between local reciprocality and quasisymmetric maps. First we introduce some definitions that are needed later on in this section.

Let $Y$ and $Z$ be metric spaces. For a homeomorphism $\phi \colon Y\rightarrow Z$, $y \in Y$ and $r > 0$, let
\begin{align*}
	L_{\phi}( y, \, r )
	&=
	\sup\left\{
		d_{Z}( \phi(y), \, \phi(w) )
		\mid
		d_{Y}( y, \, w ) \leq r
	\right\} \text{ and }
	\\
	l_{\phi}( y, \, r )
	&=
	\inf\left\{
		d_{Z}( \phi(y), \, \phi(w) )
		\mid
		d_{Y}( y, \, w ) \geq r
	\right\}.
\end{align*}
The map $\phi$ is \emph{$\eta$-quasisymmetric} if there exists a homeomorphism $\eta \colon \left[0, \, \infty\right) \rightarrow \left[0, \, \infty\right)$ for which for every $y \in Y$ and $0 < r_{1}, \, r_{2} < \diam Y$,
\begin{equation}
	L_{\phi}( y, \, r_{1} )
	\leq
	\eta\left( \frac{ r_{1} }{ r_{2} } \right)
	l_{\phi}( y, \, r_{2} ).
\end{equation}
Such a homeomorphism $\eta$ is called the \emph{(quasisymmetric) distortion function} of $\phi$. The map $\phi$ is \emph{locally quasisymmetric} if every $y \in Y$ has a neighbourhood $V = V(y)$ such that the restriction of $\phi$ to $V$ is $\eta_{V}$-quasisymmetric for some distortion function $\eta_{V}$ with $\eta_{V}$ depending on $V$. If the neighbourhood $V$ can be chosen in such a way that the restriction of $\phi$ to $V$ is $\eta$-quasisymmetric, with $\eta$ independent of the point $y \in Y$, the map $\phi$ is \emph{locally $\eta$-quasisymmetric}.

A metric surface $Y_{d}$ is \emph{Ahlfors $2$-regular} if there exists a constant $C \geq 1$ such that for every $y \in Y$ and $\diam Y > r > 0$,
\begin{equation}
	\label{eq:ahlforsregular}
	C^{-1}
	r^{2}
	\leq
	\mathcal{H}^{2}( \overline{B}(y, \, r) )
	\leq
	C r^{2}.
\end{equation}
Such a constant is refered to as the \emph{Ahlfors regularity constant} of $Y$ and is denoted by $C_{A}$.

Let $\lambda \geq 1$. A metric surface $Y_{d}$ is \emph{$\lambda$-linearly locally contractible} if for every $y \in Y$ and $0 < r < \frac{ \diam Y }{ \lambda }$, the metric ball $B( y, \, r )$ is contractible inside the ball $B( y, \, \lambda r )$. That is, there exists $y_{0} \in B( y, \, \lambda r )$ and a continuous map $H \colon B( y, \, r ) \times \left[ 0, \, 1 \right] \rightarrow B( y, \, \lambda r )$ such that for every $z \in B( y, \, r )$, $H( z, \, 0  ) = z$ and $H( z, \, 1 ) = y_{0}$. The constant $\lambda$ is refered to as the linear local contractibility constant of $Y_{d}$.

The rest of the section is split into two parts. \Cref{sec:recip:qs:local} studies when the uniformization map of a locally reciprocal surface is locally quasisymmetric. We prove a quantitative and a qualitative characterization, when the uniformization map $u$ is locally quasisymmetric. In \Cref{sec:recip:qs:global}, we prove that given a Ahlfors $2$-regular linearly locally contractible compact metric surface, the uniformization map is quasisymmetric. We prove \Cref{thm:compact:QS:unif} in the latter part.
\subsubsection{Local quasisymmetric uniformization}\label{sec:recip:qs:local}
Suppose that $Y_{d}$ is a locally reciprocal surface. Let $u \colon Y_{d_{G}} \rightarrow Y_{d}$ be the uniformization map from \Cref{cor:essentiallymaintheorem}. Recall that $u$ is an isothermal parametrization of $Y_{d}$ and therefore $\frac{ \pi }{ 2 }$-quasiconformal.
\begin{restatable}{theorem}{thmlocalQS}\label{thm:local:QS}
Let $Y_{d}$ be a locally reciprocal metric surface. Then the following are equivalent, quantitatively.
\begin{enumerate}[label=(\alph*)]
\item\label{QSparam:loc} The uniformization map $u \colon Y_{d_{G}} \rightarrow Y_{d}$ is locally $\eta$-quasisymmetric;
\item\label{QSparam:charts} The space $Y_{d}$ has an atlas of $\eta'$-quasisymmetric charts.
\end{enumerate}
\end{restatable}
We need two auxiliary results for the proof of \Cref{thm:local:QS}.
\begin{lemma}\label{qsmaps}
Suppose that $U$ is an open subset of the plane, $V$ is locally reciprocal and $\phi \colon U \rightarrow V$ is an $\eta$-quasisymmetric homeomorphism. Then $\phi$ is $K$-quasiconformal with $K$ depending only on $\eta$.

If $U$ is the Euclidean disk $\mathbb{D}$, there exists a $\frac{ \pi }{ 2 }K$-quasiconformal map $\psi \colon \mathbb{D} \rightarrow \mathbb{D}$ such that $\psi(0) = 0$ and the composition $\phi \circ \psi$ is an isothermal $\eta'$-quasisymmetric map with $\eta'$ depending only on $\eta$.
\end{lemma}
\begin{proof}
It follows from Theorem 3.13 of \cite{qsfromeuktometric} and Proposition 5.5 of \cite{williams} that such a map $\phi$ satisfies the assumptions of \Cref{thm:QC:differentequivalent} with $K$ depending only on $\eta$, i.e., $K_{O}( \phi ) \leq K < \infty$. We can assume without loss of generality that there exists an isothermal chart $f \colon V \rightarrow W \subset \mathbb{R}^{2}$.% The Condition ($N$) of $\phi$ is proved in Corollary 5.10 of \cite{qsfromeuktometric}.

The composition $f \circ \phi$ satisfies $K_{O}( f \circ \phi ) \leq \frac{ \pi }{ 2 } K$ (since $\phi$ is $\frac{\pi}{2}$-quasiconformal by \Cref{cor:minimal:charts}), hence by \Cref{thm:QC:differentequivalent}, the composition is an element of $N^{1, \, 2}_{loc}( U, \, W )$. If we set $K' = K_{O}( f \circ \phi )$, the map $f \circ \phi$ is $K'$-quasiconformal (see for example \cite[Definition 3.1.1 and Theorem 3.7.7]{astala}). Then the map $\phi = f^{-1} \circ ( f \circ \phi )$ is quasiconformal with $K_{I}( \phi ) \leq \frac{ \pi }{ 2 }K' \leq \left( \frac{ \pi }{ 2 } \right)^{2}K$ and $K_{O}( \phi ) \leq K$.% As a consequence, $\phi$ satisfies Condition ($N^{-1}$) (recall \Cref{lem:Condition(N):inverse}). The Conditions ($N$) and ($N^{-1}$) also follow from Proposition 17.2 of \cite{uniformization}.

Now suppose that $U = \mathbb{D}$. We prove the existence of the map $\psi$. Since $\phi$ is $K( \phi )$-quasiconformal, \Cref{thm:minimal:chart:existence} shows that there exists a $\frac{ \pi }{ 2 }K( \phi )$-quasiconformal map $\psi \colon \mathbb{R}^{2} \supset W \rightarrow \mathbb{D}$ for which $\phi \circ \psi$ is isothermal. Due to the Riemann mapping theorem (and \Cref{thm:minimal:chart:uniqueness}), we can assume without loss of generality that $W = \mathbb{D}$ and that $\psi( 0 ) = 0$. Then $\psi \colon \mathbb{D} \rightarrow \mathbb{D}$ is a $K_{0} = \frac{ \pi }{ 2 }K( \phi )$-quasiconformal map that fixes the origin. Such a map is $\eta_{0}$-quasisymmetric, where $\eta_{0}$ depends only on $K_{0}$: Extend the map $\psi$ via reflection to a $K_{0}$-quasiconformal map $\widetilde{\psi}$ from $\mathbb{R}^{2}$ onto itself that fixes the origin and that maps the unit circle onto itself (the existence of $\widetilde{\psi}$ follows for example from Theorem 5.9.1 of \cite{astala}). Then $\widetilde{\psi}$ is $\eta_{0}$-quasiconformal with $\eta_{0}$ depending only on $K_{0}$; see for example Corollary 3.10.4 of \cite{astala}. Therefore $\psi$ is $\eta_{0}$-quasisymmetric.

The composition $\phi \circ \psi$ is $\eta \circ \eta_{0}$-quasisymmetric. Since $\eta_{0}$ depends only on $K_{0}$ and $K_{0}$ only on $\eta$, the distortion function $\eta \circ \eta_{0}$ depends only on $\eta$.
\end{proof}
The author got the idea for the following proposition from \cite{compactqs}.
\begin{proposition}\label{QS:restriction:is:controlled}
Let $Y_{d_{G}}$ be a complete Riemannian surface of curvature $-1$, $0$, or $1$ and
\begin{equation*}
	\phi
	\colon
	\mathbb{D}
	\rightarrow
	Y_{d_{G}}
\end{equation*}
a conformal embedding.

Suppose that $Y_{d_{G}}$ is not homeomorphic to the sphere $\mathbb{S}^{2}$ or that
\begin{equation*}
	2 \diam \phi( \mathbb{D} ) \leq \diam Y_{d_{G}}.
\end{equation*}
Then there exists a constant $2^{-1} > \beta > 0$ and a distortion function $\widetilde{\eta}$ for which
\begin{equation}
	\label{eq:control:diameter}
	\phi( \beta \mathbb{D} )
	\subset
	B_{d_{G}}\left( \phi(0), \, \frac{ l_{\phi}( 0, \, \frac{1}{2} ) }{ 6 } \right)
\end{equation}
and the restriction of $\phi$ to $\beta \mathbb{D}$ is $\widetilde{\eta}$-quasisymmetric. The constant $\beta$ and distortion function $\widetilde{\eta}$ are independent of $\phi$ and the surface $Y_{d_{G}}$.
\end{proposition}
\begin{proof}
First suppose that $Y_{d_{G}}$ is not homeomorphic to the sphere $\mathbb{S}^{2}$. The surface $Y_{G} := Y_{d_{G}}$ has a universal cover $\pi \colon \Omega \rightarrow Y_{G}$, where $\pi$ is a local isometry and where $\Omega$ is either the hyperbolic disk $\mathbb{D}_{\text{hyp}}$, the Euclidean plane $\mathbb{R}^{2}$, or the Riemann sphere $\mathbb{S}^{2}$ (see for example \Cref{maximal:atlas} and the proof of \Cref{maximal:atlas:metric}). If $\Omega = \mathbb{S}^{2}$, the covering group of $\pi$ is generated by the antipodal map.

Suppose that $\phi \colon \mathbb{D} \rightarrow Y_{G}$ is as in the claim. Then there exists a conformal embedding $\psi \colon \mathbb{D} \rightarrow \Omega$ for which $\phi = \pi \circ \psi$. Since $\phi$ is an embedding so are $\psi$ and the restriction of $\pi$ to the image of $\psi$.

Claim (1): There exists a $2^{-1} > \beta' > 0$ and a distortion function $\eta$ for which the restriction of $\psi$ to $\beta'\mathbb{D}$ is $\eta$-quasisymmetric.

Proof of Claim (1): If $\Omega$ is the hyperbolic disk or the Euclidean plane, the existence of $\beta'$ and $\eta$ follow from Propositions 5 and 7 of \cite{compactqs} (which are stated for the case when $\psi$ is orientable. However, the non-orientable case follows from the orientable one by recalling that $z \mapsto \overline{z}$ is an isometry of the Euclidean unit disk $\mathbb{D}$).

Consider the case $\Omega = \mathbb{S}^{2}$. We rotate the sphere $\mathbb{S}^{2}$ in such a way that $\psi( 0 ) = (0, \, 0, \, -1)$. Moreover, we identify $\mathbb{S}^{2}$ with the extended plane $\mathbb{R}^{2} \cup \left\{ \infty \right\}$ using the stereographic projection which fixes the \emph{equator} $\mathbb{S}^{1} = \mathbb{S}^{1} \times \left\{ 0 \right\} \subset \mathbb{R}^{3}$ and maps the \emph{south pole} $( 0, \, 0, \, -1)$ to $0$. With this identification, the stereographic projection maps the southern hemisphere to the unit disk $\mathbb{D}$. Let $\tau \colon \mathbb{S}^{2} \rightarrow \mathbb{R}^{2} \cup \left\{ \infty \right\}$ denote the stereographic projection. Recall that $\tau$ is a conformal map.

By construction, the restriction of $\pi$ to the image of $\psi$ is injective. This means that the image of $\psi$ cannot compactly contain the southern hemisphere. As a consequence, we must have that for every $\norm{ x }_{2} = 10^{-1}$, the point $\psi(x)$ is contained in the southern hemisphere (equivalently, $\norm{ \tau \circ \psi(x) }_{2} < 1$). Suppose not. Then a growth estimate for conformal embeddings \cite[Theorem 2.6]{conformaldistortion} implies that $\norm{ ( \tau \circ \psi )'(0) }_{2}$ is sufficiently large in order to deduce that $\tau \circ \psi( 2^{-1}\mathbb{D} )$ contains the closed unit disk $\overline{ \mathbb{D} }$ by the same growth estimate. This contradicts the fact that the restriction of $\pi$ to the image of $\psi$ is injective. In conclusion, $\psi( 10^{-1} \mathbb{D} )$ is contained in the southern hemisphere.

The restriction of the stereographic projection $\tau$ to the southern hemisphere is a biLipschitz map. Also the restriction of $\tau \circ \psi$ to the disk $10^{-1}\mathbb{D}$ is $\eta'$-quasisymmetric with $\eta'$ independent of $\psi$ \cite[Theorem 3.6.2]{astala}. The existence of $\beta'$ and $\eta$ follows.

Claim (2): Let $\beta' > 0$ be as in Claim (1). Then there exists a constant $\beta' > \beta'' > 0$ such that
\begin{equation}
	\label{eq:balliscontainedintheimage}
	\psi( \beta'' \mathbb{D} )
	\subset
	B_{d_{\Omega}}\left( \psi(0), \, \frac{ l_{\psi}( 0, \, \frac{1}{2} ) }{ 6 } \right).
\end{equation}
Proof of Claim (2): Suppose that $\beta' > 0$ and $\eta$ are as in Claim (1) and consider $\beta' > \beta'' > 0$. Since the restriction of $\psi$ to the disk $\beta'\mathbb{D}$ is $\eta$-quasisymmetric,
\begin{equation*}
	L_{\psi}( 0, \, \beta'' )
	\leq
	\eta\left( \frac{ \beta'' }{ \beta' } \right)
	l_{\psi}( 0, \, \beta' )
	\leq
	\eta\left( \frac{ \beta'' }{ \beta' } \right)
	l_{\psi}\left( 0, \, \frac{ 1 }{ 2 } \right).
\end{equation*}
Therefore it suffices to pick $\beta'' > 0$ so small that $\eta\left( \frac{ \beta'' }{ \beta' } \right) < \frac{ 1 }{ 6 }$. Claim (2) follows.

We complete the proof of the claim using Claims (1) and (2) (when $Y_{d_{G}}$ is not homeomorphic to $\mathbb{S}^{2}$). Recall that the restriction of $\pi$ to $\psi( \mathbb{D} )$ is injective. Let $\beta'' > 0$ be as in Claim (2). Since
\begin{equation*}
	B_{d_{\Omega}}\left( \psi(0), \, l_{\psi}\left( 0, \, \frac{1}{2} \right) \right)
	\subset
	\psi\left( 2^{-1}\mathbb{D} \right),
\end{equation*}
the restriction of $\pi$ to $B_{d_{\Omega}}\left( \psi(0), \, l_{\psi}\left( 0, \, \frac{1}{2} \right) \right)$ is an isometry onto its image. This is an immediate consequence of the fact that
\begin{equation}
	\label{eq:coveringspace:localisometry}
	d_{G}( x, \, y )
	=
	\inf\left\{
		d_{\Omega}( x', \, y' )
		\mid
		x' \in \pi^{-1}( x ) \text{ and } y' \in \pi^{-1}( y )
	\right\}.
\end{equation}
The identity \eqref{eq:coveringspace:localisometry} follows from the fact that $\pi$ is a covering map that is also a local isometry between length spaces. In conclusion, the map $\psi$ can be replaced with $\phi$ and $\Omega$ with $Y_{G}$ everywhere in Claims (1) and (2). We define $\beta = \beta''$ as in Claim (2) and $\widetilde{\eta} = \eta$ as in Claim (1) to conclude the proof of \Cref{QS:restriction:is:controlled} when $Y_{G}$ is not homeomorphic to $\mathbb{S}^{2}$.

We are left to consider the case when $Y_{G}$ is homeomorphic to $\mathbb{S}^{2}$. Then there exists an isometry $\pi \colon \mathbb{S}^{2} \rightarrow Y_{G}$. Therefore there exists a conformal embedding $\psi \colon \mathbb{D} \rightarrow \mathbb{S}^{2}$ for which $\phi = \pi \circ \psi$. By rotating the sphere, we can assume that $\psi(0)$ is the south pole. The diameter bound on the image of $\phi$ implies that $\psi( 10^{-1}\mathbb{D} )$ is contained in the southern hemisphere. The rest of the proof is argued as above.
\end{proof}
\begin{proof}[Proof of \Cref{thm:local:QS}]
Suppose that $u$ is locally $\eta_{1}$-quasisymmetric. Fix a conformal homeomorphism $\phi \colon \mathbb{D} \rightarrow V \subset Y_{d_{G}} = Y_{G}$ for which the restriction of $u$ to $V$ is $\eta_{1}$-quasisymmetric. We also assume that $2 \diam V \leq \diam Y_{G}$.

Let $2^{-1} > \beta > 0$ be as in \Cref{QS:restriction:is:controlled}. Then the restriction of $\phi$ to $\beta\mathbb{D}$ is $\widetilde{\eta}$-quasisymmetric, where $\widetilde{\eta}$ is independent of $u$, $Y$, and $Y_{G}$. Therefore the restriction of $u \circ \phi$ to $\beta\mathbb{D}$ is $\eta_{2}$-quasisymmetric, where $\eta_{2}$ depends only on $\eta_{1}$. The classical uniformization theorem implies that as $\phi$ varies such restrictions of $u \circ \phi$ provide an $\eta_{2}$-quasisymmetric atlas for $Y$.

Conversely, suppose that $Y$ has an atlas of $\eta_{2}$-quasisymmetric charts. Suppose that $z \in Y_{G}$. The proof is complete if we find a neighbourhood $W_{z}$ of $z$ and a distortion function $\eta$ depending only on $\eta_{2}$, where the restriction of $u$ to $W_{z}$ is $\eta$-quasisymmetric.

Since $Y$ has the aforementioned atlas, there exists a neighbourhood $V$ of $z$ and an $\eta_{2}$-quasisymmetric chart $f \colon u( V ) \rightarrow W \subset \mathbb{R}^{2}$.

There exists a radius $r > 0$ for which $D = \mathbb{D}( f( u(z) ), \, r ) \subset f \circ u( V )$ and $W = u^{-1} \circ f^{-1}( D )$ satisfies $2 \diam W \leq \diam Y_{G}$. By translating and rescaling $D$ and applying \Cref{qsmaps}, we deduce that there exists an $\eta_{3}$-quasisymmetric isothermal chart $F \colon W \rightarrow \mathbb{D}$ with $F( z ) = 0$, where $\eta_{3}$ depends only on $\eta_{2}$.

Let $2^{-1} > \beta > 0$ and $\widetilde{\eta}$ be as above. Let $\phi$ denote the restriction of $u^{-1} \circ F^{-1}$ to $\beta\mathbb{D}$ and $W' = \phi\left( \beta\mathbb{D} \right)$. The restriction of $u$ to $W'$ coincides with $F^{-1} \circ \phi^{-1}$. As a consequence, the restriction of $u$ to $W'$ is $\eta_{1}$-quasisymmetric, where $\eta_{1}$ depends only on $\eta_{3}$ and $\widetilde{\eta}$. The set $W'$ is the desired neighbourhood of $z$ and $\eta_{1}$ the corresponding distortion function.
\end{proof}
The proof of \Cref{thm:local:QS} implies the following qualitative version.
\begin{lemma}\label{locallyqs:lem}
The uniformization map $u$ of a locally reciprocal surface $Y_{d}$ is locally quasisymmetric if and only if $Y_{d}$ has an atlas of quasisymmetric maps.
\end{lemma}
The metric surface $Y_{d}$ is \emph{locally Ahlfors $2$-regular} with constant $C_{A}$ if for every compact set $K \subset Y$, there exists a radius $0 < r < r_{K}$ such that for all $y \in K$ and $0 < r < r_{K}$, the inequalities \eqref{eq:ahlforsregular} hold with the constant $C_{A}$.

The metric surface $Y_{d}$ is \emph{locally linearly locally contractible} with constant $\lambda$ if for every compact set $K \subset Y$, there exists a radius $0 < r < r_{K}$ such that for all $y \in K$ and $0 < r < r_{K}$, the ball $B( y, \, r )$ is contractible inside the ball $B( y, \, \lambda r )$.
\begin{remark}
If the metric surface $Y_{d}$ is locally Ahlfors $2$-regular with constant $C_{A} \geq 1$ and locally linearly locally contractible with constant $\lambda \geq 1$, the surface $Y_{d}$ has an atlas of $\eta'$-quasisymmetric charts, where $\eta'$ depends only on $C_{A}$ and $\lambda$ \cite[Theorem 4.1]{locallyqs}. Locally Ahlfors $2$-regular surfaces are locally reciprocal due to \eqref{eq:local:mass:upper:bound}, therefore \Cref{thm:local:QS} applies to the aforementioned surfaces. We can apply \Cref{locallyqs:lem} to the more general surfaces studied in \cite{locallyqs}, where the constants $C_{A}$ and $\lambda$ are allowed to vary locally uniformly.

Proposition 17.1 of \cite{uniformization} provides a locally reciprocal surface which is not $2$-rectifiable. Therefore that surface cannot have a quasisymmetric atlas (which can be argued using Proposition 17.2 of \cite{uniformization} and \Cref{newtonian:seminorm}). Therefore $u$ cannot be locally quasisymmetric.

In \cite{qssingularinverse}, Ntalampekos and Romney construct a metric surface $Y_{d} \subset \mathbb{R}^{3}$ that has a quasisymmetric parametrization $\phi \colon \mathbb{R}^{2} \rightarrow Y_{d}$ which does not satisty Condition ($N^{-1}$). Consequently, the metric surface $Y_{d}$ cannot be locally reciprocal. (If $Y_{d}$ were locally reciprocal, Proposition 17.2 of \cite{uniformization} would imply that $\phi$ satisfies Condition ($N^{-1}$).)
\end{remark}
\subsubsection{Global parametrizations of compact surfaces}\label{sec:recip:qs:global}
When we say that something in this section depends only on the \emph{data of $Y_{d}$}, we mean that the "something" depends only on the Ahlfors regularity constant $C_{A}$ and linear local contractibility constant $\lambda$.

The main motivation of this section is Theorem 1.2 of \cite{compactqs}. Geyer and Wildrick proved that if $Y_{d}$ is a compact and orientable metric surface that is Ahlfors $2$-regular and linearly locally contractible, then there exists a Riemannian metric $\widetilde{d}$ of constant curvature $-1$, $0$, or $1$ on $Y$ such that the identity map
\begin{equation}
	\label{eq:uniformization:QS}
	w
	=
	\id_{Y}
	\colon
	Y_{d}
	\rightarrow
	Y_{\widetilde{d}}
\end{equation}
is $\eta$-quasisymmetric with $\eta$ depending only on the data of $Y_{d}$. The author was interested whether it is possible to choose $\widetilde{d}$ in such a way that $w$ is isothermal; recall \Cref{maximal:atlas:metric} and \Cref{cor:essentiallymaintheorem}. We prove that this is the case and that $Y_{d}$ does not have to be orientable. Therefore the distance $\widetilde{d}$ can be chosen in such a way that we have optimal control on the pointwise dilatations of $w$ (recall \Cref{cor:uniformization:pointwisedilatation:minimal}) and good control on the quasisymmetric distortion of $w$. We prove these results in \Cref{prop:sphere:isothermalgood:uniformization,prop:compact:generalization}. \Cref{thm:compact:QS:unif} is proved as a corollary at the end of this section.
\begin{proposition}\label{prop:sphere:isothermalgood}
Suppose that $Y_{d}$ is an Ahlfors $2$-regular metric surface that is linearly locally contractible and homeomorphic to $\mathbb{S}^{2}$. Then there exists an isothermal parametrization
\begin{equation}
	\label{eq:isothermalparametrization:QS}
	\phi
	\colon
	\mathbb{S}^{2}
	\rightarrow
	Y_{d}
\end{equation}
that is $\eta$-quasisymmetric with $\eta$ depending only on the data of $Y_{d}$.
\end{proposition}
We postpone the proof of \Cref{prop:sphere:isothermalgood}.
\begin{theorem}\label{prop:sphere:isothermalgood:uniformization}
Suppose that $Y_d$ is as in \Cref{prop:sphere:isothermalgood}. Then there exists a Riemannian distance $d_{G}$ on $Y$ of constant curvature $1$ for which the uniformization map
\begin{equation*}
	u
	=
	\id_{Y}
	\colon
	Y_{d_{G}}
	\rightarrow
	Y_{d}
\end{equation*}
is isothermal and $\eta$-quasisymmetric with $\eta$ depending only on the data of $Y_{d}$.
\end{theorem}
\begin{proof}[Proof of \Cref{prop:sphere:isothermalgood:uniformization} assuming \Cref{prop:sphere:isothermalgood}]
Let $Y_{d_{G}} = Y_{G}$ denote the Riemannian surface obtained from \Cref{maximal:atlas:metric}. The surface has curvature equal to one. The uniformization map $u = \id_{Y} \colon Y_{d_{G}} \rightarrow Y_{d}$ is isothermal and therefore $\frac{ \pi }{ 2 }$-quasiconformal (\Cref{cor:essentiallymaintheorem}).

If $\phi$ is the map from \eqref{eq:isothermalparametrization:QS}, the composition $u^{-1} \circ \phi$ is conformal in the Riemannian sense (\Cref{thm:isothermal:surface}). We push forward the distance of $\mathbb{S}^{2}$ using $u^{-1} \circ \phi$. Let $d_{G'}$ denote the obtained distance and $w = \id_{Y} \colon Y_{d_{G}} \rightarrow Y_{d_{G'}}$. Then
\begin{equation*}
	\psi = w \circ \left( u^{-1} \circ \phi \right) \colon \mathbb{S}^{2} \rightarrow Y_{d_{G'}}
\end{equation*}
is an isometry. Therefore $w = \psi \circ ( u^{-1} \circ \phi )^{-1}$ is conformal in the Riemannian sense. Consider the map
\begin{equation*}
	u' = \id_{Y}
	\colon
	Y_{d_{G'}}
	\rightarrow
	Y_{d}.
\end{equation*}
Then \Cref{thm:isothermal:surface} and the fact that $u^{-1} \circ u' = w^{-1}$ prove that $u'$ is isothermal. Since $\psi$ is an isometry and $u' = \phi \circ \psi^{-1}$, the map $u'$ is $\eta$-quasisymmetric. We replace $d_{G}$ and $u$ with $d_{G'}$ and $u'$ to complete the proof.
\end{proof}
\begin{theorem}\label{prop:compact:generalization}
Suppose that $Y$ is a compact Ahlfors $2$-regular and linearly locally contractible metric surface that is not homeomorphic to $\mathbb{S}^{2}$. Then the uniformization map
\begin{equation}
	u = \id_{ Y }
	\colon
	Y_{d_{G}}
	\rightarrow
	Y_{d}
\end{equation}
is $\eta$-quasisymmetric, where $\eta$ depends only on the data of $Y_{d}$.
\end{theorem}
\begin{remark}
An interesting difference between \Cref{prop:sphere:isothermalgood:uniformization,prop:compact:generalization} is that we do not need to change the Riemannian distance $d_{G}$ on $Y$ to obtain the latter result.
\end{remark}
\begin{proof}[Proof of \Cref{prop:sphere:isothermalgood}]
The Bonk--Kleiner theorem \cite[Theorem 1.1 and Section 10]{bonk-kleinerthm} establishes that there exists an $\eta'$-quasisymmetric map $\phi' \colon \mathbb{S}^{2} \rightarrow Y_{d}$ with $\eta'$ depending only on the data of $Y_{d}$. Then by \Cref{qsmaps}, the map $\phi'$ is $K$-quasiconformal with $K$ depending only on $\eta'$.

A simple consequence of the measurable Riemann mapping theorem and \Cref{thm:minimal:chart:existence} is that there exists a $\frac{ \pi }{ 2 }K$-quasiconformal map $\psi \colon \mathbb{S}^{2} \rightarrow \mathbb{S}^{2}$ with two properties. First, the map $\psi$ fixes the north and south poles of $\mathbb{S}^{2}$ and a point from the equator. Secondly, the composition $\phi = \phi' \circ \psi$ is isothermal.

The map $\psi$ is $\eta''$-quasisymmetric with $\eta''$ depending only on $\frac{ \pi }{ 2 } K$ hence only on $\eta'$ (see for example Proposition 9.1 and Section 3 of \cite{bonk-kleinerthm}). Therefore the composition $\phi \circ \psi$ is $\eta$-quasisymmetric, where $\eta$ depends only on the data of $Y_{d}$.
\end{proof}
\begin{proof}[Proof of \Cref{prop:compact:generalization}]
For the rest of the section, we assume that $\diam Y_{d} = 1$. This can be done without loss of generality since rescaling does not change the data of $Y_{d}$, the relevant quasisymmetric distortion functions, or the isothermal charts or parametrizations of $Y_{d}$. The diameter normalization is needed for the results we use from \cite{compactqs}.

The next result is an immediate consequence of Theorem 9 of \cite{compactqs} and \Cref{qsmaps}.
\begin{theorem}\label{uniformlygoodcharts}
There is a quantity $A_{0} \geq 1$ and a distortion function $\eta$, each depending only on the data of $Y_{d}$, such that for every $0 < R \leq \frac{ 1 }{ A_{0} }$ and $y \in Y_{d}$, there is a neighbourhood $U$ of $y$ for which
\begin{enumerate}[label=(\alph*)]
\item $B( y, \, \frac{ R }{ A_{0} } ) \subset U \subset B( y, \, A_{0} R )$;
\item there exists an $\eta$-quasisymmetric homeomorphism $f \colon U \rightarrow \mathbb{D}$ that is an isothermal chart of $Y_{d}$ with $f( y ) = 0$.
\end{enumerate}
Here $\mathbb{D}$ is the Euclidean unit disk.
\end{theorem}
We combine \Cref{uniformlygoodcharts} and \Cref{QS:restriction:is:controlled} with Lemma 10 of \cite{compactqs} to conclude the following.
\begin{lemma}\label{QS:good:subatlas}
Suppose that $2^{-1} > \beta > 0$ is the universal constant from \Cref{QS:restriction:is:controlled} and $\eta$ is as in \Cref{uniformlygoodcharts}. Then there exist radii $\alpha$ and $r_{0} > 0$ and a positive integer $n$ such that the following statements hold.
\begin{enumerate}[label=(\alph*)]
\item\label{QS:good:subatlas.1} There exists an atlas $\mathcal{A}_{\beta} = \left\{ \left( U_{j}, \, f_{j} \right) \right\}_{ j = 1 }^{ n }$, where every $f_{j}$ is an $\eta$-quasisymmetric isothermal chart of $Y_{d}$ with $f_{j}( U_{j} ) = \mathbb{D}$.
\item\label{QS:good:subatlas.2} Let $x_{j} = f_{j}^{-1}( 0 )$. The collection $\left\{ B( x_{j}, \, r_{0} ) \right\}_{ j = 1 }^{ n }$ is pairwise disjoint.
\item\label{QS:good:subatlas.3} The collection $\left\{ B( x_{j}, \, 2r_{0} ) \right\}_{ j = 1 }^{ n }$ covers $Y_{d}$.
\item\label{QS:good:subatlas.4} For each $j = 1, \, \dots, \, n$, it holds $B( x_{j}, \, 10 r_{0} ) \subset U_{j}$ and
\begin{equation*}
	\alpha\mathbb{D}
	\subset
	f_{j}( B( x_{j}, \, r_{0} ) )
	\subset
	f_{j}( B( x_{j}, \, 10 r_{0} ) )
	\subset
	\beta\mathbb{D}.
\end{equation*}
\end{enumerate}
The radii $\alpha$ and $r_{0}$, and the integer $n$ depend only on the data of $Y_{d}$ and $\beta$. The function $\eta$ depends only on the data of $Y_{d}$.
\end{lemma}
The rest of the proof follows along the proof of Theorem 12 of \cite{compactqs}, therefore we only recall the main points. We prove that $v = u^{-1} = \id_{Y} \colon Y_{d} \rightarrow Y_{d_{G}}$ is $\eta_{2}$-quasisymmetric with $\eta_{2}$ depending only on the data of $Y_{d}$. This is sufficient for the claim.

Claim ($\alpha$): For each $j = 1, \, 2, \, \dots, \, n$, the restriction of $v$ to each ball $B( x_{j}, \, 10 r_{0} )$ is $\eta_{1}$-quasisymmetric with $\eta_{1}$ depending only on the data of $Y_{d}$.

Proof of Claim ($\alpha$): For each $j = 1, \, 2, \, \dots, \, n$, let $\phi_{j}$ denote the restriction of $v \circ f_{j}^{-1}$ to the set $W_{j} = f_{j}( B( x_{j}, \, 10 r_{0} ) )$. \Cref{QS:good:subatlas} Part \ref{QS:good:subatlas.4} and \Cref{QS:restriction:is:controlled} show that $\phi_{j}$ is $\widetilde{\eta}$-quasisymmetric with $\widetilde{\eta}$ independent of the data of $Y_{d}$. By assumption, the restriction of $f_{j}$ to $B( x_{j}, \, 10r_{0} )$ is $\eta$-quasisymmetric with $\eta$ depending only on the data of $Y_{d}$. In conclusion,
\begin{equation*}
	\restr{ v }{ B(x_{j}, \, 10r_{0}) }
	=
	\restr{ \phi_{j} }{ W_{j} }
	\circ
	\restr{ f_{j} }{ B( x_{j}, \, 10r_{0} ) }
\end{equation*}
is $\eta_{1} = \widetilde{\eta} \circ \eta$-quasisymmetric. This concludes the proof.

Claim ($\beta$): The Lebesgue number $L$ of the cover $\left\{ B( x_{j}, \, 10r_{0} ) \right\}_{ j = 1 }^{ n }$ is $8r_{0}$. This is clear from \ref{QS:good:subatlas.3} of \Cref{QS:good:subatlas}.

Claim ($\gamma$): For each $x, \, x' \in Y_{d}$ with $d( x, \, x' ) = 4r_{0}$, we have that
\begin{equation*}
	d_{G}( v(x), \, v(y) ) \geq \delta = C^{-1}\diam Y_{d_{G}},
\end{equation*}
where $C$ depends only on the data of $Y_{d}$.

Proof of Claim ($\gamma$): This is a delicate part of the proof of \cite[Theorem 12]{compactqs}. See \cite{compactqs} for the proof.

A theorem by Tukia and Väisälä \cite[Theorem 2.23]{qs:localtoglobal} (see also \cite[Theorem 4]{compactqs}) states that $v$ is $\eta_{2}$-quasisymmetric, where $\eta_{2}$ depends only on $\eta_{1}$ from Claim ($\alpha$), and the ratios $\frac{ \diam Y_{d} }{ L } = \frac{ 1 }{ L }$ and $\frac{ \diam Y_{d_{G}} }{ \delta }$. Then Claims ($\alpha$) to ($\gamma$) and \Cref{QS:good:subatlas} show that $\eta_{2}$ depends only on the data of $Y_{d}$.
\end{proof}
Next we prove \Cref{thm:compact:QS:unif}. Recall the statement. \thmcompactQSunif*
By recalling the proof of \Cref{thm:uniformization:general}, we find that "the map in \Cref{thm:uniformization:general}" is the uniformization map. Therefore the claim follows from \Cref{prop:sphere:isothermalgood:uniformization,prop:compact:generalization}.
\section{Open problems}\label{sec:concluding}
After proving \Cref{cor:dilatations}, we have a better understanding of quasiconformal maps between locally reciprocal surfaces. Also, \Cref{cor:conformality} and \Cref{conformalmaps:betweenreciprocal} provide insight to when two such surfaces are conformally equivalent.
\begin{openproblem}\label{end:Q1}
Let $Y_{d}$ be locally reciprocal. Is the metric surface $Y_{d}$ conformally equivalent to a metric surface $Z_{d}$ where $Z_{d}$ has desirable geometric properties?
\end{openproblem}
Some of the desirable properties are listed below.
\begin{enumerate}[label=(\alph*)]
\item\label{end:Q1.1} The space $Z_{d}$ is $\sqrt{2}$-biLipschitz equivalent to $Y_{d_{G}}$. Here $Y_{d_{G}}$ is the complete Riemannian surface of constant curvature $-1$, $0$ or $1$ obtained from \Cref{maximal:atlas:metric};
\item\label{end:Q1.2} The space $Z_{d}$ is biLipschitz equivalent to a complete Riemannian surface whose curvature is bounded from below;
\item\label{end:Q1.3} The space $Z_{d}$ has locally $2$-bounded geometry in the sense of \cite{locallyboundedgeometry};
\item\label{end:Q1.4} The minimal upper gradient of $\id_{Z}$ equals $\chi_{ Z }$ $\mathcal{H}^{2}_{Z}$-almost everywhere;
\item\label{end:Q1.5} The space $Z_{d}$ is $2$-rectifiable.
\end{enumerate}
The Properties \ref{end:Q1.1} to \ref{end:Q1.5} are listed in such a way that the preceding one is always stronger.

Property \ref{end:Q1.1} is motivated by the Riemannian version of the classical uniformization theorem. The classical result states that every smooth Riemannian metric is conformally equivalent to a Riemannian metric of constant curvature. In the setting of locally reciprocal surfaces, Property \ref{end:Q1.1} would be a close analog of this statement. We outline an argument below why this would be expected to be true.

Suppose $Y_{d}$ is a smooth surface whose distance $d$ is obtained from a continuous Finslerian norm field $\apmd{}$. This means that $\apmd{}(x)$ is a norm at every point of $x \in Y_{d}$ and that the map $x \mapsto \apmd{}(x)$ is a continuous section of the norm bundle of $Y$. The distance $d$ is obtained by minimizing the length functional induced by $\apmd{}$ over the paths that are absolutely continuous with respect to the smooth structure on $Y$. This is what we mean when we say that $d$ is obtained from the continuous norm field $\apmd{}$.

Let $G$ denote the Riemannian norm field obtained from \Cref{maximal:atlas:metric}. Then the norm field $\apmd{d}$ has $\apmd{}$ as its representative (this is not difficult to prove using the continuity of $\apmd{}$ but it also follows from \cite{lipmanifolds}). We rescale the Finslerian norm $\apmd{d}$ with the operator norm of
\begin{equation*}
	D\id \colon ( TY, \, \apmd{d} ) \rightarrow ( TY, \, G ),
\end{equation*}
which is the minimal upper gradient of the inverse of the uniformization map, to obtain a continuous Finslerian norm $H$ that satisfies
\begin{equation*}
	G \leq H \leq \rho( G, \, \apmd{d} )G,
\end{equation*}
where $\rho( G, \, \apmd{d} )$ is the Banach--Mazur distance between $G$ and $\apmd{d}$. The latter inequality uses the fact that the uniformization map is isothermal.

John's theorem implies that $\rho( G, \,\apmd{d} ) \leq \sqrt{ 2 }$. If $d_{H}$ is the geodesic distance induced by $H$, then $Y_{d_{H}}$ has Property \ref{end:Q1.1}: the metric space $Z_{d} = Y_{d_{H}}$ is $\sqrt{2}$-biLipschitz equivalent to $Y_{d_{G}}$ and the identity map from $Y_{d}$ to $Z_{d}$ is conformal. Consequently, \Cref{end:Q1} can be answered affirmatively if the distance $d$ is obtained from a continuous Finslerian distance.

More generally, the norm field $\apmd{d}$ does not have much regularity aside from $L^{2}_{loc}$-integrability (recall that the minimal upper gradients of the uniformization map $u$ and its inverse $u^{-1}$ are elements of $L^{2}_{loc}$). Nevertheless, suppose that we try to define $H$ as above. Then the distance $d_{H}$ should be the same even if we change the representative of $\apmd{d}$ (or $H$). Even if we take this into account by defining $d_{H}$ carefully --- for example, by following the techniques of Section 3 of \cite{lipmanifolds} --- it is not clear that we can guarantee that $H = \apmd{d_{H}}$ almost everywhere; see Section 5 of \cite{lipmanifolds} for a related problem on Lipschitz manifolds.
\begin{openproblem}\label{end:Q2}
Even if the norm field $\apmd{d}$ of the distance $d$ is not continuous, is the surface $Y_{d_{H}}$ always conformally equivalent to the locally reciprocal surface $Y_{d}$? If not, which geometric or analytic assumptions on the surface $Y_{d}$ guarantee that $Y_{d_{H}}$ is conformally equivalent to $Y_{d}$?
\end{openproblem}
If $\apmd{d}$ is induced by an inner product almost everywhere, we can set $Z_{d} = Y_{d_{G}}$ to guarantee Property \ref{end:Q1.1} since the uniformization map is conformal in this case (\Cref{prop:Möbiusgroup}). The property (ET) studied in \cite[Section 11]{metricderivative} is relevant to this special case.

Continuing down the list of desirable properties, clearly Property \ref{end:Q1.1} implies \ref{end:Q1.2}. Property \ref{end:Q1.2} implies \ref{end:Q1.3}, see \cite{locallyboundedgeometry}. Property \ref{end:Q1.3} implies \ref{end:Q1.4} by Lemma 8.17 of \cite{locallyboundedgeometry} and \Cref{identity:minimal:upper:gradient}. Property \ref{end:Q1.4} implies \ref{end:Q1.5} as a consequence of \Cref{lem:Condition(N):inverse} and \Cref{newtonian:seminorm}. We recall from Proposition 17.1 of \cite{uniformization} that there exists a locally reciprocal surface that is not $2$-rectifiable. That example is conformally equivalent to $\mathbb{R}^{2}$, therefore Property \ref{end:Q1.1} holds in that case.
%\clearpage
%\bibliographystyle{amsplain}
\bibliographystyle{amsalpha}
\bibliography{UniformizationIkonen}

\providecommand{\bysame}{\leavevmode\hbox to3em{\hrulefill}\thinspace}
\providecommand{\MR}{\relax\ifhmode\unskip\space\fi MR }
\providecommand{\MRhref}[2]{%
  \href{http://www.ams.org/mathscinet-getitem?mr=#1}{#2}
}
\providecommand{\href}[2]{#2}
\begin{thebibliography}{HKST15}

\bibitem[AB60]{ahlfors-bers}
Lars Ahlfors and Lipman Bers, \emph{Riemann's mapping theorem for variable
  metrics}, Ann. of Math. (2) \textbf{72} (1960), 385--404. \MR{0115006}

\bibitem[AIM09]{astala}
Kari Astala, Tadeusz Iwaniec, and Gaven Martin, \emph{Elliptic partial
  differential equations and quasiconformal mappings in the plane}, Princeton
  Mathematical Series, vol.~48, Princeton University Press, Princeton, NJ,
  2009. \MR{2472875}

\bibitem[AK00]{ambrosio2000}
Luigi Ambrosio and Bernd Kirchheim, \emph{Rectifiable sets in metric and
  {B}anach spaces}, Math. Ann. \textbf{318} (2000), no.~3, 527--555.
  \MR{1800768}

\bibitem[BBI01]{lengthspace}
Dmitri Burago, Yuri Burago, and Sergei Ivanov, \emph{A course in metric
  geometry}, 2001, pp.~xiv+415. \MR{1835418}

\bibitem[BK02]{bonk-kleinerthm}
Mario Bonk and Bruce Kleiner, \emph{Quasisymmetric parametrizations of
  two-dimensional metric spheres}, Invent. Math. \textbf{150} (2002), no.~1,
  127--183. \MR{1930885}

\bibitem[BK05]{cannon:bonk-kleiner}
\bysame, \emph{Conformal dimension and {G}romov hyperbolic groups with 2-sphere
  boundary}, Geom. Topol. \textbf{9} (2005), 219--246. \MR{2116315}

\bibitem[Bog07]{bogachev1}
V.~I. Bogachev, \emph{Measure theory. {V}ol. {I}, {II}}, Springer-Verlag,
  Berlin, 2007. \MR{2267655}

\bibitem[Che99]{PIspaces}
J.~Cheeger, \emph{Differentiability of {L}ipschitz functions on metric measure
  spaces}, Geom. Funct. Anal. \textbf{9} (1999), no.~3, 428--517. \MR{1708448}

\bibitem[DCP95]{lipmanifolds}
Giuseppe De~Cecco and Giuliana Palmieri, \emph{L{IP} manifolds: from metric to
  {F}inslerian structure}, Math. Z. \textbf{218} (1995), no.~2, 223--237.
  \MR{1318157}

\bibitem[DEKS18]{1poincareandmodulus}
Estibalitz {Durand-Cartagena}, Sylvester {Eriksson-Bique}, Riikka {Korte}, and
  Nageswari {Shanmugalingam}, \emph{{Equivalence of two BV classes of functions
  in metric spaces, and existence of a Semmes family of curves under a
  $1$-Poincar{\'e} inequality}}, arXiv e-prints (2018), arXiv:1809.03861.

\bibitem[Dud07]{chainrulepaths}
Jakub Duda, \emph{Absolutely continuous functions with values in a metric
  space}, Real Anal. Exchange \textbf{32} (2007), no.~2, 569--581. \MR{2369866}

\bibitem[Dur83]{conformaldistortion}
Peter~L. Duren, \emph{Univalent functions}, Grundlehren der Mathematischen
  Wissenschaften [Fundamental Principles of Mathematical Sciences], vol. 259,
  Springer-Verlag, New York, 1983. \MR{708494}

\bibitem[Fed69]{federer}
Herbert Federer, \emph{Geometric measure theory}, Die Grundlehren der
  mathematischen Wissenschaften, Band 153, Springer-Verlag New York Inc., New
  York, 1969. \MR{0257325}

\bibitem[FO19]{currentspoincare}
Katrin F\"{a}ssler and Tuomas Orponen, \emph{Metric currents and the
  {P}oincar\'{e} inequality}, Calc. Var. Partial Differential Equations
  \textbf{58} (2019), no.~2, Art. 69, 20. \MR{3925560}

\bibitem[GW18]{compactqs}
Lukas Geyer and Kevin Wildrick, \emph{Quantitative quasisymmetric
  uniformization of compact surfaces}, Proc. Amer. Math. Soc. \textbf{146}
  (2018), no.~1, 281--293. \MR{3723140}

\bibitem[HK98]{controlledgeometry}
Juha Heinonen and Pekka Koskela, \emph{Quasiconformal maps in metric spaces
  with controlled geometry}, Acta Math. \textbf{181} (1998), no.~1, 1--61.
  \MR{1654771}

\bibitem[HKST01]{locallyboundedgeometry}
Juha Heinonen, Pekka Koskela, Nageswari Shanmugalingam, and Jeremy~T. Tyson,
  \emph{Sobolev classes of banach space-valued functions and quasiconformal
  mappings}, J. Anal. Math. \textbf{85} (2001), 87--139. \MR{1869604}

\bibitem[HKST15]{metricsobolev}
\bysame, \emph{Sobolev spaces on metric measure spaces}, New Mathematical
  Monographs, vol.~27, Cambridge University Press, Cambridge, 2015, An approach
  based on upper gradients. \MR{3363168}

\bibitem[Hub06]{hubbard}
John~Hamal Hubbard, \emph{Teichm\"{u}ller theory and applications to geometry,
  topology, and dynamics. {V}ol. 1}, Matrix Editions, Ithaca, NY, 2006,
  Teichm\"{u}ller theory, With contributions by Adrien Douady, William Dunbar,
  Roland Roeder, Sylvain Bonnot, David Brown, Allen Hatcher, Chris Hruska and
  Sudeb Mitra, With forewords by William Thurston and Clifford Earle.
  \MR{2245223}

\bibitem[HW41]{dimension}
W.~Hurewicz and H.~Wallman, \emph{Dimension theory}, 1 ed., Princeton
  Mathematical Series, vol.~4, Princeton University Press, 1941.

\bibitem[IT92]{imayoshi}
Y.~Imayoshi and M.~Taniguchi, \emph{An introduction to {T}eichm\"{u}ller
  spaces}, Springer-Verlag, Tokyo, 1992, Translated and revised from the
  Japanese by the authors. \MR{1215481}

\bibitem[Kir94]{kirchheim}
Bernd Kirchheim, \emph{Rectifiable metric spaces: Local structure and
  regularity of the hausdorff measure}, Proceedings of the American
  Mathematical Society \textbf{121} (1994), no.~1, 113--123.

\bibitem[LD11]{uniquetangents}
Enrico Le~Donne, \emph{Metric spaces with unique tangents}, Ann. Acad. Sci.
  Fenn. Math. \textbf{36} (2011), no.~2, 683--694. \MR{2865538}

\bibitem[Lee18]{leeRiemannian}
John~M. Lee, \emph{Introduction to {R}iemannian manifolds}, Graduate Texts in
  Mathematics, vol. 176, Springer, Cham, 2018, Second edition of [ MR1468735].
  \MR{3887684}

\bibitem[Leh87]{lehto}
Olli Lehto, \emph{Univalent functions and {T}eichm\"{u}ller spaces}, Graduate
  Texts in Mathematics, vol. 109, Springer-Verlag, New York, 1987. \MR{867407}

\bibitem[LV09]{lott-villani}
John Lott and C\'{e}dric Villani, \emph{Ricci curvature for metric-measure
  spaces via optimal transport}, Ann. of Math. (2) \textbf{169} (2009), no.~3,
  903--991. \MR{2480619}

\bibitem[LW17a]{collapsedisk}
A.~{Lytchak} and S.~{Wenger}, \emph{Canonical parametrizations of metric
  discs}, arXiv e-prints (2017).

\bibitem[LW17b]{metricderivative}
Alexander Lytchak and Stefan Wenger, \emph{Area minimizing discs in metric
  spaces}, Archive for Rational Mechanics and Analysis \textbf{223} (2017),
  no.~3, 1123--1182.

\bibitem[MW13]{qs:koebe}
Sergei Merenkov and Kevin Wildrick, \emph{Quasisymmetric {K}oebe
  uniformization}, Rev. Mat. Iberoam. \textbf{29} (2013), no.~3, 859--909.
  \MR{3090140}

\bibitem[NR18]{qssingularinverse}
Dimitrios {Ntalampekos} and Matthew {Romney}, \emph{On the inverse absolute
  continuity of quasiconformal mappings on hypersurfaces}, arXiv e-prints
  (2018).

\bibitem[Raj17]{uniformization}
Kai Rajala, \emph{Uniformization of two-dimensional metric surfaces}, Invent.
  Math. \textbf{207} (2017), no.~3, 1301--1375. \MR{3608292}

\bibitem[Rom19]{upper:modulus:bound}
Matthew Romney, \emph{Quasiconformal parametrization of metric surfaces with
  small dilatation}, Indiana Univ. Math. J. \textbf{68} (2019), 1003--1011.

\bibitem[RR19]{uniformization:lower:reciprocal}
Kai Rajala and Matthew Romney, \emph{Reciprocal lower bound on modulus of curve
  families in metric spaces}, Ann. Acad. Sci. Fenn. Math. \textbf{44} (2019),
  681--692.

\bibitem[RRR19]{RRR}
Kai {Rajala}, Martti {Rasimus}, and Matthew {Romney}, \emph{Uniformization with
  infinitesimally metric measures}, arXiv e-prints (2019), arXiv:1907.07124.

\bibitem[Sha00]{newtonianspacesorigin}
Nageswari Shanmugalingam, \emph{Newtonian spaces: an extension of {S}obolev
  spaces to metric measure spaces}, Rev. Mat. Iberoamericana \textbf{16}
  (2000), no.~2, 243--279. \MR{1809341}

\bibitem[Stu06a]{sturm1}
Karl-Theodor Sturm, \emph{On the geometry of metric measure spaces. {I}}, Acta
  Math. \textbf{196} (2006), no.~1, 65--131. \MR{2237206}

\bibitem[Stu06b]{sturm2}
\bysame, \emph{On the geometry of metric measure spaces. {II}}, Acta Math.
  \textbf{196} (2006), no.~1, 133--177. \MR{2237207}

\bibitem[TJ89]{enoughsymmetries}
Nicole Tomczak-Jaegermann, \emph{Banach-{M}azur distances and
  finite-dimensional operator ideals}, Pitman Monographs and Surveys in Pure
  and Applied Mathematics, vol.~38, Longman Scientific \& Technical, Harlow;
  copublished in the United States with John Wiley \& Sons, Inc., New York,
  1989. \MR{993774}

\bibitem[TV80]{qs:localtoglobal}
P.~Tukia and J.~V\"{a}is\"{a}l\"{a}, \emph{Quasisymmetric embeddings of metric
  spaces}, Ann. Acad. Sci. Fenn. Ser. A I Math. \textbf{5} (1980), no.~1,
  97--114. \MR{595180}

\bibitem[Tys00]{qsfromeuktometric}
Jeremy~T. Tyson, \emph{Analytic properties of locally quasisymmetric mappings
  from {E}uclidean domains}, Indiana Univ. Math. J. \textbf{49} (2000), no.~3,
  995--1016. \MR{1803219}

\bibitem[Wil10]{locallyqs}
Kevin Wildrick, \emph{Quasisymmetric structures on surfaces}, Transactions of
  the American Mathematical Society \textbf{362} (2010), no.~2, 623--659.

\bibitem[Wil12]{williams}
Marshall Williams, \emph{Geometric and analytic quasiconformality in metric
  measure spaces}, Proc. Amer. Math. Soc. \textbf{140} (2012), no.~4,
  1251--1266. \MR{2869110}

\end{thebibliography}
\end{document}